\documentclass[pdflatex,sn-mathphys-num]{sn-jnl}

\usepackage{graphicx}%
\usepackage{multirow}%
\usepackage{amsmath,amssymb,amsfonts}%
\usepackage{amsthm}%
\usepackage{mathrsfs}%
\usepackage[title]{appendix}%
\usepackage{xcolor}%
\usepackage{textcomp}%
\usepackage{manyfoot}%
\usepackage{booktabs}%
\usepackage{algorithm}%
\usepackage{algorithmicx}%
\usepackage{algpseudocode}%
\usepackage{listings}%
\usepackage{mathtools}
\usepackage{microtype}
\usepackage{booktabs}
\usepackage{epstopdf}
\usepackage{subcaption}
\usepackage{array}
\usepackage{diagbox}
\usepackage{verbatim}
\usepackage{amsmath}
\usepackage{amssymb}
\usepackage{amsfonts}
\usepackage{bm}
\usepackage{amsthm}
\usepackage{graphicx}
\usepackage{float}
\newtheorem{assumption}{Assumption}
\newtheorem{lemma}{Lemma}

\theoremstyle{thmstyleone}%
\newtheorem{theorem}{Theorem}
\theoremstyle{thmstyletwo}%
\newtheorem{example}{Example}%
\newtheorem{remark}{Remark}%

\theoremstyle{thmstylethree}%

\begin{document}

\title[Deep neural network yields regularization for ill-posed inverse problems]{Deep neural network yields regularization for ill-posed inverse problems}


\author[1]{\fnm{Qiao} \sur{Zhu}}\email{qzhu@bit.edu.cn}

\author[1]{\fnm{Lan} \sur{Wang}}\email{wanglan@bit.edu.cn}

\author*[1,2]{\fnm{Ye} \sur{Zhang}}\email{ye.zhang@smbu.edu.cn}

\affil[1]{\orgdiv{School of Mathematics and Statistics}, 
\orgname{Beijing Institute of Technology}, 
\orgaddress{\city{Beijing}, \postcode{100081}, \country{China}}}

\affil[2]{\orgdiv{MSU-BIT-SMBU Joint Research Center of Applied Mathematics}, \orgname{Shenzhen MSU-BIT University}, \orgaddress{\city{Shenzhen}, \postcode{518172}, \country{China}}}


\abstract{This paper studies the regularization of ill-posed inverse problems by deep neural networks (DNNs). We extend architecture-based regularization from shallow networks to deep models by developing a deterministic framework in which the admissible network class is enlarged adaptively and the resulting architecture complexity acts as the regularization mechanism. We propose two discrepancy-principle-driven expanding DNN algorithms to treat the cases where an explicit parameter-radius bound is available and unavailable, respectively. For both algorithms, we prove the finite termination of the adaptive expansion procedure and the convergence of the regularized solutions as the noise level vanishes. In addition, we derive explicit asymptotic bounds on the terminal network architecture, thereby quantifying how the required network complexity scales with the noise level. Numerical experiments on several representative linear and non-linear inverse problems support the theoretical findings and illustrate the practical usefulness of the proposed framework.}




\maketitle

\section{Introduction}\label{sec:Introduction}

Inverse problems arise in a wide range of scientific and engineering applications, including medical imaging, non-destructive testing, and the calibration of complex physical systems \cite{engl1996regularization, isakov2006inverse, schuster2012regularization}. The primary mathematical objective is to recover an unknown target quantity $f^\dagger \in \mathcal{F} \subset \mathcal{X}_A$ from indirect observations. The underlying physical process is typically modeled by a forward operator $A:\mathcal{X}_A \to \mathcal{Y}$ between Banach spaces, leading to the operator equation
\begin{equation}\label{eq:operator_equation}
A(f)=g,
\end{equation}
where $g=A(f^\dagger)$ denotes the ideal noise-free data. In practical applications, however, the exact data $g$ is unavailable. Instead, one is given a noisy measurement $g^\delta \in \mathcal{Y}$ satisfying the deterministic noise model $\|g^\delta-g\|_{\mathcal Y}\le \delta$, where $\delta>0$ is the noise level. A defining feature of inverse problems is their ill-posedness in the sense of Hadamard: the inverse mapping $A^{-1}$ is typically discontinuous, so that small perturbations in the data may lead to large deviations in the reconstruction. Stable recovery therefore requires regularization. Classical frameworks restore well-posedness by incorporating prior knowledge via variational penalties \cite{tikhonov1977solutions, ito2014inverse}, iterative procedures \cite{gong2020new, jin2023convergence, zhang2023stochastic, jin2025regularizing}, or finite-dimensional projection methods \cite{natterer2001mathematics}.

In recent years, deep learning has led to major advances in inverse problems; see, for example, the broad overviews in \cite{mccann2017convolutional, arridge2019solving, ongie2020deep, scarlett2023theoretical}. A large part of this literature focuses on supervised and model-based deep reconstruction methods \cite{jin2017deep}, including U-Net-type architectures \cite{ronneberger2015u}, learned primal-dual schemes \cite{adler2018learned}, variational networks \cite{hammernik2018learning}, and generative-model-based approaches \cite{bora2017compressed, mardani2018deep}, all of which have demonstrated strong empirical performance in imaging applications. At the same time, the mathematical limitations of purely data-driven reconstruction pipelines have become increasingly apparent, especially with regard to instability and sensitivity to perturbations \cite{antun2020instabilities}. This has motivated another line of research that connects deep learning more directly to classical regularization theory. In this setting, neural networks define learned regularizers or data-adapted priors within variational formulations, as in the NETT framework \cite{li2020nett}, adversarial regularization \cite{lunz2018adversarial}, learned variational penalties \cite{lunz2022learned}, and total deep variation \cite{kobler2020total}. Related studies also investigate the regularizing properties of modern learning and optimization procedures for inverse problems \cite{jin2020convergence, long2024accelerated}. In parallel, physics-informed neural networks (PINNs) provide a complementary paradigm by encoding governing equations directly into the training objective \cite{raissi2019physics, karniadakis2021physics}.
Particularly relevant to the present work is the recent interest in untrained neural priors and architecture-induced regularization, which highlights the possibility that network architecture itself may contribute to the regularization of inverse problems. Starting from the Deep Image Prior (DIP) paradigm \cite{ulyanov2018deep}, a growing literature has shown that a randomly initialized network, without being
pretrained on external data, can be optimized directly against a single degraded observation and thereby act as an implicit architectural prior. Representative viewpoints include concise untrained decoders such as the Deep Decoder \cite{heckel2019deep}, regularization by architecture for inverse problems \cite{dittmer2020regularization}, deterministic recovery guarantees for unsupervised neural approaches \cite{buskulic2024convergence}, and principled stopping strategies for deep image priors \cite{wang2023early}. See also the recent survey \cite{qayyum2022untrained} for a broader overview of untrained neural priors in inverse imaging. These developments suggest that neural network architecture is not merely an implementation choice, but may itself act as a regularizing ingredient. This observation raises a natural question: rather than learning the regularization functional or relying on heuristically fixed network sizes, can one develop a deterministic regularization framework in which the admissible network architecture class is expanded adaptively and selected by a discrepancy principle?

To utilize network architecture as a regularization parameter, one must first guarantee that neural networks possess the requisite approximation capacity to represent the unknown true solution $f^\dagger$ as their structural complexity increases. This foundational assurance is provided by modern neural network approximation theory. Classical results established universal approximation for shallow networks with a single hidden layer \cite{cybenko1989approximation, hornik1989multilayer, hornik1991approximation, stinchcombe1989universal, pinkus1999approximation}, while quantitative rates were obtained for Barron-type function classes under spectral assumptions \cite{barron2002universal, bach2017breaking, li2024two}. In the deep setting, rigorous analyses have demonstrated genuine advantages of depth \cite{eldan2016power, telgarsky2016benefits, mhaskar2016deep, poggio2017and, lu2017expressive}, and quantitative approximation results for ReLU networks are now available for H\"older/continuous classes \cite{yarotsky2017error, yarotsky2018optimal, shen2019nonlinear, shen2020deep, shen2022optimal}, smooth and Sobolev-type classes \cite{yarotsky2020phase, guhring2020error, lu2021deep, hon2022simultaneous}, and piecewise smooth functions \cite{petersen2018optimal}. In addition, some recent works provide explicit parameter-constrained approximation results for deep networks, including norm-constrained ReLU networks \cite{jiao2023approximation} and constructive Kolmogorov-type deep representations \cite{schmidt2021kolmogorov}. Altogether, these results confirm that width and depth directly govern the approximation capacity available for representing the unknown solution, making them natural candidates for structural regularization parameters.

Despite this flourishing landscape of empirical success and approximation theory, an important theoretical gap remains in the context of inverse problems: \textit{how can the architectural complexity of a deep neural network be rigorously integrated into the mathematical framework of regularization theory?} In our previous work \cite{wang2025shallow}, we addressed this question in the shallow-network setting and proved that the width of a single-hidden-layer network can serve as a standalone regularization parameter selected adaptively by a discrepancy principle.

However, extending the architecture-based regularization philosophy from shallow to deep networks introduces substantial mathematical challenges. First, unlike the shallow case, the approximation capacity of deep architectures is governed jointly by width and depth, so that a stable regularization path must account for the coupled expansion of these two structural dimensions. Second, deterministic regularization theory relies on compactness to ensure existence and convergence of regularized solutions, which in turn requires optimization over hypothesis classes with explicitly bounded parameters. Although a few recent approximation results provide explicit parameter constraints, the vast majority of deep-network approximation theorems---particularly in general Sobolev-type settings---still provide error rates mainly in terms of width and depth, without yielding an explicit usable bound on the parameter radius. This disconnect prevents the direct construction of constrained network classes and poses a major theoretical obstacle to deep-network regularization.

The central objective of this paper is to bridge this gap by establishing a regularization theory for \textit{expanding deep neural networks} in an unsupervised, training-data-free setting, in the sense that the method uses only the given noisy observation $g^\delta$ and does not require external training datasets or ground-truth solution labels. Rather than relying on fixed architectures, we adaptively enlarge the admissible network class, allowing the structural complexity of the network to act as the regularization mechanism. The main contributions of this paper are summarized as follows:

\begin{itemize}
    \item We formulate a novel regularization framework for general ill-posed inverse problems in which the unknown solution is represented by a deep neural network and its architectural complexity---primarily governed by the joint scaling of width and depth---serves as the central regularizing mechanism.

    \item We develop two discrepancy-driven expanding algorithms corresponding to the cases where an explicit parameter-radius bound is available and unavailable, respectively.

    \item We prove the existence of minimizers and the finite termination of the adaptive expansion procedure. Furthermore, we establish the convergence of the regularized solutions as the noise level vanishes and derive explicit asymptotic bounds on the terminal network architecture.
\end{itemize}

In this way, the paper places the relationship between network complexity, approximation accuracy, and stability into a rigorous regularization framework, thereby extending the shallow-network theory in \cite{wang2025shallow} to deep architectures.

The remainder of this paper is organized as follows. Section~\ref{sec:Preliminaries} introduces the notation, the class of deep neural networks with bounded parameters, and the abstract approximation properties. Section~\ref{sec:framework} presents the expanding deep neural network regularization algorithms and establishes their well-posedness, termination, and convergence properties. Section~\ref{sec:simulation} reports numerical experiments on several representative inverse problems. Finally, Section~\ref{sec:conclusion} concludes the paper.

\section{Preliminaries}\label{sec:Preliminaries}
We introduce the notation, the class of deep neural networks with bounded parameters, and the abstract approximation property used in the subsequent analysis.

\subsection{Notation}

Throughout this paper, $\mathbb{N}^+$ and $\mathbb{N}$ denote the sets of positive and non-negative integers, respectively. For two quantities $X$ and $Y$, we write $X \vee Y := \max\{X,Y\}$. We use $X \lesssim Y$ (equivalently, $Y \gtrsim X$) to indicate that there exists a generic constant $C>0$, independent of the approximation parameters under consideration, such that $X \le CY$. We write $X \asymp Y$ if both $X \lesssim Y$ and $Y \lesssim X$ hold. For any $\zeta \in \mathbb{R}$, we denote its floor and ceiling by $\lfloor \zeta \rfloor := \max\{\, i \in \mathbb{Z} : i \le \zeta\,\}$ and $\lceil \zeta \rceil := \min\{\, i \in \mathbb{Z} : i \ge \zeta\,\}$. Furthermore, for a bounded domain $\Omega \subset \mathbb{R}^d$, where
$d\in\mathbb N^+$ is the spatial dimension, we employ standard notation for
function spaces. We denote by $C(\overline{\Omega})$ the space of continuous
functions on $\overline{\Omega}$. For $s\in\mathbb N^+$,
$C^s(\overline{\Omega})$ denotes the space of functions whose partial
derivatives up to order $s$ are continuous on $\overline{\Omega}$. For
$\alpha\in(0,1]$ and $\lambda>0$, we define
\[
\mathrm{H\ddot{o}lder}(\overline{\Omega},\alpha,\lambda)
:=
\left\{
f:\overline{\Omega}\to\mathbb R
\,\middle|\,
|f(x)-f(y)|\le \lambda |x-y|_\infty^\alpha
\ \text{for all } x,y\in\overline{\Omega}
\right\}.
\]
Moreover, $L^p(\Omega)$ denotes the standard Lebesgue space. For $s>0$ and $1\le q\le\infty$, $W^{s,q}(\Omega)$ denotes the Sobolev space of smoothness $s$ and integrability index $q$, equipped with its usual norm; fractional values of $s$ are understood in the standard fractional Sobolev sense. We also use $B^s_{q,\rho}(\Omega)$, with $s>0$ and $1\le q,\rho\le\infty$, to denote the Besov space with smoothness $s$, integrability index $q$, and fine index $\rho$.

\subsection{Deep neural networks with bounded parameters}

We consider fully connected feedforward neural networks with scalar outputs defined on $\overline{\Omega}$. A specific network architecture is determined by its depth $K \in \mathbb{N}^+$ (i.e., the number of hidden layers), together with layer widths $N_0, N_1, \dots, N_{K+1} \in \mathbb{N}^+$, where $N_0=d$ is the input dimension and $N_{K+1}=1$ is the output dimension. The trainable parameter set is denoted by $\Theta = \{(W_\ell, b_\ell)\}_{\ell=1}^{K+1}$, with weight matrices $W_\ell \in \mathbb{R}^{N_\ell \times N_{\ell-1}}$ and bias vectors $b_\ell \in \mathbb{R}^{N_\ell}$. We use a continuous activation function $\sigma:\mathbb R\to\mathbb R$, applied componentwise in each hidden layer. Typical examples include the ReLU activation $\sigma(t)=\max\{0,t\}$, and the sigmoid activation. The corresponding network realization $\phi(\cdot;\Theta): \overline{\Omega} \to \mathbb{R}$ is defined recursively via
\begin{equation}\label{eq:NN_recursion}
\begin{aligned}
\widetilde{\phi}_0(x) &= x,\\
\phi_\ell(x) &= W_\ell\widetilde{\phi}_{\ell-1}(x) + b_\ell, \qquad
\widetilde{\phi}_\ell(x) = \sigma(\phi_\ell(x)), \qquad \ell=1,\dots,K,\\
\phi(x;\Theta) &= W_{K+1}\widetilde{\phi}_K(x) + b_{K+1}.
\end{aligned}
\end{equation}

To obtain hypothesis classes with suitable compactness properties, we impose a
finite norm constraint on the parameter set. More precisely, for each fixed
network architecture, we regard the parameter collection $\Theta=\{(W_\ell,b_\ell)\}_{\ell=1}^{K+1}$ as the vector consisting of all scalar weights and biases, and require
\begin{equation}\label{eq:param_bound}
\|\Theta\|\le r,
\end{equation}
where $r>0$ is a prescribed finite radius. Here $\|\cdot\|$ denotes any fixed
norm on the finite-dimensional parameter space associated with the chosen
architecture, for instance the $\ell^1$, $\ell^2$, or $\ell^\infty$ norm. Accordingly, for $N,L\in\mathbb N^+$ and $r>0$, we define
\begin{equation}\label{eq:class_linear}
\mathcal N(N,L,r)
:=
\left\{
\phi(\cdot;\Theta)\in C(\overline{\Omega})
\,\middle|\,
K\le L,\ \max_{1\le \ell\le K}N_\ell\le N,\ \|\Theta\|\le r
\right\}.
\end{equation}
By construction, these classes are nested: if $N_1\le N_2$, $L_1\le L_2$, and $r_1\le r_2$, then
$$
\mathcal N(N_1,L_1,r_1)\subseteq \mathcal N(N_2,L_2,r_2).
$$

\subsection{Approximation theorems}
Our theoretical framework relies on quantitative DNN approximation results. At its core, the essential issue is whether target functions can be efficiently approximated by networks belonging to norm-constrained classes of the form $\mathcal N(N,L,r)$. The most convenient situation is when the approximation theorem provides an explicit upper bound on the admissible parameter radius. Such a bound can be used directly to define the constrained hypothesis classes underlying the known-bound algorithm (Algorithm~\ref{alg:main_algorithm}). However, this setting is rather restrictive. Many classical approximation results, especially those in stronger topologies such as Sobolev spaces, provide approximation rates only in terms of network width and depth. The corresponding parameter radius is not given explicitly. In such cases, the known-bound algorithm is no longer directly applicable. This limitation motivates the introduction of a separate unknown-bound algorithm (Algorithm~\ref{alg:two_stage}).


We now recall one representative theorem of each type. We begin with an approximation result that yields an explicit admissible parameter radius. Since this radius depends on the target function $f$, it is not directly available in inverse problems where the target function is unknown, unless additional a priori information yields a usable bound independent of $f$.

\begin{theorem}[Explicitly bounded parameters, {\cite[Theorem~3]{schmidt2021kolmogorov}}]\label{thm:explicit_weights}
Let $p\in(1,\infty)$, $d\ge2$, and $m\in\mathbb N^+$. Assume that
$f\in \mathrm{H\ddot{o}lder}([0,1]^d,\alpha,\lambda)$ for some
$\alpha\in(0,1]$ and $\lambda>0$. Then there exists a ReLU network
$
\phi\in \mathcal N\!\left(\max\{4d,\,2^{md}+1\},\,2m+3,\,2\bigl(md\vee\|f\|_\infty\bigr)2^{m(d\vee p\alpha)}\right)
$, 
such that
$$
\|f-\phi\|_{L^p([0,1]^d)}\le 2\bigl(\lambda+\|f\|_\infty\bigr)2^{-\alpha m}.
$$
\end{theorem}



We next recall a Sobolev-norm approximation result. In this case, no explicit radius bound is available, although the existence of a finite admissible radius is still guaranteed by the existence of the approximating network itself.

\begin{theorem}[Sobolev approximation, {\cite[Theorem~1.1]{hon2022simultaneous}}]\label{thm:sobolev_approx}
Suppose that $f\in C^s([0,1]^d)$ with $s\in\mathbb N$, $s>1$, satisfies  $\|\partial^{\boldsymbol{\alpha}}f\|_{L^\infty([0,1]^d)}<1$ for all $|\boldsymbol{\alpha}|\le s$. Then, for any $n,l\in\mathbb N^+$ and any $p\in(1,\infty)$, there exists a ReLU network $\phi$ with width at most $16\,s^{d+1}d(n+2)\log_2(8n)$ and depth at most $27\,s^2(l+2)\log_2(4l)$ such that
$$
\|f-\phi\|_{W^{1,p}([0,1]^d)}\le 85(s+1)^d8^s\,n^{-2(s-1)/d}l^{-2(s-1)/d}.
$$
\end{theorem}

It is important to emphasize that Theorem~\ref{thm:explicit_weights} and Theorem~\ref{thm:sobolev_approx} serve merely as concrete illustrations. They will be employed in our simulation study (Section~\ref{sec:simulation}). The framework developed herein is modular and can accommodate other DNN approximation results, provided the approximation error is measured in a reflexive Banach space,
such as $L^p$ or $W^{s,p}$ with $s>0$ and $1<p<\infty$. While other norm-constrained network
approximation results exist, for example \cite{jiao2023approximation}, they
often rely on settings that are not directly compatible with the reflexive-space
requirement of the present framework.
More commonly, approximation rates are given in terms of network architecture,
such as width and depth, without an explicit parameter-radius bound.
Compatible results of this type are available for continuous and H\"older
functions in $L^p$ norms \cite{shen2020deep,shen2022optimal}, smooth functions
measured in Sobolev norms \cite{hon2022simultaneous}, and Sobolev or Besov
classes in $L^p$ \cite{yang2025optimal}. Table~\ref{TabRates} summarizes these
representative approximation rates over different target function classes.

\begin{table}[h]
\centering
\caption{Representative approximation results for ReLU neural networks over different target function classes.}
\label{TabRates}
\footnotesize
\renewcommand{\arraystretch}{1.32}
\setlength{\tabcolsep}{1.5pt}
\begin{tabular}{
  >{\centering\arraybackslash}p{1.25cm}
  >{\centering\arraybackslash}p{2.00cm}
  >{\centering\arraybackslash}p{1.55cm}
  >{\centering\arraybackslash}p{1.55cm}
  >{\centering\arraybackslash}p{3.65cm}
  >{\centering\arraybackslash}p{2.30cm}
}
\toprule
Reference & Function class & Width & Depth & Approximation error & Norm \\
\midrule

\cite{shen2022optimal}
& $\mathrm{H\ddot{o}lder}([0,1]^d,\alpha,\lambda)$
& $\mathcal O(N)$
& $\mathcal O(L)$
& $\mathcal O\!\left((N^2L^2\log N)^{-\alpha/d}\right)$
& $L^p([0,1]^d)$, \newline $p\in(1,\infty)$ \\

\midrule

\cite{hon2022simultaneous}
& $C^s([0,1]^d)$
& $\mathcal O(N\log N)$
& $\mathcal O(L\log L)$
& $\mathcal O\!\left(N^{-2(s-s_0)/d}L^{-2(s-s_0)/d}\right)$
& $W^{s_0,p}([0,1]^d)$, \newline $p\in(1,\infty)$, $s_0<s$ \\

\midrule

\cite{yang2025optimal}
& $W^{s,q}([0,1]^d)$ or \newline $B^s_{q,\rho}([0,1]^d)$
& $\mathcal O(N)$
& $\mathcal O(L)$
& $\mathcal O\!\left((NL)^{-2s/d}\right)$
& $L^p([0,1]^d)$, \newline $p\in(1,\infty)$ \\

\bottomrule
\end{tabular}
\end{table}

The preceding constructive approximation results motivate the following abstract assumption, cf. Assumption~\ref{ass:generic_approx}, formulated directly in terms of the network architecture parameters $N$ and $L$.

\begin{assumption}[Generic DNN approximability condition]
\label{ass:generic_approx}
Let $\mathcal X_1$ be a Banach space, and let
$\mathcal F\subset \mathcal X_1$ be a target function class. Assume that there exist minimal architecture sizes $N_{\min},L_{\min}\in\mathbb N^+$ and an error profile $\{\mathcal E_{N,L}\}_{N\ge N_{\min},\,L\ge L_{\min}}
\subset [0,\infty)$, nonincreasing in each argument and satisfying
$
\lim_{N,L\to\infty}\mathcal E_{N,L}=0,
$
such that for every $f\in\mathcal F$ and every
$N\ge N_{\min}$, $L\ge L_{\min}$, there exist a finite radius
$r>0$ and a network $\phi\in \mathcal N(N,L,r)$
satisfying
\[
\|f-\phi\|_{\mathcal X_1}
\le
\mathcal E_{N,L}.
\]
\end{assumption}

\begin{remark} 
Assumption~\ref{ass:generic_approx} is an abstract approximability condition rather than a definition of a specific network architecture. It is understood in an existential sense with respect to the admissible radius: for each $f\in\mathcal F$ and each complexity level $(N,L)$, there exists an approximating network in some class $\mathcal N(N,L,r)$ with a finite radius $r$. No uniform upper bound on the admissible radius $r$ is imposed; the radius may depend on $f$, $N$, and $L$. 
\end{remark}

\begin{remark}[Choice of the error profile $\mathcal E_{N,L}$]
The approximation theorems above determine admissible choices of
$\mathcal E_{N,L}$ by selecting the largest constructive parameters compatible
with the prescribed width and depth budgets $(N,L)$.

In Theorem~\ref{thm:explicit_weights}, the construction is indexed by $m$ and
requires $\max\{4d,2^{md}+1\}\le N,2m+3\le L$. Choosing the largest such $m=m(N,L)$ gives
\[
\mathcal E_{N,L}
=
2(\lambda+\|f\|_\infty)2^{-\alpha m(N,L)}
\asymp
\max\{N^{-\alpha/d},2^{-\alpha L/2}\},
\]
up to constants and integer rounding.

In Theorem~\ref{thm:sobolev_approx}, the parameters $n$ and $l$ control the
width and depth, respectively. Choosing the largest admissible
$n=n(N)$ and $l=l(L)$ yields
\[
\mathcal E_{N,L}
=
85(s+1)^d8^s\,
n(N)^{-2(s-1)/d}l(L)^{-2(s-1)/d}.
\]
Moreover, by the width and depth constraints, there exists a constant
$C_{s,d}>0$, depending only on $s$ and $d$, such that
\[
\mathcal E_{N,L}
\le
C_{s,d}
\left({\log N}/{N}\right)^{2(s-1)/d}
\left({\log L}/{L}\right)^{2(s-1)/d}.
\]
The resulting profile may be stepwise because the constructive parameters are
integer-valued, but this is enough for the monotonicity and convergence
requirements in Assumption~\ref{ass:generic_approx}.
\end{remark}

\section{A general framework of expanding neural network methods}\label{sec:framework}
In this section, we present a general framework for expanding neural-network regularization methods. We first state the assumptions, then introduce two algorithms corresponding to the cases with and without an a priori parameter-radius bound, and finally establish their convergence.

\subsection{Assumptions}
Throughout this subsection, the network class $\mathcal N(N,L,r)$ is understood
as the hypothesis class defined in \eqref{eq:class_linear}. We first record a
basic compactness property of this class, which will be used to verify the
abstract compactness assumption below in concrete settings. The proof of the following lemma is deferred to Appendix~\ref{appendix:lem_NN_compact}.

\begin{lemma}[Compactness of the hypothesis class]\label{lem:NN_compact_Lq}
Let $N,L\in\mathbb N^+$ be fixed, and let $0\le r<\infty$. Assume that the activation function $\sigma:\mathbb R\to\mathbb R$ is continuous. Then the network class $\mathcal{N}(N,L,r)$ is compact in $C(\overline{\Omega})$ equipped with the uniform norm $\|f\|_{C(\overline{\Omega})}:=\max_{x\in \overline{\Omega}}|f(x)|$. Consequently, for every $p\in[1,\infty)$, $\mathcal{N}(N,L,r)$ is also compact in $L^p(\Omega)$.
\end{lemma}

The approximation framework is formulated in the Banach space $\mathcal X_1$.
To establish the existence of minimizers for the network-constrained
minimization problems considered below, we require a compactness mechanism for
minimizing sequences that also preserves the admissibility of the limit. Although
Lemma~\ref{lem:NN_compact_Lq} provides compactness in $C(\overline{\Omega})$
and $L^p(\Omega)$, the ambient space $\mathcal X_1$ used in the regularization
analysis may be different. Therefore, we introduce an auxiliary Banach space
$\mathcal X_0$, endowed with a weaker topology in the sense that
$\mathcal X_1$ is continuously embedded into $\mathcal X_0$, in which the
network class $\mathcal N(N,L,r)$ is strongly compact. This two-space approach makes it possible to extract strongly convergent subsequences in $\mathcal X_0$ while retaining weak compactness in $\mathcal X_1$. In practice, the compactness of $\mathcal N(N,L,r)$ is often significantly easier to verify in a weaker topology than directly in $\mathcal X_1$, especially for nonsmooth activations (e.g., ReLU) and stronger norms. We therefore impose the following assumption.

\begin{assumption}[Topological setting]\label{ass:X1_pivot_general}
Let $\mathcal X_1$ be a reflexive Banach space of functions on $\Omega$, and let $\mathcal X_0$ be an auxiliary Banach space of functions on $\Omega$ such that:
\begin{enumerate}
\item[(i)] $\mathcal N(N,L,r)\subset \mathcal X_1$ for all $N,L\in\mathbb N^+$ and $r>0$;
\item[(ii)] the embedding $\mathcal X_1\hookrightarrow \mathcal X_0$ is continuous;
\item[(iii)] for every fixed $(N,L,r)$, the class $\mathcal N(N,L,r)$ is compact in $\mathcal X_0$.
\end{enumerate}
\end{assumption}

\begin{remark}
(a) If $\mathcal N(N,L,r)$ is already compact in the chosen space $\mathcal X_1$,
then no auxiliary space is needed. In this case, one may take
$\mathcal X_0=\mathcal X_1$, so that Assumption~\ref{ass:X1_pivot_general}(ii)
is trivial and Assumption~\ref{ass:X1_pivot_general}(iii) reduces to compactness
in $\mathcal X_1$ itself.

(b) The inclusion $\mathcal N(N,L,r)\subset\mathcal X_1$ in
Assumption~\ref{ass:X1_pivot_general}(i) is mild in many standard settings.
Typical reflexive choices for $\mathcal X_1$ include $L^p(\Omega)$ and
$W^{s,p}(\Omega)$ with $1<p<\infty$ and $s\in\mathbb N^+$. For instance, if
the activation function is locally Lipschitz, such as ReLU, then, for fixed
$(N,L,r)$, all network realizations are Lipschitz continuous on bounded domains.
Hence, by Rademacher's theorem, they belong to $W^{1,\infty}(\Omega)$. For
higher-order Sobolev spaces $W^{s,p}(\Omega)$ with $s\ge2$, stronger
smoothness assumptions on the activation function are generally needed.


(c) Assumption~\ref{ass:X1_pivot_general}(iii) can be verified in many common
situations by Lemma~\ref{lem:NN_compact_Lq}. In particular, one may take
$\mathcal X_0=L^q(\Omega)$ whenever the embedding $\mathcal X_1\hookrightarrow L^q(\Omega)$ is continuous. If a continuous embedding $\mathcal X_1\hookrightarrow C(\overline{\Omega})$ is available, then $\mathcal X_0=C(\overline{\Omega})$ is also admissible. The appropriate choice depends on the ambient space $\mathcal X_1$ and the specific application under consideration.
\end{remark}

With this topological setting, we now introduce the assumptions on the forward operator $A$ and the regularizer $\mathcal R$ that are required for the existence analysis of the network-constrained minimization problems.

\begin{assumption}[Forward operator]\label{ass:forward_full}
Let $\mathcal X_A$ be a Banach space and let $A:\mathcal X_A\to\mathcal Y$ be the forward operator. Assume that:
\begin{enumerate}
\item[(i)] \textbf{(Well-definedness)} The embedding $\mathcal X_1\hookrightarrow \mathcal X_A$ is continuous, so that $A$ is well defined on the network class $\mathcal N(N,L,r)\subset \mathcal X_1$.
\item[(ii)] \textbf{(Weak continuity)} $A$ is sequentially weak-to-weak continuous on bounded subsets of $\mathcal X_1$, i.e., $f_m\rightharpoonup f$ in $\mathcal X_1$ with $\sup_m\|f_m\|_{\mathcal X_1}<\infty$ implies $A(f_m)\rightharpoonup A(f)$ in $\mathcal Y$ as $m\to\infty$.
\item[(iii)] \textbf{(Injectivity)} The equation $A(f)=g$ admits a unique solution $f^\dagger$ in $\mathcal X_1$.
\item[(iv)] \textbf{(Local H\"older continuity)} There exist constants $\theta\in(0,1]$, $L_A>0$, and $\eta>0$ such that $\|A(f)-A(f^\dagger)\|_{\mathcal Y} \le L_A\,\|f-f^\dagger\|_{\mathcal X_1}^{\theta}$ for all $f\in\mathcal X_1$ satisfying $\|f-f^\dagger\|_{\mathcal X_1}\le \eta$.
\end{enumerate}
\end{assumption}

\begin{assumption}[Regularizer]\label{ass:R_full}
Let $\mathcal R:\mathcal X_1\to[0,\infty]$ be a proper functional. Assume that:
\begin{enumerate}
\item[(i)] \textbf{(Weak lower semicontinuity)} $\mathcal R$ is sequentially weakly lower semicontinuous on $\mathcal X_1$.
\item[(ii)] \textbf{(Coercivity)} $\mathcal R(f)\to\infty$ as $\|f\|_{\mathcal X_1}\to\infty$.
\item[(iii)] \textbf{(Finiteness on the network class)} For every $(N,L,r)$, there exists at least one network $\phi\in\mathcal N(N,L,r)$ such that $\mathcal R(\phi)<\infty$.
\item[(iv)] \textbf{(Uniform regularizer bound on approximants)} For the exact solution $f^\dagger\in\mathcal F$, there exists a constant $C_{f^\dagger}>0$, independent of $(N,L)$, such that the approximants in Assumption~\ref{ass:generic_approx} can be chosen to satisfy
\[
\mathcal R(\phi_{N,L})\le C_{f^\dagger},
\qquad \forall\, N\ge N_{\min},L\ge L_{\min}.
\]
\end{enumerate}
\end{assumption}

\begin{remark}[A canonical choice of $\mathcal R$]\label{rem:R_choice}
A natural and standard choice is $\mathcal R(f):=\|f\|_{\mathcal X_1}$. Then Assumption~\ref{ass:R_full}(iv) is satisfied. Indeed, by Assumption~\ref{ass:generic_approx}, for $f^\dagger\in\mathcal F$ and $N\ge N_{\min},L\ge L_{\min}$, there exists an approximant $\phi\in\mathcal N(N,L,r)$ for some $r>0$ such that
$\|f^\dagger-\phi\|_{\mathcal X_1}\le \mathcal E_{N,L}\le \mathcal{E}_{N_{\min},L_{\min}}.$
Hence, by the triangle inequality,
\[
\mathcal R(\phi_{N,L})
=\|\phi\|_{\mathcal X_1}
\le \|f^\dagger\|_{\mathcal X_1}+\|f^\dagger-\phi\|_{\mathcal X_1}
\le \|f^\dagger\|_{\mathcal X_1}+\mathcal E_{N_{\min},L_{\min}},
\]
where the last inequality follows from the monotonicity of $\mathcal E_{N,L}$. Therefore, one may take $C_{f^\dagger}:=\|f^\dagger\|_{\mathcal X_1}+\mathcal E_{N_{\min},L_{\min}}$,
which is independent of $(N,L)$.
\end{remark}
\subsection{Expanding neural network regularization algorithms}

We now formulate the expanding deep neural network (DNN) regularization algorithms for the ill-posed operator equation \eqref{eq:operator_equation}. The central idea is to use the network architecture itself as the regularization mechanism, so that the structural complexity of the network plays the role of a regularization parameter. Instead of fixing the architecture \emph{a priori}, we consider a nested sequence of neural network classes with gradually increasing width and depth, and at each stage solve a regularized minimization problem over the current admissible class. A discrepancy-type stopping rule selects the final architecture in a noise-dependent manner, thereby determining a stable regularized reconstruction.

For a given architecture pair $(N,L)$ and an admissible parameter radius $r$, we define the regularized objective functional
\begin{equation}\label{eq:def_J}
J_{N,L,r}^{\delta}(\phi):=\|A(\phi)-g^\delta\|_{\mathcal Y}+\beta_{N,L}\mathcal R(\phi).
\end{equation}
The first term measures consistency with the noisy data, while the regularization term $\mathcal R(\phi)$ stabilizes the reconstruction. The balance is controlled by the weight $\beta_{N,L}$, which should be tied to the approximation level of the corresponding network class. Based on the error profile $\mathcal E_{N,L}$ from
Assumption~\ref{ass:generic_approx}, we first define the extended
algorithmic approximation scale
\begin{equation}
\label{defoverline_E}
    \begin{aligned}
        \overline{\mathcal E}_{N,L}
:=
\begin{cases}
\mathcal E_{N_{\min},L_{\min}},
& \text{if } N<N_{\min}\ \text{or}\ L<L_{\min},\\[1mm]
\mathcal E_{N,L},
& \text{if } N\ge N_{\min}\ \text{and}\ L\ge L_{\min}.
\end{cases}
    \end{aligned}
\end{equation}
We then define
\begin{equation}\label{eq:beta_effective}
\beta_{N,L}:=
c_0\,\overline{\mathcal E}_{N,L}^{\theta},
\end{equation}
where $c_0>0$ is fixed and $\theta\in(0,1]$ is the exponent from
Assumption~\ref{ass:forward_full}(iv). This truncated definition ensures
that $\beta_{N,L}$ is well defined even when the algorithm visits
architectures below the minimal approximation threshold. For such smaller
architectures, $\overline{\mathcal E}_{N,L}$ is used only as a
regularization scale, not as an approximation guarantee for
$\mathcal N(N,L,\cdot)$.

To describe the architectural expansion, we prescribe two nondecreasing and unbounded sequences $\{N_k\}_{k\ge1}, \{L_k\}_{k\ge1}\subset \mathbb N^+$ to dictate the network width and depth at iteration $k$. After solving the corresponding regularized minimization problem, the algorithm evaluates a stopping criterion to decide whether a further expansion is needed.

To guarantee finite termination for this expansion mechanism, we introduce a noise-dependent target approximation level $\mathcal E_\delta$ and a corresponding target architecture $(N_{\rm tar}(\delta),L_{\rm tar}(\delta))$. The guiding principle is to ensure that for a sufficiently large radius $r$ and $N\ge N_{\min},L\ge L_{\min}$, there exists an approximant $\phi\in \mathcal{N}(N,L,r)$ such that the objective $J_{N,L,r}^\delta(\phi)$ falls below the threshold $\tau\delta$. For approximants $\phi$ satisfying $\|\phi-f^\dagger\|_{\mathcal X_1}\le \eta$, the local H\"older continuity of $A$ (Assumption~\ref{ass:forward_full}(iv)), the approximation property (Assumption~\ref{ass:generic_approx}), the uniform regularizer bound on admissible approximants (Assumption~\ref{ass:R_full}(iv)), and the triangle inequality yield:
\begin{equation}\label{eq:J_decomposition}
\begin{aligned}
J_{N,L,r}^\delta(\phi) 
&\le \|A(f^\dagger) - g^\delta\|_{\mathcal{Y}} + \|A(\phi) - A(f^\dagger)\|_{\mathcal{Y}} + \beta_{N,L} \mathcal{R}(\phi) \\
&\le \delta + L_A \|\phi - f^\dagger\|_{\mathcal{X}_1}^\theta + c_0 C_{f^\dagger} \mathcal{E}_{N,L}^\theta.
\end{aligned}
\end{equation}
By requiring each of the two architecture-dependent bias terms in \eqref{eq:J_decomposition} to not exceed $\frac{\tau-1}{2}\delta$, we naturally derive the target error level:
\begin{equation}\label{eq:E_delta}
\mathcal{E}_\delta := \min \left\{ 
\left( \frac{\tau-1}{2L_A} \delta \right)^{1/\theta}, 
\left( \frac{\tau-1}{2c_0 C_{f^\dagger}} \delta \right)^{1/\theta}, 
\eta 
\right\}.
\end{equation}
We then select target budgets $(N_{\rm tar}(\delta),L_{\rm tar}(\delta))$ with
$N_{\rm tar}(\delta)\ge N_{\min}$ and $L_{\rm tar}(\delta)\ge L_{\min}$ such that
\begin{equation}
    \label{def_NLtardelta}
    q_0\mathcal E_\delta
\le
\mathcal E_{N_{\rm tar}(\delta),L_{\rm tar}(\delta)}
\le
\mathcal E_\delta
\end{equation}
for some $q_0\in(0,1]$. Since $\mathcal E_\delta\le\eta$, every
architecture dominating $(N_{\rm tar}(\delta),L_{\rm tar}(\delta))$ lies in the admissible regime required by the stopping rule and the local stability estimate.

The target architecture $(N_{\rm tar}(\delta),L_{\rm tar}(\delta))$ introduced above plays distinct roles in the two cases. In Case~I, it serves purely as an analytical tool for the termination proof and does not enter the algorithm explicitly. More precisely, in this setting the target architecture is understood through the constructive parameterization of the underlying approximation theorem with explicit radius bounds. In Case~II, by contrast, $(N_{\rm tar},L_{\rm tar})$ is used both in the theoretical analysis and as the explicit reference scale for the transition from Phase~I to Phase~II.

\medskip
\noindent\textbf{Case I (Explicit radius bound available).}
We first consider the situation where the underlying approximation theory provides an explicit, nondecreasing upper bound for the admissible parameter radius. Although this bound generally depends on several fixed problem-specific constants, we write it as $r_{\max}(N,L)$ in order to emphasize its functional dependence on the network architecture.

At the $k$-th expansion stage, with the architecture pair $(N_k, L_k)$, we define the stage-specific radius $r_k := r_{\max}(N_k, L_k)$. The admissible class is restricted to $\mathcal N(N_k, L_k, r_k)$. The resulting approximation $f_k^\delta$ is defined as a minimizer of the regularized objective \eqref{eq:def_J} over this constrained class, namely
\begin{equation}\label{eq:f_k_case1}
f_k^{\delta} \in \operatorname*{argmin}_{\phi \in \mathcal{N}(N_k, L_k,\, r_k)} J_{N_k, L_k,\, r_k}^{\delta}(\phi).
\end{equation}

Since the approximation profile is meaningful only above the minimal
architecture threshold and the local stability estimate requires the
approximating error to lie within the $\eta$-neighborhood of
$f^\dagger$, we allow the algorithm to stop only after these admissibility
conditions are met. We evaluate
\[
S_k^{1,\delta}:=\|A(f_k^{\delta})-g^\delta\|_{\mathcal Y},
\]
and use the stopping criterion
\begin{equation}\label{eq:stopping_case1}
N_k\ge N_{\min},\qquad
L_k\ge L_{\min},\qquad
\mathcal E_{N_k,L_k}\le \eta,
\qquad
S_k^{1,\delta}\le \tau\delta .
\end{equation}
Here the Morozov discrepancy test is still imposed purely on the data residual; the first three conditions are admissibility requirements for the theoretical error analysis.

We denote by $k(\delta)$ the first iteration at which \eqref{eq:stopping_case1} is satisfied. The resulting reconstruction is then given by $f^\delta := f_{k(\delta)}^\delta$, with the selected architecture given by $(N_{k(\delta)},L_{k(\delta)})$.

The corresponding procedure is summarized in Algorithm~\ref{alg:main_algorithm}.

\begin{algorithm}[h]
\caption{Expanding DNN regularization under an explicit radius bound}
\label{alg:main_algorithm}
\begin{algorithmic}[1]
\State \textbf{Input:}
Expansion sequences $\{N_k\}, \{L_k\}$; explicit radius bound function $r_{\max}(N,L)$; discrepancy parameter $\tau>1$; exponent $\theta\in(0,1]$; local radius $\eta>0$; constant $c_0>0$; noise level $\delta$; noisy data $g^\delta$; regularizer $\mathcal R$; error profile $\overline{\mathcal E}_{N,L}$.
\State \textbf{Output:}
Reconstructed solution $f^\delta$, final architecture dimensions $N_{k(\delta)}$, $L_{k(\delta)}$, and stopping index $k(\delta)$.

\State $k\gets 0$
\Repeat
    \State $k \gets k+1$
    \State $r_k \gets r_{\max}(N_k, L_k)$
    \State $\beta_{N_k,L_k}\gets
    c_0\,\overline{\mathcal E}_{N_k,L_k}^{\theta}$
    \State Compute a minimizer $f_k^{\delta}$ via \eqref{eq:f_k_case1}
    \State Evaluate the residual:
    $S_k^{1,\delta} \gets \|A(f_k^{\delta})-g^\delta\|_{\mathcal Y}$
\Until{$N_k\ge N_{\min}$, $L_k\ge L_{\min}$,
$\mathcal E_{N_k,L_k}\le \eta$, and
$S_k^{1,\delta}\le \tau\delta$}

\State $k(\delta)\gets k$
\State $f^\delta \gets f_{k(\delta)}^{\delta}$
\State \Return $f^\delta,\, N_{k(\delta)},\, L_{k(\delta)},\, k(\delta)$
\end{algorithmic}
\end{algorithm}

\medskip
\noindent\textbf{Case II (Explicit radius bound unavailable).}
We now turn to the more general scenario where no explicit upper bound for the admissible parameter radius is available \emph{a priori}. In this setting, expanding the network architecture alone is insufficient, since a restricted parameter radius may prevent the objective functional from reaching the stopping threshold $\tau\delta$, even if the structural capacity $(N_k, L_k)$ is immense. 

To address this, we employ a two-phase strategy. Phase I simultaneously enlarges the architecture and an exploratory radius. If the stopping criterion is not met once the architecture reaches the theoretically sufficient scale $(N_{\rm tar}(\delta),L_{\rm tar}(\delta))$ defined above, Phase II freezes the architecture at this target scale and solely increases the radius to ensure termination. Because reaching $(N_{\rm tar}(\delta),L_{\rm tar}(\delta))$ already guarantees that the approximation error is well within the required theoretical bounds, any remaining obstruction to satisfying the stopping criterion at that stage is strictly attributed to an insufficient radius rather than to inadequate architectural capacity.

\textbf{Phase I} jointly expands the architecture and a prescribed exploratory radius sequence $\{r_k^{\mathrm I}\}$. At the $k$-th iteration, we introduce a secondary sub-iteration index $j=0$ to denote this base radius state, and compute the minimizer
\begin{equation}\label{eq:f_k_case2_phase1}
f_{k,0}^{\delta} \in \operatorname*{argmin}_{\phi\in\mathcal N(N_k, L_k,\,r_k^{\mathrm I})} J_{N_k, L_k,\,r_k^{\mathrm I}}^{\delta}(\phi).
\end{equation}
We stop Phase I as soon as the objective $S_{k,0}^{2,\delta} := J_{N_k, L_k,\,r_k^{\mathrm I}}^{\delta}(f_{k,0}^{\delta})$ satisfies $S_{k,0}^{2,\delta} \le \tau\delta$. If this occurs, we set the stopping index $j(\delta)=0$. Otherwise, once $N_k \ge N_{\rm tar}(\delta)$ and $L_k \ge L_{\rm tar}(\delta)$, we retain the current iteration index $k$ and enter Phase II.

In \textbf{Phase II}, the architecture is frozen at the target scale $(N_{\rm tar}(\delta),L_{\rm tar}(\delta))$, and the admissible radius is enlarged along an unbounded secondary sequence $\{r_j^{\mathrm{II}}\}_{j\ge1}$ with $r_j^{\mathrm{II}} \uparrow \infty$. At the $j$-th sub-iteration, we compute the minimizer
\begin{equation}\label{eq:f_k_case2_phase2}
f_{k,j}^{\delta} \in \operatorname*{argmin}_{\phi \in \mathcal{N}(N_{\rm tar}(\delta),L_{\rm tar}(\delta), \, r_j^{\mathrm{II}})} J_{N_{\rm tar}, L_{\rm tar}, \, r_j^{\mathrm{II}}}^{\delta}(\phi).
\end{equation}
We then evaluate the corresponding stopping quantity: $S_{k,j}^{2,\delta} := J_{N_{\rm tar}(\delta),L_{\rm tar}(\delta), \, r_j^{\mathrm{II}}}^{\delta}\bigl(f_{k,j}^{\delta}\bigr)$, and terminate Phase~II at the first index $j$ for which $S_{k,j}^{2,\delta}\le \tau\delta$. This stopping index is denoted by $j(\delta)$, while $k(\delta)$ is inherited as the iteration index at which Phase~I concluded. Since the architecture remains fixed throughout this phase, the final selected structural parameters are $N_{k(\delta)}=N_{\rm tar}(\delta)$ and $L_{k(\delta)}=L_{\rm tar}(\delta)$, and the final stable reconstruction is given by $f^\delta := f_{k(\delta),j(\delta)}^\delta$.

The complete procedure is summarized in Algorithm~\ref{alg:two_stage}.

\begin{algorithm}[h]
\caption{Two-stage expanding DNN regularization without explicit radius bound}
\label{alg:two_stage}
\begin{algorithmic}[1]
\State \textbf{Input:} Expansion sequences $\{N_k\}, \{L_k\}$; target architecture $(N_{\rm tar}(\delta),L_{\rm tar}(\delta))$; Phase~I radii $\{r_k^{\mathrm I}\}_{k\ge1}$; Phase~II radii $\{r_j^{\mathrm{II}}\}_{j\ge1}$ with $r_j^{\mathrm{II}}\uparrow\infty$; discrepancy parameter $\tau>1$; exponent $\theta\in(0,1]$; constant $c_0>0$; noise level $\delta$; noisy data $g^\delta$; regularizer $\mathcal R$; error profile $\overline{\mathcal E}_{N,L}$.

\State \textbf{Phase I: joint expansion of architecture and parameter radius}
\State $k\gets 0$, $j\gets 0$
\Repeat
    \State $k\gets k+1$
    \State $\beta_{N_k, L_k}\gets c_0\,\overline{\mathcal E}_{N_k,L_k}^{\theta}$
    \State Compute the minimizer $f_{k,0}^{\delta}$ via \eqref{eq:f_k_case2_phase1}
    \State Evaluate the full objective: $S_{k,0}^{2,\delta}\gets J_{N_k, L_k,\, r_k^{\mathrm I}}^{\delta}(f_{k,0}^{\delta})$
\Until{$S_{k,0}^{2,\delta}\le \tau\delta$ \textbf{or} ($N_k\ge N_{\rm tar}(\delta)$ \textbf{and} $L_k\ge L_{\rm tar}(\delta)$)}

\If{$S_{k,0}^{2,\delta}\le \tau\delta$}
    \State $k(\delta)\gets k$, $j(\delta)\gets 0$
    \State $f^\delta\gets f_{k(\delta),0}^{\delta}$
\Else
    \State \textbf{Phase II: freeze architecture at $(N_{\rm tar}(\delta),L_{\rm tar}(\delta))$ and enlarge radius}
    \Repeat
        \State $j \gets j+1$
        \State Compute the minimizer $f_{k,j}^{\delta}$ via \eqref{eq:f_k_case2_phase2}
        \State Evaluate the full objective: $S_{k,j}^{2,\delta} \gets J_{N_{\rm tar}(\delta),L_{\rm tar}(\delta),\, r_j^{\mathrm{II}}}^{\delta}(f_{k,j}^{\delta})$
    \Until{$S_{k,j}^{2,\delta} \le \tau\delta$}
    \State $k(\delta) \gets k, \quad j(\delta) \gets j$
    \State $N_{k(\delta)} \gets N_{\rm tar}(\delta), \ L_{k(\delta)} \gets L_{\rm tar}(\delta)$
    \State $f^\delta \gets f_{k(\delta), j(\delta)}^{\delta}$
\EndIf

\State \Return $f^\delta,\,N_{k(\delta)},\,L_{k(\delta)},\,k(\delta),\,j(\delta)$
\end{algorithmic}
\end{algorithm}

\begin{remark}
The target architecture $(N_{\rm tar}(\delta),L_{\rm tar}(\delta))$ is introduced as a theoretical device to rigorously guarantee finite termination. However, the resulting bounds for $N_{\rm tar}(\delta)$ and $L_{\rm tar}(\delta)$ are typically extremely large and conservative due to the worst-case nature of the underlying approximation theory. In practice, the algorithm usually terminates in Phase~I long before the network budgets approach this theoretical scale. For instance, in Example~\ref{ex:conv} with \(\delta=10^{-4}\) and \(\tau=1.6\), the theoretical target architecture is already of extremely large magnitude. Depending on the concrete constants used in the estimate, one obtains values on the order of \(N_{\rm tar}\sim 10^8\) and \(L_{\rm tar}\sim 10^7\), which are far beyond practical computational limits.
\end{remark}

\begin{remark}[Computational efficiency via warm start]
In the implementation of Algorithm~\ref{alg:main_algorithm} and Algorithm~\ref{alg:two_stage}, we use a warm-start strategy between consecutive
architecture levels. When passing from stage $k$ to stage $k+1$, the weights
learned at stage $k$ are embedded into the enlarged network and used to
initialize the new optimization. This reduces the computational cost and improves
the stability of the non-convex training process.
\end{remark}

\begin{remark}[Computational transition to Phase~II]
Although any unbounded sequence $\{r_j^{\mathrm{II}}\}$ guarantees eventual termination at the theoretical level, for computational efficiency it is natural to choose
$r_1^{\mathrm{II}} \ge r_{k_{\mathrm{tr}}}^{\mathrm I}$, 
where $k_{\mathrm{tr}}$ denotes the iteration index at which Phase~I terminates and Phase~II is activated. This preserves a monotonically nested parameter domain, so that the network parameters obtained at the end of Phase~I remain admissible in Phase~II. Consequently, they can be used as a warm-start initialization for the subsequent non-convex optimization, which avoids restarting from an unrelated initialization and may improve computational efficiency in practice.
\end{remark}

To establish finite termination and quantify the final architectural complexity across both algorithmic cases, we impose the following unifying structural conditions on the expansion schedules.

\begin{assumption}(Expansion schedules)\label{ass:expansion_generic}
Let $\{N_k\}_{k\ge1}$ and $\{L_k\}_{k\ge1}$ be nondecreasing and unbounded sequences in $\mathbb N^+$. Assume that:
\begin{itemize}
\item[(i)] There exist constants $C_N, C_L \ge 1$ such that $N_{k+1}\le C_N N_k$ and $L_{k+1}\le C_L L_k$ for all $k\ge1$.

\item[(ii)] Let $k_N^*(\delta):=\min\{k\in\mathbb N^+: N_k\ge N_{\rm tar}(\delta)\}$ and $k_L^*(\delta):=\min\{k\in\mathbb N^+: L_k\ge L_{\rm tar}(\delta)\}$, and define $k^*(\delta):=\max\{k_N^*(\delta), k_L^*(\delta)\}$. Once one architectural component reaches its target level, it is kept fixed until the other component also reaches its target; that is, $N_k=N_{k_N^*(\delta)}$ for all $k\in[k_N^*(\delta),k^*(\delta)]$, and $L_k=L_{k_L^*(\delta)}$ for all $k\in[k_L^*(\delta),k^*(\delta)]$.

\item[(iii)] There exists a constant $q_{\rm sc}\in(0,1]$ such that $\mathcal E_{C_N N,\;C_L L}\ge q_{\rm sc}\,\mathcal E_{N,L}$ for all $N,L\in\mathbb N^+$. 
\end{itemize}
\end{assumption}

\subsection{Convergence analysis of Algorithms}

We now state the main theorem.

\begin{theorem}\label{maintheorem}
Let $f^\delta$ denote the approximate solution to the operator equation $A(f)=g$ generated by either Algorithm~\ref{alg:main_algorithm} or Algorithm~\ref{alg:two_stage}. Under Assumptions~\ref{ass:generic_approx}-\ref{ass:expansion_generic}, the following assertions hold:

\begin{itemize}
\item[(a)] \textbf{Existence.}  
For every $\delta>0$, $N,L\in\mathbb N^+$, and every admissible radius $r>0$, the regularized minimization problem$$f_{N,L,r}^{\delta}\in \operatorname*{argmin}_{\phi\in\mathcal N(N,L,r)}
\Big(\|A(\phi)-g^\delta\|_{\mathcal Y}+\beta_{N,L}\mathcal R(\phi)\Big)$$admits at least one minimizer.

\item[(b)] \textbf{Finite termination and architecture control.}  
For every $\delta>0$, both Algorithm~\ref{alg:main_algorithm} and Algorithm~\ref{alg:two_stage} terminate after finitely many steps. In particular, the stopping index \(k(\delta)\) of Algorithm~\ref{alg:main_algorithm} and the Phase~I stopping index \(k(\delta)\) of Algorithm~\ref{alg:two_stage} satisfy $k(\delta)\le \max\bigl\{k_N^*(\delta),\, k_L^*(\delta)\bigr\}$, where $k_N^*(\delta)$, $k_L^*(\delta)$ are the target reaching indices defined in Assumption~\ref{ass:expansion_generic}(ii). Moreover, if Algorithm~\ref{alg:two_stage} enters Phase~II, then the corresponding inflation index \(j(\delta)\) is finite. The selected width and depth of both algorithms satisfy
\begin{equation}\label{eq:NLdelta_upper_pointwise}
N(\delta)\le \max\bigl\{N_1,\, C_N N_{\rm tar}(\delta)\bigr\},
\qquad
L(\delta)\le \max\bigl\{L_1,\, C_L L_{\rm tar}(\delta)\bigr\},
\end{equation}
where $N_1$ and $L_1$ are fixed initial constants, and $N_{\rm tar}(\delta),L_{\rm tar}(\delta)$ are selected according to \eqref{def_NLtardelta}.
\item[(c)] \textbf{Convergence.}
For every sequence $\delta_n\downarrow0$, we have
\[
f^{\delta_n}\rightharpoonup f^\dagger \quad\text{in }\mathcal X_1,
\qquad
A(f^{\delta_n})\to g \quad\text{in }\mathcal Y
\quad\text{as }n\to\infty.
\]
If, in addition, there exists a Banach space $\mathcal X_2$ such that
$\mathcal X_1$ is compactly embedded into $\mathcal X_2$, then
\[
f^{\delta_n}\to f^\dagger
\quad\text{in }\mathcal X_2
\quad\text{as }n\to\infty.
\]
\end{itemize}
\end{theorem}

To make the architecture control in Theorem~\ref{maintheorem}(b) explicit, we insert the target error level \(\mathcal E_\delta=\mathcal O(\delta^{1/\theta})\) into the approximation rates of Theorems~\ref{thm:sobolev_approx} and \ref{thm:explicit_weights}. This yields the following asymptotic growth laws for the stopping architecture as \(\delta\to0\):
\begin{itemize}
    \item \textbf{H\"older setting (Algorithm~\ref{alg:main_algorithm}).}
    Suppose that \(f^\dagger\) satisfies the assumptions of Theorem~\ref{thm:explicit_weights}. The stopping depth and width satisfy
    \begin{equation}\label{eq:scaling_hoelder}
        L(\delta)=\mathcal O\left(\log(1/\delta)\right),
        \qquad
        N(\delta)=\mathcal O\left(\delta^{-\frac{d}{\alpha\theta}}\right)\quad \delta\to 0.
    \end{equation}
    
    Taking \(\mathcal X_1=L^p([0,1]^d)\), we have $f^\delta \rightharpoonup f^\dagger$ in $L^p([0,1]^d)$ as $\delta \to 0$.

    \item \textbf{Sobolev setting (Algorithm~\ref{alg:two_stage}).}
    Suppose that \(f^\dagger\) satisfies the assumptions of Theorem~\ref{thm:sobolev_approx} for some integer \(s>1\). The stopping architecture satisfies
    \begin{equation}\label{eq:scaling_sobolev}
        N(\delta)L(\delta)
        =
        \mathcal O\left(
        \delta^{-\frac{d}{2(s-1)\theta}}|\log\delta|^2
        \right)\quad \delta\to 0.
    \end{equation}
    Taking \(\mathcal X_1=W^{1,p}([0,1]^d)\), we obtain $f^\delta \rightharpoonup f^\dagger$ in $W^{1,p}([0,1]^d)$ as $\delta\to 0$.
    In particular, by the compact embedding $ W^{1,p}([0,1]^d)\hookrightarrow\hookrightarrow L^p([0,1]^d)$, it follows that $f^\delta \to f^\dagger$ in $L^p([0,1]^d)$ as $\delta\to 0$.
  
\end{itemize}

\begin{remark}
These asymptotic bounds show that, as the noise level \(\delta\) decreases, the stopping architecture must become richer. In the H\"older setting, the dominant growth appears in the width, while the depth increases only logarithmically. In the Sobolev setting, the width-depth product exhibits polynomial growth in \(\delta^{-1}\), up to logarithmic corrections.
\end{remark}

We end this section by the proof of Theorem \ref{maintheorem}.
\begin{proof}
\textbf{(a) Existence.}
Fix $\delta>0$, architecture dimensions $N,L\in\mathbb N^+$, and an admissible search radius $r>0$. Define
\[
J_{N,L,r}^\delta(f):=\|A(f)-g^\delta\|_{\mathcal Y}+\beta_{N,L}\mathcal R(f),
\qquad f\in\mathcal N(N,L,r).
\]
Let $v:=\inf_{f\in\mathcal N(N,L,r)}J_{N,L,r}^\delta(f)$, and choose a minimizing sequence $\{f_n\}\subset\mathcal N(N,L,r)$ such that $J_{N,L,r}^\delta(f_n)\to v$ as $n\to\infty$. By Assumption~\ref{ass:R_full}(iii), there exists $f_0\in\mathcal N(N,L,r)$ such that $\mathcal R(f_0)<\infty$. Since $A$ is well defined on $\mathcal N(N,L,r)$ by Assumption~\ref{ass:forward_full}(i), it follows that $J_{N,L,r}^\delta(f_0)<\infty$. Passing to a tail if necessary, we may assume $J_{N,L,r}^\delta(f_n)\le J_{N,L,r}^\delta(f_0)+1$ for all $n\in\mathbb N$. As the data fidelity term is nonnegative, this gives $\beta_{N,L}\mathcal R(f_n)\le J_{N,L,r}^\delta(f_n)\le J_{N,L,r}^\delta(f_0)+1$, hence $\mathcal R(f_n)\le \beta_{N,L}^{-1}(J_{N,L,r}^\delta(f_0)+1)$ for all $n\in\mathbb N$. By the coercivity of $\mathcal R$, the sequence $\{f_n\}$ is bounded in $\mathcal X_1$.

Since $\mathcal X_1$ is reflexive, there exist a subsequence, not relabeled, and an element $f^*\in\mathcal X_1$ such that $f_n\rightharpoonup f^*$ in $\mathcal X_1$. Since $\{f_n\}\subset\mathcal N(N,L,r)$ and $\mathcal N(N,L,r)$ is compact in $\mathcal X_0$ by Assumption~\ref{ass:X1_pivot_general}(iii), there exist a further subsequence, again not relabeled, and some $\tilde f\in\mathcal X_0$ such that $f_n\to \tilde f$ strongly in $\mathcal X_0$. The continuous embedding $\mathcal X_1\hookrightarrow\mathcal X_0$ implies that
$f_n\rightharpoonup f^*$ also in $\mathcal X_0$. Since strong convergence in
$\mathcal X_0$ implies weak convergence in $\mathcal X_0$, the uniqueness of weak
limits gives $\tilde f=f^*$ in $\mathcal X_0$. Since $\mathcal N(N,L,r)$ is compact, hence closed, in $\mathcal X_0$, we conclude that $f^*\in\mathcal N(N,L,r)$.

By Assumption~\ref{ass:forward_full}(ii), since $\{f_n\}$ is bounded in
$\mathcal X_1$ and $f_n\rightharpoonup f^*$ in $\mathcal X_1$, we have
$A(f_n)\rightharpoonup A(f^*)$ in $\mathcal Y$, and therefore $A(f_n)-g^\delta\rightharpoonup A(f^*)-g^\delta$ in $\mathcal Y$. Since the norm is weakly lower semicontinuous, $\|A(f^*)-g^\delta\|_{\mathcal Y}\le \liminf_{n\to\infty}\|A(f_n)-g^\delta\|_{\mathcal Y}$. Moreover, by the weak lower semicontinuity of $\mathcal R$, we have $\mathcal R(f^*)\le \liminf_{n\to\infty}\mathcal R(f_n)$. Consequently,
\[
J_{N,L,r}^\delta(f^*)\le \liminf_{n\to\infty}J_{N,L,r}^\delta(f_n)=v.
\]
Since $f^*\in\mathcal N(N,L,r)$ and $v$ is the infimum over $\mathcal N(N,L,r)$, we also have $v\le J_{N,L,r}^\delta(f^*)$. Hence $J_{N,L,r}^\delta(f^*)=v$, so $f^*$ is a minimizer. This proves \textbf{(a)}.

\textbf{(b) Finite termination.}
Fix any $\delta>0$, and choose target budgets $N_{\rm tar}(\delta)\ge N_{\min},L_{\rm tar}(\delta)\ge L_{\min}$ such that $q_0\,\mathcal E_\delta\le \mathcal E_{N_{\rm tar}(\delta),L_{\rm tar}(\delta)}\le \mathcal E_\delta$. Since the sequences $\{N_k\}_{k\ge1}$ and $\{L_k\}_{k\ge1}$ are nondecreasing and unbounded, the indices
\[
k_N^*(\delta):=\min\{k\in\mathbb N^+: N_k\ge N_{\rm tar}(\delta)\},\qquad
k_L^*(\delta):=\min\{k\in\mathbb N^+: L_k\ge L_{\rm tar}(\delta)\}
\]
are well defined. Let $k^*(\delta):=\max\{k_N^*(\delta),k_L^*(\delta)\}$ and $(\overline N,\overline L):=(N_{k^*(\delta)},L_{k^*(\delta)})$.

We first consider Algorithm~\ref{alg:main_algorithm}. By the explicit-bounded approximation theorem underlying Case~I, there exists $\phi^{\mathrm I}\in \mathcal N\bigl(N_{\rm tar}(\delta),L_{\rm tar}(\delta),r_{\max}(N_{\rm tar}(\delta),L_{\rm tar}(\delta))\bigr)$ such that $\|f^\dagger-\phi^{\mathrm I}\|_{\mathcal X_1}\le \mathcal E_{N_{\rm tar}(\delta),L_{\rm tar}(\delta)}\le \mathcal E_\delta$ and $\mathcal R(\phi^{\mathrm I})\le C_{f^\dagger}$. Since $\mathcal E_\delta\le \eta$, Assumption~\ref{ass:forward_full}(iv) applies to $\phi^{\mathrm I}$. Moreover, by the monotonicity of $r_{\max}$ and the nestedness of the network classes, $\phi^{\mathrm I}\in \mathcal N\bigl(\overline N,\overline L,r_{\max}(\overline N,\overline L)\bigr)$. Since $\overline{N}\ge N_{\min}$ and $\overline{L}\ge L_{\min}$, the definition \eqref{eq:beta_effective} gives $\beta_{\overline N,\overline L}=c_0\overline{\mathcal E}_{\overline N,\overline L}^{\theta}=c_0\mathcal E_{\overline N,\overline L}^{\theta}$. Hence
\begin{align*}
J_{\overline N,\overline L,r_{\max}(\overline N,\overline L)}^\delta(\phi^{\mathrm I})
&=
\|A(\phi^{\mathrm I})-g^\delta\|_{\mathcal Y}
+\beta_{\overline N,\overline L}\mathcal R(\phi^{\mathrm I}) \\
&\le
\|A(\phi^{\mathrm I})-A(f^\dagger)\|_{\mathcal Y}
+\|g-g^\delta\|_{\mathcal Y}
+\beta_{\overline N,\overline L}\mathcal R(\phi^{\mathrm I}) \\
&\le
L_A\|\phi^{\mathrm I}-f^\dagger\|_{\mathcal X_1}^{\theta}
+\delta
+c_0\mathcal E_{\overline N,\overline L}^{\theta} C_{f^\dagger} \\
&\le
L_A\mathcal E_\delta^\theta+\delta+c_0\mathcal E_\delta^\theta C_{f^\dagger}.
\end{align*}
By the definition of $\mathcal E_\delta$, we have $L_A\mathcal E_\delta^\theta\le \frac{\tau-1}{2}\delta$ and $c_0C_{f^\dagger}\mathcal E_\delta^\theta\le \frac{\tau-1}{2}\delta$, and therefore
\[
J_{\overline N,\overline L,r_{\max}(\overline N,\overline L)}^\delta(\phi^{\mathrm I})
\le \frac{\tau-1}{2}\delta+\delta+\frac{\tau-1}{2}\delta
= \tau\delta.
\]
Let $f^\delta$ denote the minimizer computed by Algorithm~\ref{alg:main_algorithm} at iteration $k^*(\delta)$. Since the admissible class contains $\phi^{\mathrm I}$, the minimality of $f^\delta$ implies
\[
J_{\overline N,\overline L,r_{\max}(\overline N,\overline L)}^\delta(f^\delta)
\le
J_{\overline N,\overline L,r_{\max}(\overline N,\overline L)}^\delta(\phi^{\mathrm I})
<\tau\delta.
\]
Hence $S_{k^*(\delta)}^{1,\delta}=\|A(f^\delta)-g^\delta\|_{\mathcal Y}\le J_{\overline N,\overline L,r_{\max}(\overline N,\overline L)}^\delta(f^\delta)<\tau\delta$. Since \(N_{k^*(\delta)}\ge N_{\rm tar}(\delta)\ge N_{\min}\) and
\(L_{k^*(\delta)}\ge L_{\rm tar}(\delta)\ge L_{\min}\), and $\mathcal E_{N_{k^*(\delta)},L_{k^*(\delta)}}
\le
\mathcal E_{N_{\rm tar}(\delta),L_{\rm tar}(\delta)}
\le
\mathcal E_\delta
\le \eta$, all admissibility conditions in \eqref{eq:stopping_case1} are satisfied at
iteration \(k^*(\delta)\). Therefore Algorithm~\ref{alg:main_algorithm}
terminates no later than iteration \(k^*(\delta)\).

We next consider Algorithm~\ref{alg:two_stage}. By Assumption~\ref{ass:generic_approx} and Assumption~\ref{ass:R_full}(iv), there exist a finite radius $r_\delta^*>0$ and an approximant $\phi^{\mathrm{II}}\in \mathcal N(N_{\rm tar}(\delta),L_{\rm tar}(\delta),r_\delta^*)$ such that $\|f^\dagger-\phi^{\mathrm{II}}\|_{\mathcal X_1}\le \mathcal E_{N_{\rm tar}(\delta),L_{\rm tar}(\delta)}\le \mathcal E_\delta$ and $\mathcal R(\phi^{\mathrm{II}})\le C_{f^\dagger}$. Again, since $\mathcal E_\delta\le \eta$, Assumption~\ref{ass:forward_full}(iv) applies to $\phi^{\mathrm{II}}$. Because $r_j^{\mathrm{II}}\uparrow\infty$, there exists a finite index $j^*:=\min\{j\in\mathbb N^+: r_j^{\mathrm{II}}\ge r_\delta^*\}$ such that $\phi^{\mathrm{II}}\in \mathcal N(N_{\rm tar}(\delta),L_{\rm tar}(\delta),r_{j^*}^{\mathrm{II}})$. Moreover,
\begin{align*}
J_{N_{\rm tar}(\delta),L_{\rm tar}(\delta),r_{j^*}^{\mathrm{II}}}^\delta(\phi^{\mathrm{II}})
&=
\|A(\phi^{\mathrm{II}})-g^\delta\|_{\mathcal Y}
+\beta_{N_{\rm tar}(\delta),L_{\rm tar}(\delta)}\mathcal R(\phi^{\mathrm{II}}) \\
&\le
\|A(\phi^{\mathrm{II}})-A(f^\dagger)\|_{\mathcal Y}
+\|g-g^\delta\|_{\mathcal Y}
+\beta_{N_{\rm tar}(\delta),L_{\rm tar}(\delta)}\mathcal R(\phi^{\mathrm{II}}) \\
&\le
L_A\|\phi^{\mathrm{II}}-f^\dagger\|_{\mathcal X_1}^{\theta}
+\delta
+c_0\mathcal E_{N_{\rm tar}(\delta),L_{\rm tar}(\delta)}^{\theta} C_{f^\dagger} \\
&\le
L_A\mathcal E_\delta^\theta+\delta+c_0\mathcal E_\delta^\theta C_{f^\dagger}
<\tau\delta.
\end{align*}
If Phase~I already stops before or at iteration $k^*(\delta)$, then there is nothing to prove. Otherwise, by the definition of $k^*(\delta)$, Phase~I reaches the target budget by iteration $k^*(\delta)$ and enters Phase~II. At the inflation index $j^*$, the approximant $\phi^{\mathrm{II}}$ is admissible for the Phase~II minimization problem, so the minimizer $f^\delta$ computed at $(k^*(\delta),j^*)$ satisfies
\[
S_{k^*(\delta),j^*}^{2,\delta}
=
J_{N_{\rm tar}(\delta),L_{\rm tar}(\delta),r_{j^*}^{\mathrm{II}}}^\delta(f^\delta)
\le
J_{N_{\rm tar},L_{\rm tar},r_{j^*}^{\mathrm{II}}}^\delta(\phi^{\mathrm{II}})
<\tau\delta.
\]
Hence Phase~II terminates after finitely many inflation steps. In particular, $k(\delta)\le k^*(\delta)=\max\{k_N^*(\delta),k_L^*(\delta)\}$.

It remains to prove the complexity bounds. For the width sequence, if $k_N^*(\delta)=1$, then $N_{k_N^*(\delta)}=N_1$; if $k_N^*(\delta)>1$, then the minimality of $k_N^*(\delta)$ implies $N_{k_N^*(\delta)-1}<N_{\rm tar}(\delta)$, and Assumption~\ref{ass:expansion_generic}(i) yields $N_{k_N^*(\delta)}\le C_NN_{k_N^*(\delta)-1}<C_NN_{\rm tar}(\delta)$. Hence $N_{k_N^*(\delta)}\le \max\{N_1,\,C_NN_{\rm tar}(\delta)\}$. Similarly, $L_{k_L^*(\delta)}\le \max\{L_1,\,C_LL_{\rm tar}(\delta)\}$.
For Algorithm~\ref{alg:main_algorithm}, Assumption~\ref{ass:expansion_generic}(ii) yields $N_{k^*(\delta)}=N_{k_N^*(\delta)}$ and $L_{k^*(\delta)}=L_{k_L^*(\delta)}$. Thus,
\[
    N(\delta)\le N_{k^*(\delta)}=N_{k_N^*(\delta)}\le \max\{N_1,\,C_NN_{\rm tar}(\delta)\},
\]
and 
\[
    L(\delta)\le L_{k^*(\delta)}=L_{k_L^*(\delta)}\le \max\{L_1,\,C_LL_{\rm tar}(\delta)\}.
\]
For Algorithm~\ref{alg:two_stage}, if Phase~I terminates, then $k(\delta)\le k^*(\delta)$ and the monotonicity of the sequences gives $N(\delta)=N_{k(\delta)}\le N_{k^*(\delta)}$ and $L(\delta)=L_{k(\delta)}\le L_{k^*(\delta)}$, so the same bounds follow. If Phase~II is activated, then by construction the architecture is frozen at $(N_{\rm tar},L_{\rm tar})$, hence the above bounds also hold. 
This completes the proof of \textbf{(b)}.

\textbf{(c) Convergence.}
Let \(\delta\downarrow0\), and let \(f^{\delta}\) be the corresponding outputs of either algorithm with final indices \((N_{\delta},L_{\delta})\). We first derive a uniform bound for \(\mathcal R(f^{\delta})\).

For Algorithm~\ref{alg:main_algorithm}, by the stopping criterion
\eqref{eq:stopping_case1}, the final architecture satisfies $N_\delta\ge N_{\min},L_\delta\ge L_{\min}$ and $\mathcal E_{N_\delta,L_\delta}\le \eta .$ Hence the approximation property in Assumption~\ref{ass:generic_approx}
is applicable at the final architecture. Together with
Assumption~\ref{ass:R_full}(iv), this yields an approximant
$\phi_\delta\in
\mathcal N\bigl(
N_\delta,L_\delta,r_{\max}(N_\delta,L_\delta)
\bigr)$ such that $\|f^\dagger-\phi_\delta\|_{\mathcal X_1}
\le
\mathcal E_{N_\delta,L_\delta}\le \eta$, and $\mathcal R(\phi_\delta)\le C_{f^\dagger}$. Hence Assumption~\ref{ass:forward_full}(iv) applies to \(\phi_\delta\).
Denote $\beta_{\delta}=c_0\,\overline{\mathcal E}_{N_{\delta},L_{\delta}}^{\theta}$. Since $N_\delta\ge N_{\min}$ and $L_\delta\ge L_{\min}$, the definition \eqref{defoverline_E} gives that $\beta_{\delta}=c_0\mathcal E_{N_{\delta},L_{\delta}}^{\theta}$. By the optimality of \(f^\delta\), we obtain
\[
\|A(f^\delta)-g^\delta\|_{\mathcal Y}
+
\beta_\delta\mathcal R(f^\delta)
\le
\|A(\phi_\delta)-g^\delta\|_{\mathcal Y}
+
\beta_\delta\mathcal R(\phi_\delta).
\]
Discarding the nonnegative residual term on the left-hand side and using the local Hölder continuity of \(A\), we obtain
\begin{align*}
\mathcal R(f^\delta)
&\le \frac{\|A(\phi_\delta)-g^\delta\|_{\mathcal Y}}{\beta_\delta}+\mathcal R(\phi_\delta)\le \frac{\|A(\phi_\delta)-A(f^\dagger)\|_{\mathcal Y}+\|g-g^\delta\|_{\mathcal Y}}{c_0\mathcal E_{N_\delta,L_\delta}^{\theta}}+C_{f^\dagger} \\
&\le \frac{L_A\|f^\dagger-\phi_\delta\|_{\mathcal X_1}^{\theta}+\delta}{c_0\mathcal E_{N_\delta,L_\delta}^{\theta}}+C_{f^\dagger} \le \frac{L_A}{c_0}+\frac{\delta}{c_0\mathcal E_{N_\delta,L_\delta}^{\theta}}+C_{f^\dagger}.
\end{align*}
By the complexity bounds established in \textbf{(b)}, for sufficiently
small \(\delta\), we have
\(N_\delta\le C_NN_{\rm tar}(\delta)\) and
\(L_\delta\le C_LL_{\rm tar}(\delta)\). Since \(\mathcal E_{N,L}\) is
nonincreasing in each argument, Assumption~\ref{ass:expansion_generic}(iii)
and the target choice
\(q_0\mathcal E_\delta\le
\mathcal E_{N_{\rm tar}(\delta),L_{\rm tar}(\delta)}\le \mathcal E_\delta\) give
\[
\mathcal E_{N_\delta,L_\delta}
\ge
\mathcal E_{C_NN_{\rm tar}(\delta),\,C_LL_{\rm tar}(\delta)}
\ge
q_{\rm sc}\mathcal E_{N_{\rm tar}(\delta),L_{\rm tar}(\delta)}
\ge
q_{\rm sc}q_0\mathcal E_\delta.
\]
Therefore $\frac{\delta}{\mathcal E_{N_\delta,L_\delta}^{\theta}}
\le
\frac{\delta}{(q_{\rm sc}q_0)^\theta\mathcal E_\delta^\theta}$, which is uniformly bounded by the definition of \(\mathcal E_\delta\). Thus \(\mathcal R(f^\delta)\) is uniformly bounded for
Algorithm~\ref{alg:main_algorithm}.

For Algorithm~\ref{alg:two_stage}, the stopping rule gives
$J^\delta(f^\delta)\le \tau\delta$, hence
$\beta_\delta\mathcal R(f^\delta)\le \tau\delta$. Since
$\beta_\delta=c_0\mathcal E_{N_{\delta},L_{\delta}}^{\theta}$, and using the same lower bound derived above, we have
\[
    \mathcal R(f^\delta)
\le
\frac{\tau\delta}
{c_0\mathcal E_{N_{\delta},L_{\delta}}^{\theta}}\le \frac{\tau\delta}
{c_0(q_{\rm sc}q_0)^\theta\mathcal E_\delta^\theta}.
\]
Since \(\delta/\mathcal E_\delta^\theta\) is uniformly bounded,
\(\mathcal R(f^\delta)\) is uniformly bounded in both cases.

Now let $\delta_n\downarrow0$. Since $\mathcal R(f^{\delta_n})$ is uniformly bounded and $\mathcal R$ is coercive, the sequence $\{f^{\delta_n}\}$ is bounded in $\mathcal X_1$. By reflexivity, there exist a subsequence, not relabeled, and an element $f^*\in\mathcal X_1$ such that $f^{\delta_n}\rightharpoonup f^*$ in $\mathcal X_1$. By Assumption~\ref{ass:forward_full}(ii), we have
$A(f^{\delta_n})\rightharpoonup A(f^*)$ in $\mathcal Y$. For Algorithm~\ref{alg:main_algorithm}, the stopping rule yields $\|A(f^{\delta_n})-g^{\delta_n}\|_{\mathcal Y}\le \tau\delta_n$. For Algorithm~\ref{alg:two_stage}, since $\|A(f^{\delta_n})-g^{\delta_n}\|_{\mathcal Y}\le J^{\delta_n}(f^{\delta_n})\le \tau\delta_n$, the same estimate holds. Therefore, in either case,
\[
\|A(f^{\delta_n})-g\|_{\mathcal Y}
\le
\|A(f^{\delta_n})-g^{\delta_n}\|_{\mathcal Y}+\|g^{\delta_n}-g\|_{\mathcal Y}
\le
(\tau+1)\delta_n\to0.
\]
Hence $A(f^{\delta_n})\to g$ strongly in $\mathcal Y$. Since strong convergence implies weak convergence, the uniqueness of weak limits yields
$A(f^*)=g=A(f^\dagger)$. By Assumption~\ref{ass:forward_full}(iii), the equation
$A(f)=g$ admits the unique solution $f^\dagger$ in $\mathcal X_1$. Hence
$f^*=f^\dagger$.
Since the above argument applies to every subsequence of $\{f^{\delta_n}\}$, every
subsequence admits a further subsequence converging weakly to $f^\dagger$ in
$\mathcal X_1$. By the standard subsequence criterion, the whole sequence
$\{f^{\delta_n}\}$ converges weakly to $f^\dagger$ in $\mathcal X_1$.

Finally, suppose that $\mathcal X_1$ is compactly embedded into a Banach space $\mathcal X_2$. Then the bounded sequence $\{f^{\delta_n}\}$ admits a further subsequence, again not relabeled, such that $f^{\delta_n}\to \tilde f$ in $\mathcal X_2$ for some $\tilde f\in\mathcal X_2$. Since the embedding $\mathcal X_1\hookrightarrow\mathcal X_2$ is continuous, the weak convergence in $\mathcal X_1$ implies weak convergence in $\mathcal X_2$, so $f^{\delta_n}\rightharpoonup f^\dagger$ in $\mathcal X_2$. Since strong convergence implies weak convergence and weak limits are unique in $\mathcal X_2$, we obtain $\tilde f=f^\dagger$. Therefore, every subsequence of $\{f^{\delta_n}\}$ has a further subsequence
converging strongly to $f^\dagger$ in $\mathcal X_2$. This implies the strong
convergence of the whole sequence. Indeed, otherwise there would exist
$\varepsilon>0$ and a subsequence such that
\[
\|f^{\delta_n}-f^\dagger\|_{\mathcal X_2}\ge \varepsilon
\qquad \text{for all } n,
\]
which contradicts the existence of a further subsequence converging strongly to
$f^\dagger$ in $\mathcal X_2$. Hence $f^{\delta_n}\to f^\dagger$ strongly in $\mathcal X_2$.

\end{proof}

\section{Experimental results}
\label{sec:simulation}

In this section, we present numerical experiments on three classical ill-posed inverse problems to empirically validate the proposed expanding deep neural network schemes (Algorithms~\ref{alg:main_algorithm} and \ref{alg:two_stage}). Specifically, these experiments are designed to illustrate the regularization properties and convergence results established in Theorem~\ref{maintheorem}, and to demonstrate that the discrepancy principle serves as an effective stopping criterion for yielding accurate, noise-robust reconstructions. In all examples, the noisy data $g^\delta$ are generated by adding random noise to the exact data $g$ such that $\|g^\delta - g\|_{\mathcal{Y}} \le \delta$. To instantiate the theoretical framework, we adopt the canonical regularizer $\mathcal R(f):=\|f\|_{\mathcal X_1}$ as discussed in Remark~\ref{rem:R_choice}.

A crucial practical observation regarding Algorithm~\ref{alg:two_stage} is that, although it incorporates a two-stage mechanism to ensure finite termination theoretically, the discrepancy criterion $S_{k,0}^{2,\delta} \le \tau \delta$ is consistently met during Phase~I in our experiments. This occurs because the practical expressivity of deep neural networks significantly surpasses the worst-case approximation bounds ($N_{\rm tar}, L_{\rm tar}$) derived in Theorem~\ref{thm:sobolev_approx}. Consequently, the expanding architecture combined with a heuristically increasing radius $r_k^{\mathrm I}$ provides sufficient capacity to capture the regularized solution within the noise tolerance. Since Phase~II is never triggered in practice, we exclusively report the architectural evolution and results from Phase~I.

All numerical experiments are implemented in Python using the PyTorch framework and executed on a workstation with two NVIDIA L40 GPUs. The source code is publicly available at \url{https://github.com/z1998w/DNNip2}. To ensure an unbiased algorithmic comparison, we employ a unified ReLU network expansion schedule: width $N_k=\max \{4d,2^{kd}+1\}$ and depth $L_k=2k+3$. This schedule is directly motivated by the H\"older approximation in Theorem~\ref{thm:explicit_weights}, and is therefore particularly natural for Algorithm~\ref{alg:main_algorithm}. For Algorithm~\ref{alg:two_stage}, although this choice is not tied to its underlying approximation theorem, it still provides a valid monotone and unbounded architecture expansion path. We adopt this same schedule in the experiments to place both algorithms on an identical architectural scale and thereby enable a fair comparison of their stopping behavior and reconstruction performance. While the architecture schedule is unified, the parameter restrictions are matched to their respective theoretical frameworks. Specifically, Algorithm~\ref{alg:main_algorithm} enforces the theoretical explicit radius bound $r_k = 2(k d \vee \|f^\dagger\|_{\infty}) 2^{k(d \vee(p \alpha))}$. In contrast, Phase~I of Algorithm~\ref{alg:two_stage} utilizes a heuristic linearly growing exploratory radius $r_k^{\mathrm I} = 1000k$. Finally, to prevent severe gradient instability and prohibitive computational costs, we restrict the maximum expansion index to $k=5$.

\begin{table}[htpb]
\centering
\caption{Problem-specific hyperparameters and discretization sizes. Hyphens (-) indicate that the corresponding parameter is not applicable because Algorithm~\ref{alg:main_algorithm} lacks an explicit radius bound for Example~\ref{ex:EIT}.}\label{tab:hyperparameters}
\renewcommand{\arraystretch}{1.2}
\setlength{\tabcolsep}{8pt}
\begin{tabular}{c|lccc}
\toprule
\multicolumn{2}{c}{Settings}
& Example~\ref{ex:conv} 
& Example~\ref{ex:heat}
& Example~\ref{ex:EIT} \\
\midrule
\multirow{2}{*}{Algorithm~\ref{alg:main_algorithm}} 
  & $c_0$  & $0.02$ & $0.024$ & - \\
  & $\tau$ & $1.6$   & $1.02$   & - \\
\midrule
\multirow{2}{*}{Algorithm~\ref{alg:two_stage}} 
  & $c_0$  & $10^{-8}$ & $10^{-8}$ & $8\times 10^{-7}$ \\
  & $\tau$ & $1.05$ & $1.2$ & $1.2$ \\
\midrule
\multirow{2}{*}{Grid} 
  & $M_{\text{train}}$ & $100\times 100$ & $100\times 100$ & $50\times 50$ \\
  & $M_{\text{test}}$  & $200\times 200$ & $200\times 200$ & $100\times 100$ \\
\bottomrule
\end{tabular}
\end{table}
During the optimization process, networks are trained using full-batch gradient descent with a maximum of 50,000 epochs per expansion stage. We utilize the Adam optimizer (seed 2026) with an initial learning rate of $10^{-3}$, paired with a ReduceLROnPlateau scheduler that halves the learning rate upon 2000 epochs of stagnation (down to a minimum of $10^{-6}$). 


For evaluation, the domain $\Omega$ is discretized into a uniform training grid $\mathcal{X}_{\text{train}} = \{x_j\}_{j=1}^{M_{\text{train}}}$ for network fitting, and an independent test grid $\mathcal{X}_{\text{test}} = \{z_i\}_{i=1}^{M_{\text{test}}}$ for error assessment. The reconstruction accuracy is measured via the relative discrete $L^2$ error:
\[
  e_f := \frac{\| f^\delta - f^\dagger \|_{L^2(\mathcal{X}_{\text{test}})}}{\| f^\dagger \|_{L^2(\mathcal{X}_{\text{test}})}} 
  = \left( \frac{\sum_{i=1}^{M_{\text{test}}} \bigl| f^\delta(z_i) - f^\dagger(z_i) \bigr|^2}{\sum_{i=1}^{M_{\text{test}}} \bigl| f^\dagger(z_i) \bigr|^2} \right)^{1/2}.
\]
The problem-specific hyperparameters and grid configurations for the three examples are summarized in Table~\ref{tab:hyperparameters}.

\begin{figure}[H]
  \centering
  \includegraphics[width=0.9\linewidth]{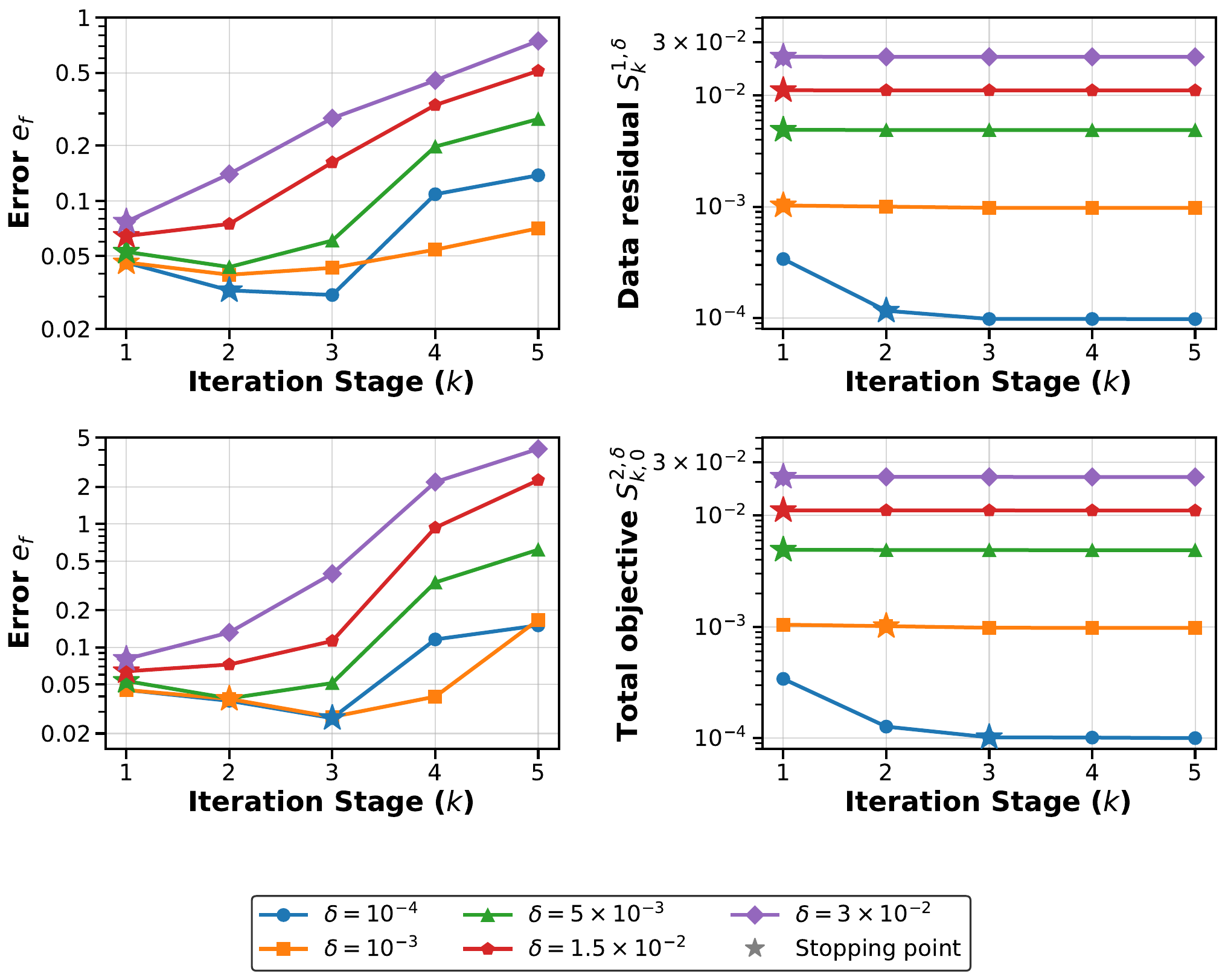}
  \caption{Example~\ref{ex:conv}: Evolution of the relative test error $e_f$ (left) and the stopping criterion value (right) for Algorithm~\ref{alg:main_algorithm} ($S_k^{1,\delta}$) and Algorithm~\ref{alg:two_stage} ($S_{k,0}^{2,\delta}$) under different noise levels $\delta$. Stars indicate the first iteration at which the stopping criterion is satisfied.}
  \label{fig:deconv_error_residuals}
\end{figure}

\begin{example}[Fredholm Integral Equation / Deconvolution]\label{ex:conv}
We consider a two-dimensional Fredholm integral equation of the first kind on the unit square \(\Omega=[0,1]^2\), corresponding to a Gaussian deconvolution problem. The forward operator is given by
\[
A(f)(\boldsymbol{x})=\int_{\Omega}\kappa(\boldsymbol{x},\boldsymbol{y})f(\boldsymbol{y})\,d\boldsymbol{y}=g(\boldsymbol{x}),
\]
where $\kappa(\boldsymbol{x}, \boldsymbol{y}) = \frac{1}{2\pi\ell_1^2} \exp\left(-\frac{|\boldsymbol{x}-\boldsymbol{y}|^2}{2\ell_1^2}\right)$ with $\ell_1=0.1$. The exact solution is chosen as \(f^\dagger(x_1,x_2)=0.1\sin(\pi x_1)\sin(\pi x_2)\).
\end{example}

\paragraph{Verification for Algorithm~\ref{alg:main_algorithm}}
We set \(\mathcal X_1=\mathcal X_A=\mathcal Y=L^2(\Omega)\). Since \(f^\dagger\) is smooth, it satisfies the Lipschitz condition required by Theorem~\ref{thm:explicit_weights} with \(\alpha=1\), and hence Assumption~\ref{ass:generic_approx} holds. Assumption~\ref{ass:X1_pivot_general} is satisfied by setting \(\mathcal X_0=H^{-1}(\Omega)\). For Assumption~\ref{ass:forward_full}, condition (i) follows immediately from \(\mathcal X_1=\mathcal X_A=L^2(\Omega)\). Moreover, \(A\) is a bounded linear operator on \(L^2(\Omega)\) by Young's inequality. Therefore, condition (ii) holds by sequential weak-to-weak continuity of bounded linear operators, and condition (iv) holds with \(\theta=1\). Finally, condition (iii) follows from the injectivity induced by the strict positive definiteness of the Gaussian kernel.
\begin{figure}[H]
  \centering
  \includegraphics[width=\linewidth]{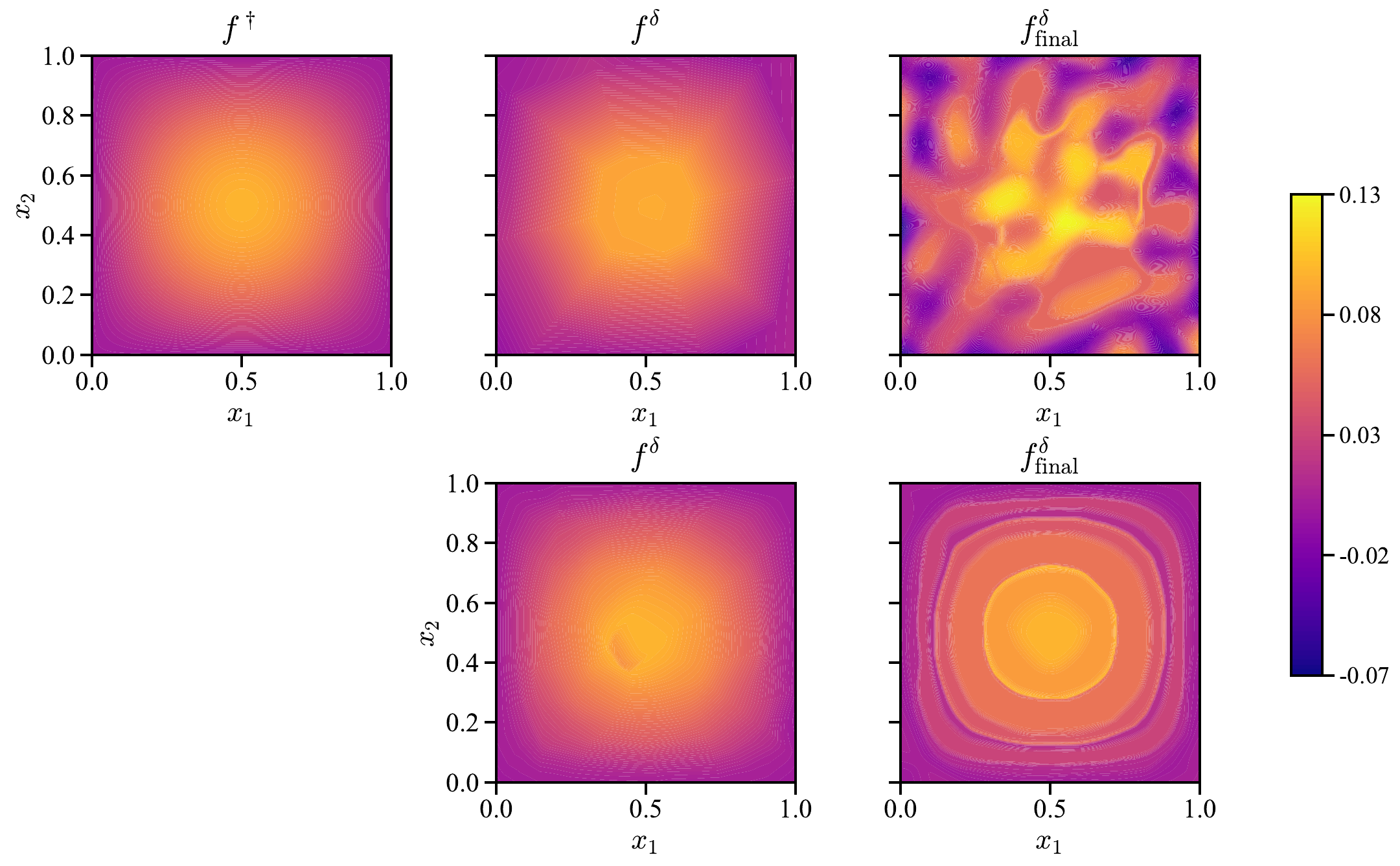}
  \caption{Example \ref{ex:conv}: True solution \(f^\dagger\) and the corresponding reconstructions \(f^\delta\) and \(f_{\text{final}}^\delta\) obtained by Algorithm~\ref{alg:main_algorithm} (first row) and Algorithm~\ref{alg:two_stage} (second row) for \(\delta = 0.0001\). Here, \(f_{\text{final}}^\delta\) denotes the reconstruction at the final expansion stage (\(k=5\)).}
  \label{fig:deconv_reconstructions}
\end{figure}

\paragraph{Verification for Algorithm~\ref{alg:two_stage}}
We set \(\mathcal X_1=H^1(\Omega)\), and \(\mathcal X_A=\mathcal Y=L^2(\Omega)\). By Theorem~\ref{thm:sobolev_approx} with \(s=2\) and \(p=2\), Assumption~\ref{ass:generic_approx} holds for the present example. Assumption~\ref{ass:X1_pivot_general} is satisfied by setting \(\mathcal X_0=L^2(\Omega)\). For Assumption~\ref{ass:forward_full}, well-definedness follows from the continuous embedding \(H^1(\Omega)\hookrightarrow L^2(\Omega)\), and injectivity is inherited from the \(L^2\)-setting. Furthermore, the compact embedding \(H^1(\Omega)\hookrightarrow\hookrightarrow L^2(\Omega)\) implies that weak convergence in \(H^1(\Omega)\) yields strong convergence in \(L^2(\Omega)\), and hence the required weak-to-weak continuity. Finally, let $\|A\|_{\mathrm{op}}$ denote the operator norm of $A$ on $L^2(\Omega)$ and $C_{\mathrm{emb}}$ denote the embedding constant of $H^1(\Omega)\hookrightarrow L^2(\Omega)$. Using the boundedness of $A$ together with this continuous embedding, we obtain$$\|A(f)-A(f^\dagger)\|_{L^2(\Omega)} \le \|A\|_{\mathrm{op}}\,\|f-f^\dagger\|_{L^2(\Omega)} \le \|A\|_{\mathrm{op}}\,C_{\mathrm{emb}}\|f-f^\dagger\|_{H^1(\Omega)}.$$This verifies Assumption~\ref{ass:forward_full}(iv) with $\theta=1$ and the corresponding constant $L_A = \|A\|_{\mathrm{op}}\,C_{\mathrm{emb}}$.

\begin{figure}[H]
\centering
\begin{subfigure}[t]{0.36\linewidth}
  \centering
  \includegraphics[width=\linewidth]{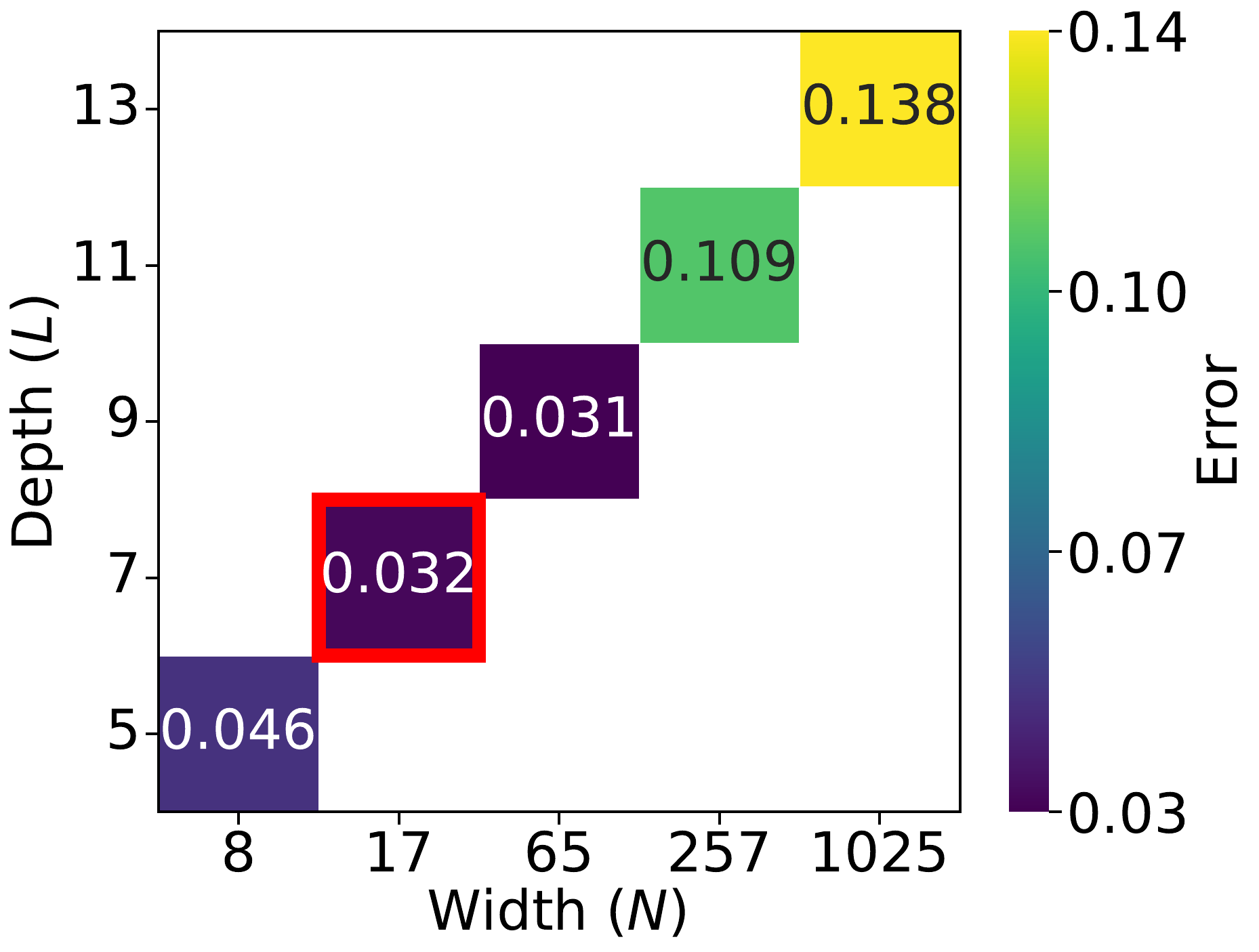}
  \caption{$\delta=0.0001$}
  \label{fig:Ex1_A1_results_delta_0_00001}
\end{subfigure}
\begin{subfigure}[t]{0.36\linewidth}
  \centering
  \includegraphics[width=\linewidth]{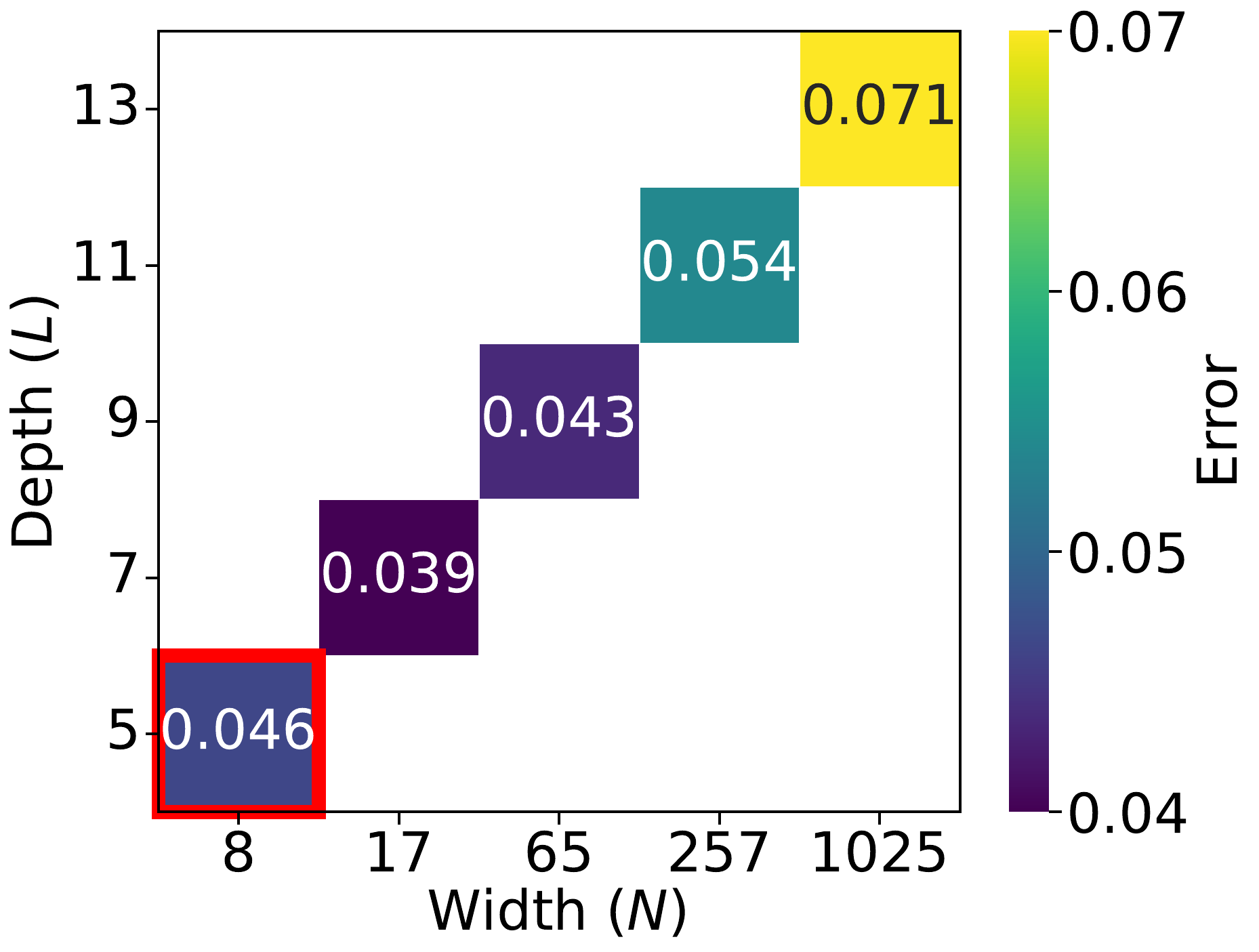}
  \caption{$\delta=0.001$}
  \label{fig:Ex1_A1_results_delta_0_001}
\end{subfigure}
\begin{subfigure}[t]{0.36\linewidth}
  \centering
  \includegraphics[width=\linewidth]{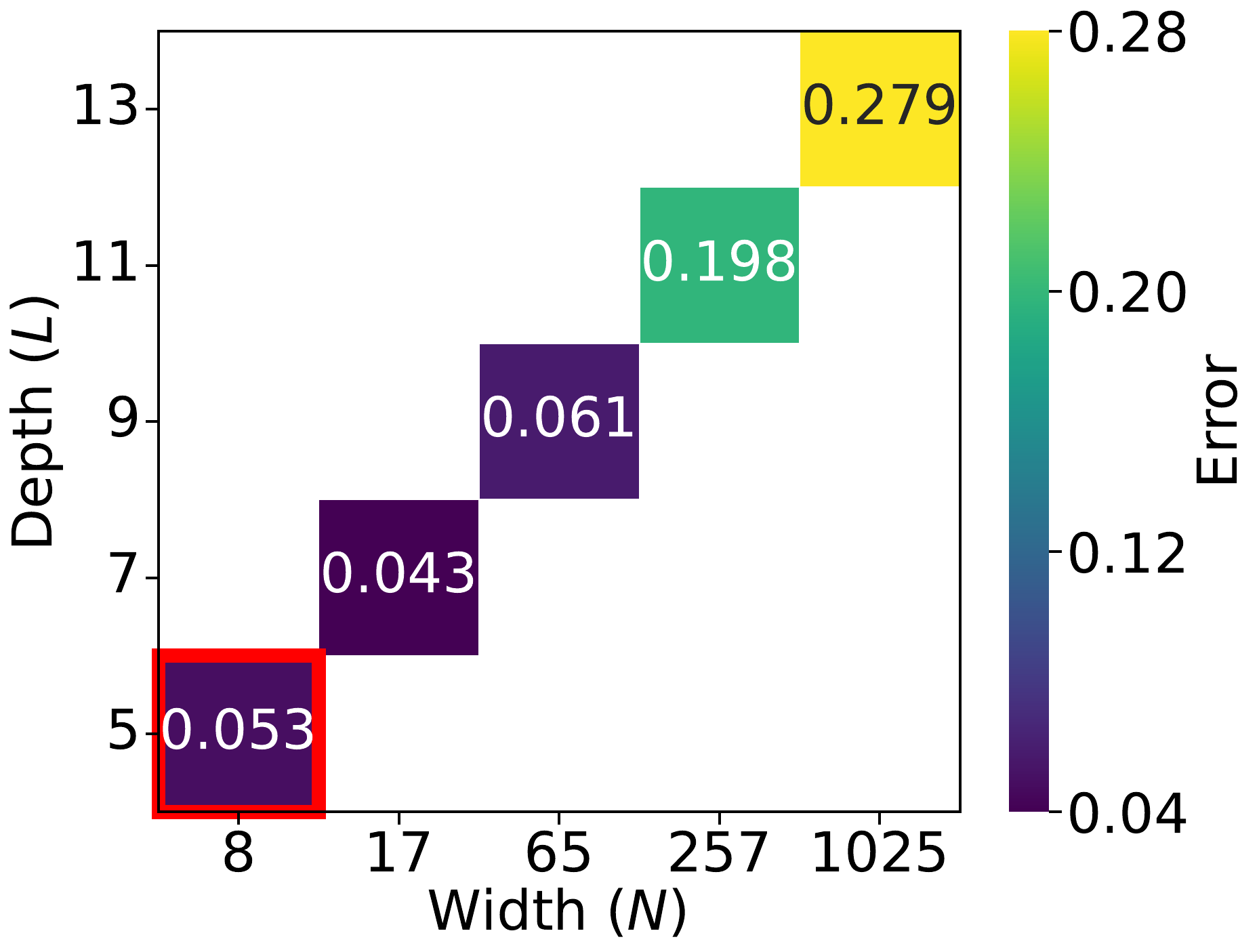}
  \caption{$\delta=0.005$}
  \label{fig:Ex1_A1_results_delta_0_005}
\end{subfigure}
\begin{subfigure}[t]{0.36\linewidth}
  \centering
  \includegraphics[width=\linewidth]{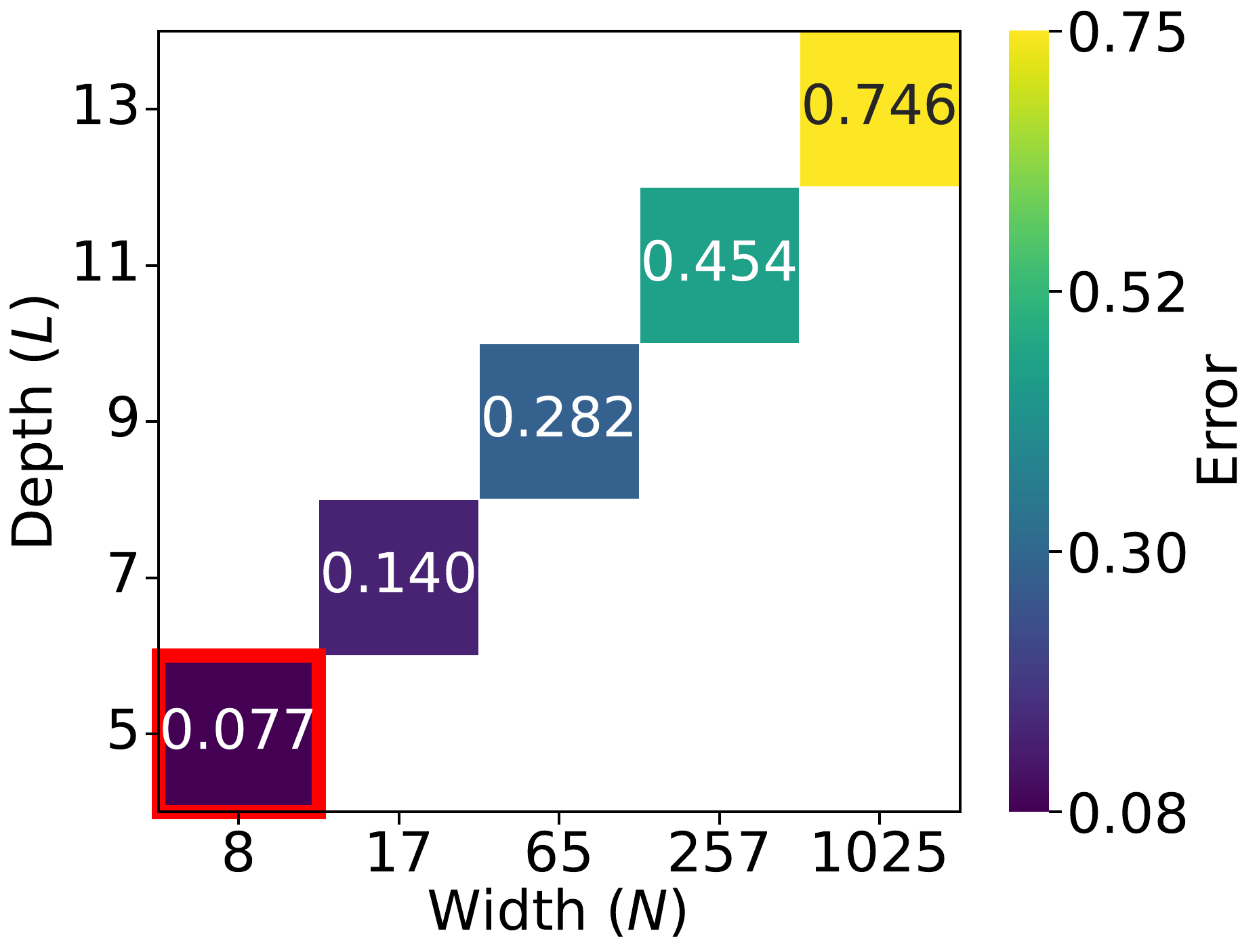}
  \caption{$\delta=0.03$}
  \label{fig:Ex1_A1_results_delta_0_03}
\end{subfigure}

\vspace{0.6ex}

\begin{subfigure}[t]{0.36\linewidth}
  \centering
  \includegraphics[width=\linewidth]{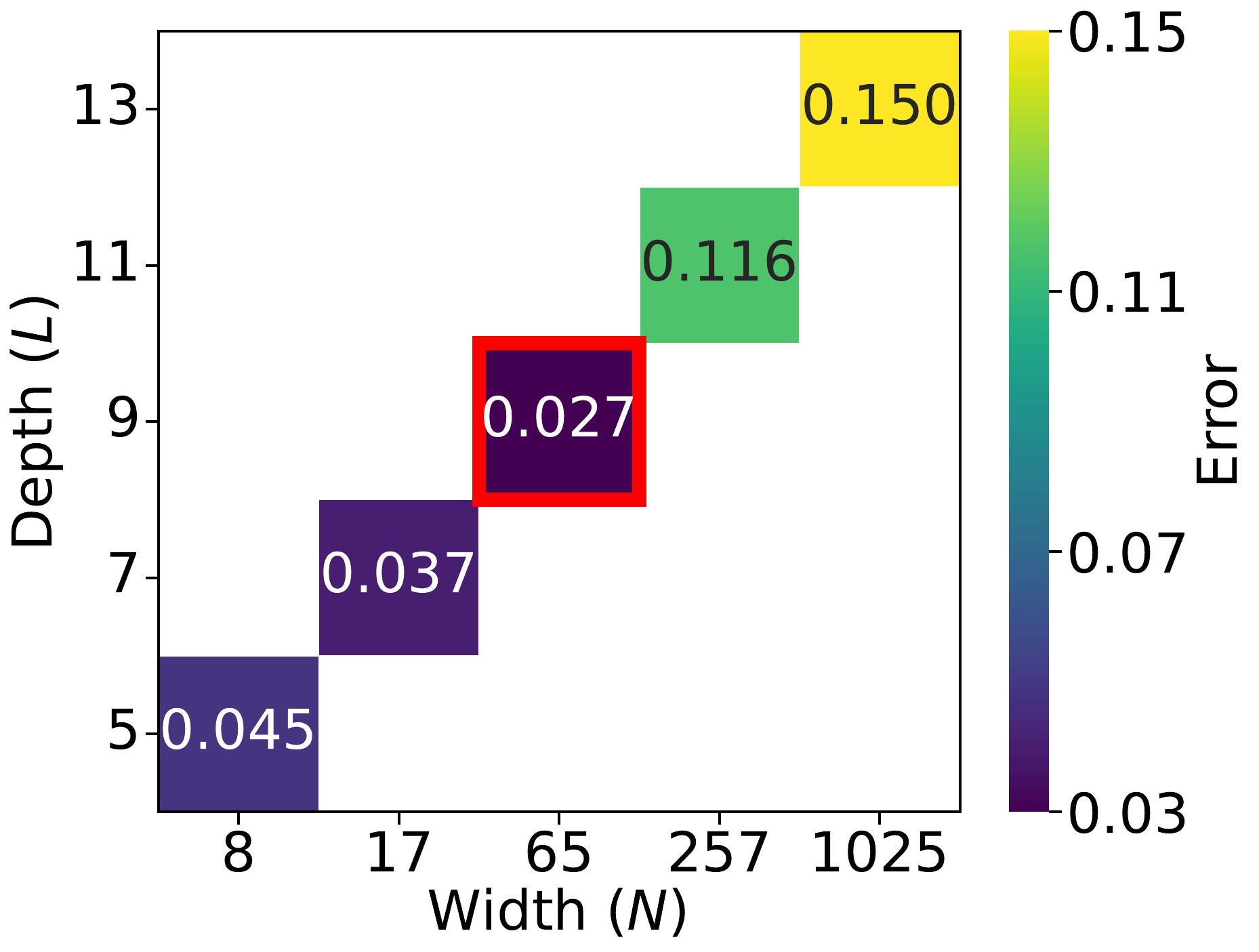}
  \caption{$\delta=0.0001$}
  \label{fig:Ex1_A2_results_delta_0_0001}
\end{subfigure}
\begin{subfigure}[t]{0.36\linewidth}
  \centering
  \includegraphics[width=\linewidth]{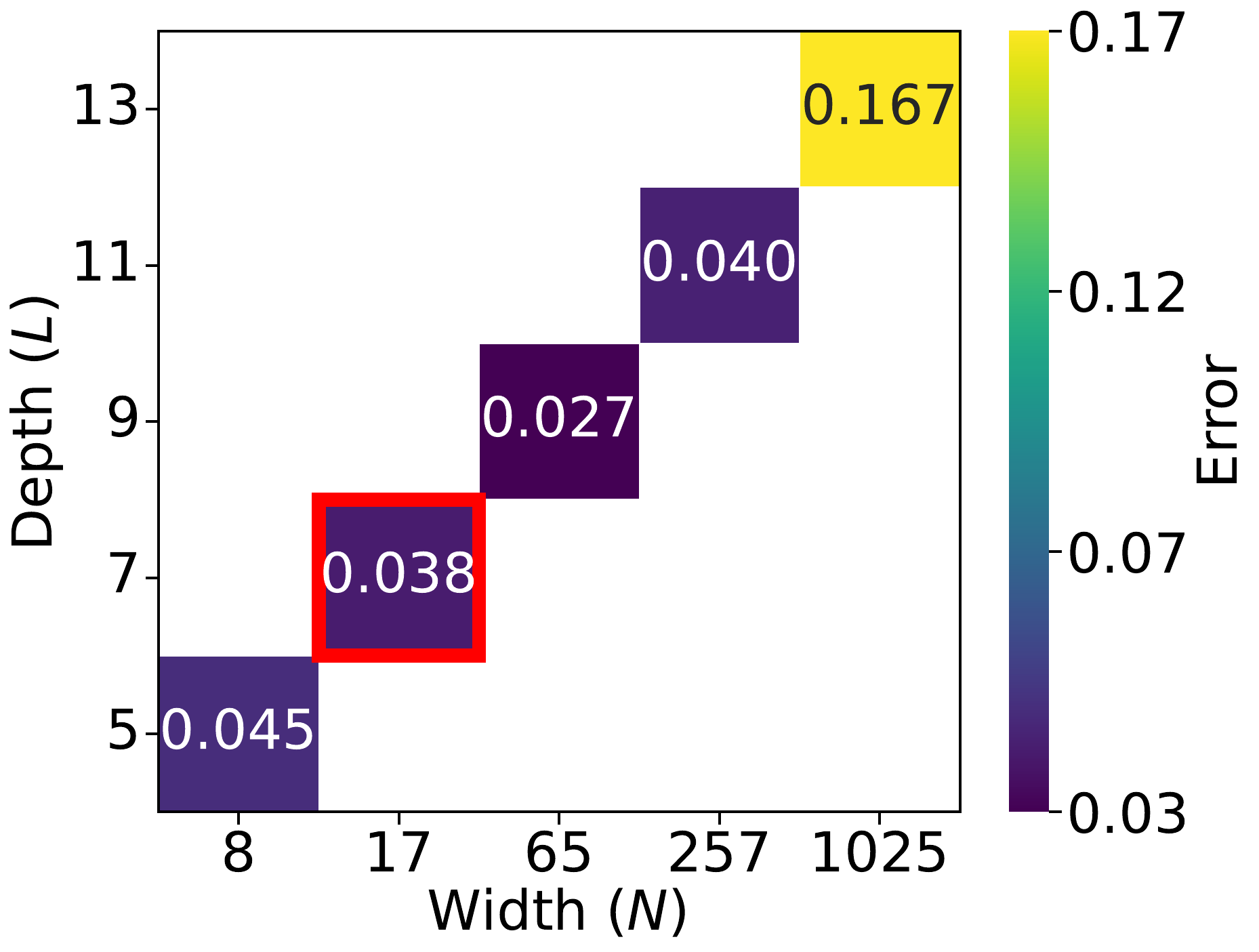}
  \caption{$\delta=0.001$}
  \label{fig:Ex1_A2_results_delta_0_001}
\end{subfigure}
\begin{subfigure}[t]{0.36\linewidth}
  \centering
  \includegraphics[width=\linewidth]{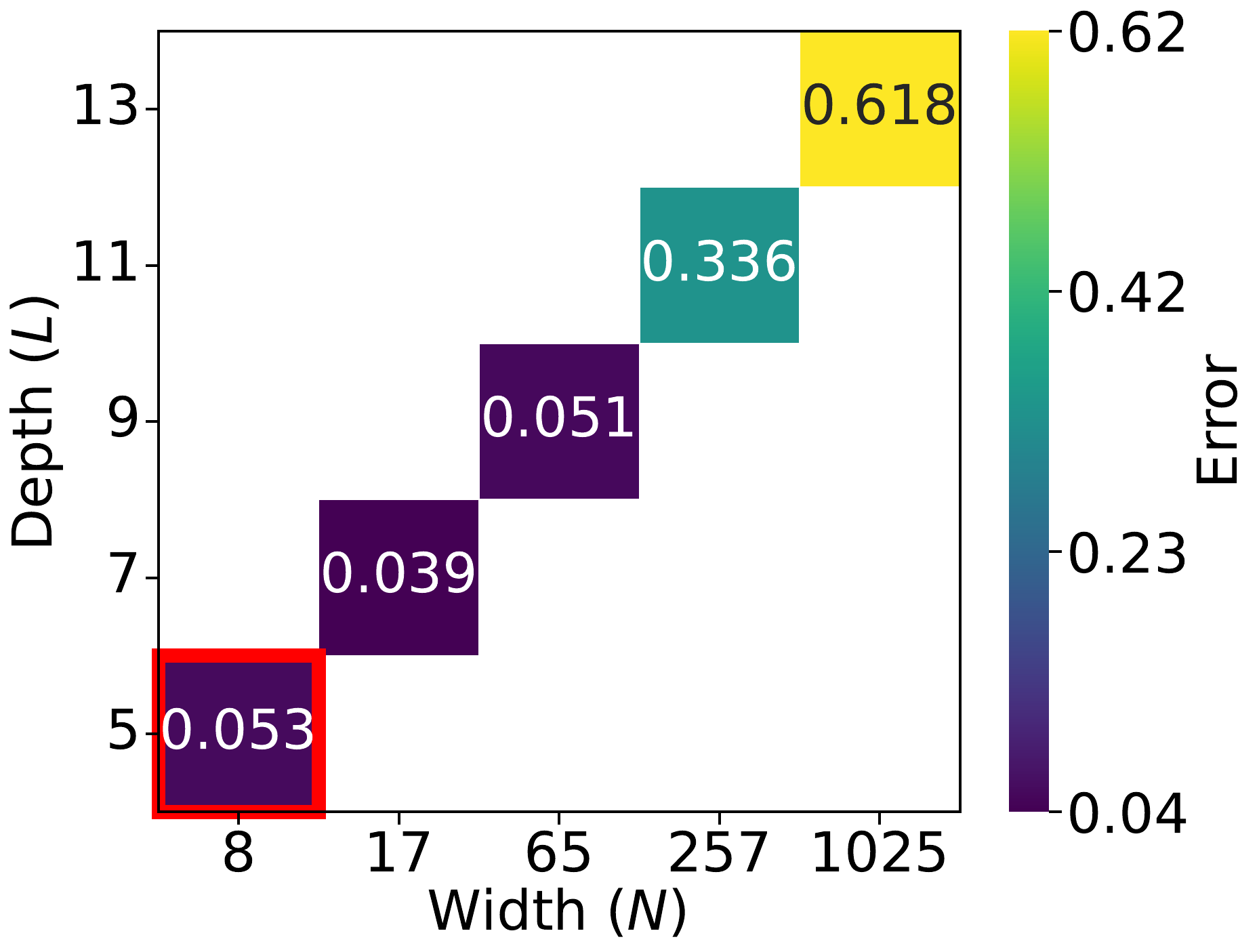}
  \caption{$\delta=0.005$}
  \label{fig:Ex1_A2_results_delta_0_005}
\end{subfigure}
\begin{subfigure}[t]{0.36\linewidth}
  \centering
  \includegraphics[width=\linewidth]{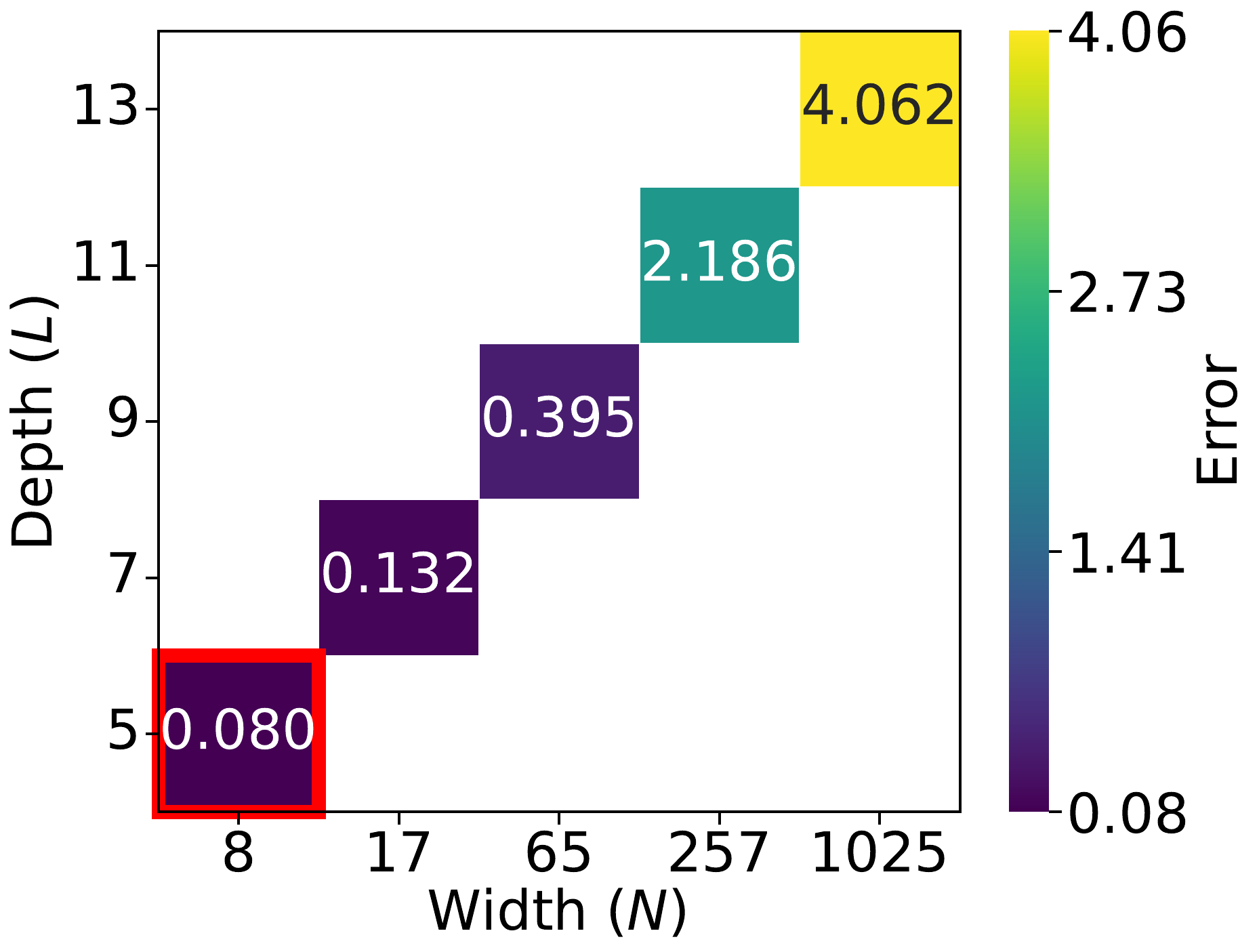}
  \caption{$\delta=0.03$}
  \label{fig:Ex1_A2_results_delta_0_03}
\end{subfigure}

\caption{Example \ref{ex:conv}: Relative \(L^2\) error \(e_f\) of the predicted \(f\) obtained by Algorithm~\ref{alg:main_algorithm} (first row) and Algorithm~\ref{alg:two_stage} (second row) for different network architectures (\(L\) layers, \(N\) neurons). Each column corresponds to a different noise level \(\delta\) in the data \(g\). The red box marks the first architecture for which the stopping criterion is satisfied.}
\label{fig:deconv_delta_grid}
\end{figure}

\begin{example}[Backward Heat Conduction]\label{ex:heat}
We consider the backward heat conduction problem on the unit square \(\Omega=[0,1]^2\) governed by
\begin{equation}\label{eq:heat_poly}
\left\{
\begin{aligned}
    \partial_t u-\Delta u &= 0, && (\boldsymbol{x},t)\in\Omega\times(0,T],\\
    u &= 0, && (\boldsymbol{x},t)\in\partial\Omega\times[0,T],\\
    u(\boldsymbol{x},0) &= f(\boldsymbol{x}), && \boldsymbol{x}\in\Omega.
\end{aligned}
\right.
\end{equation}
The forward operator is the solution map at final time $T=0.01$, given by $A(f) := u(\cdot, T) = g$, where $g$ denotes the final temperature distribution. We choose the exact solution as \(f^\dagger(x_1,x_2)=0.4\,x_1(1-x_1)x_2(1-x_2)\).
\end{example}
\paragraph{Verification for Algorithm~\ref{alg:main_algorithm}}
As in Example~\ref{ex:conv}, we set \(\mathcal X_1=\mathcal X_A=\mathcal Y=L^2(\Omega)\) and \(\mathcal X_0=H^{-1}(\Omega)\). Assumptions~\ref{ass:generic_approx} and \ref{ass:X1_pivot_general} are verified exactly as in Example~\ref{ex:conv}. For Assumption~\ref{ass:forward_full}, condition (i) follows immediately from \(\mathcal X_1=\mathcal X_A=L^2(\Omega)\). Moreover, the heat solution operator is bounded and linear on \(L^2(\Omega)\) by the standard energy estimate; therefore, condition (ii) holds by sequential weak-to-weak continuity of bounded linear operators, and condition (iv) holds with \(\theta=1\). Finally, condition (iii) follows from the classical backward uniqueness of the heat equation.
\begin{figure}[H]
  \centering
  \includegraphics[width=0.9\linewidth]{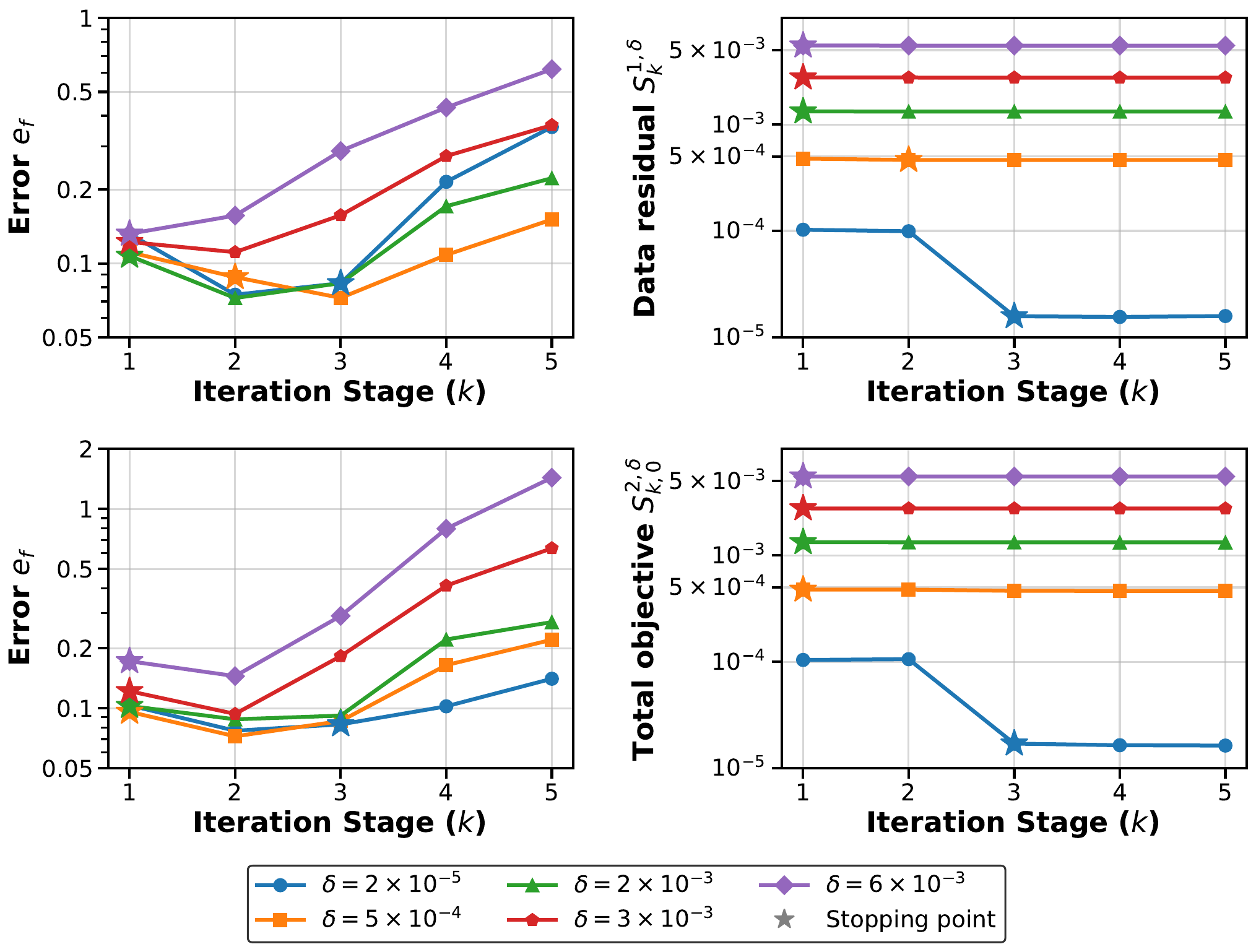}
  \caption{Example \ref{ex:heat}: Evolution of the relative test error $e_f$ (left) and the stopping criterion value (right) for Algorithm~\ref{alg:main_algorithm} ($S_k^{1,\delta}$) and Algorithm~\ref{alg:two_stage} ($S_{k,0}^{2,\delta}$) under different noise levels $\delta$. Stars indicate the first iteration at which the stopping criterion is satisfied.}
  \label{fig:heat_error_residual}
\end{figure}
\begin{figure}[H]
  \centering
  \includegraphics[width=\linewidth]{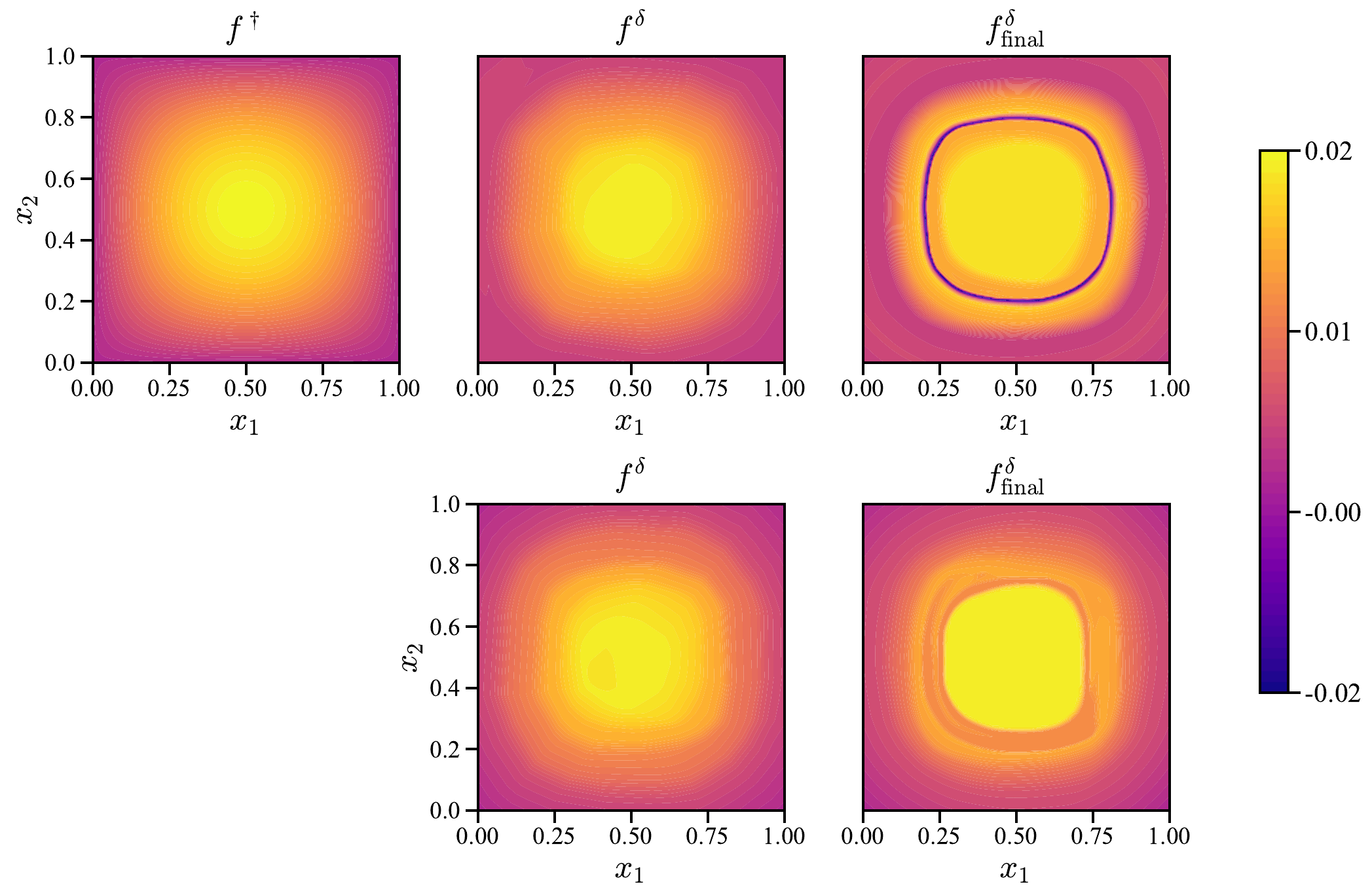}
  \caption{Example \ref{ex:heat}: True solution \(f^\dagger\) and the corresponding reconstructions \(f^\delta\) and \(f_{\text{final}}^\delta\) obtained by Algorithm~\ref{alg:main_algorithm} (first row) and Algorithm~\ref{alg:two_stage} (second row) for \(\delta = 0.00002\). Here, \(f_{\text{final}}^\delta\) denotes the reconstruction at the final expansion stage (\(k=5\)).}
  \label{fig:heat_reconstructions}
\end{figure}

\paragraph{Verification for Algorithm~\ref{alg:two_stage}}
Similarly, by setting \(\mathcal X_1=H^1(\Omega)\), \(\mathcal X_0=L^2(\Omega)\), and \(\mathcal X_A=\mathcal Y=L^2(\Omega)\), Assumptions~\ref{ass:generic_approx}, \ref{ass:X1_pivot_general}, and \ref{ass:forward_full} are verified exactly as in Example~\ref{ex:conv}.

\begin{example}[Electrical Impedance Tomography]\label{ex:EIT}
We consider the non-linear parameter identification problem for electrical impedance tomography (EIT) on the unit square \(\Omega=[0,1]^2\). The governing conductivity equation is
\begin{equation}\label{eq:eit_pde}
    \nabla \cdot \bigl(f(\boldsymbol{x})\nabla u(\boldsymbol{x})\bigr)=0, \qquad \boldsymbol{x}\in\Omega,
\end{equation}
where \(f(\boldsymbol{x})\) denotes the electrical conductivity and \(u(\boldsymbol{x})\) the electrical potential. For a prescribed boundary voltage \(h\), the potential satisfies the Dirichlet condition \(u|_{\partial\Omega}=h\), and the induced 
 boundary current is given by \(f\,\partial_\nu u|_{\partial\Omega}\), where \(\nu\) is the unit outward normal vector. The corresponding forward operator is the Dirichlet-to-Neumann (DN) map \(\Lambda_f:h\mapsto f\,\partial_\nu u|_{\partial\Omega}\). Accordingly, the classical Calder\'on problem of recovering the conductivity from boundary measurements can be formulated as the non-linear operator equation \(A(f)=g\), where \(A(f):=\Lambda_f\) and \(g\) denotes the measured DN data. To ensure uniform ellipticity of \eqref{eq:eit_pde}, the conductivity is required to satisfy \(f(\boldsymbol{x})\ge c>0\). We choose the exact conductivity \(f^\dagger(x_1,x_2)=0.1+0.1\sin(\pi x_1)\sin(\pi x_2)\), which satisfies \(f^\dagger\ge 0.1\) on \(\Omega\).

\begin{figure}[H]
\centering
\begin{subfigure}[t]{0.36\linewidth}
  \centering
  \includegraphics[width=\linewidth]{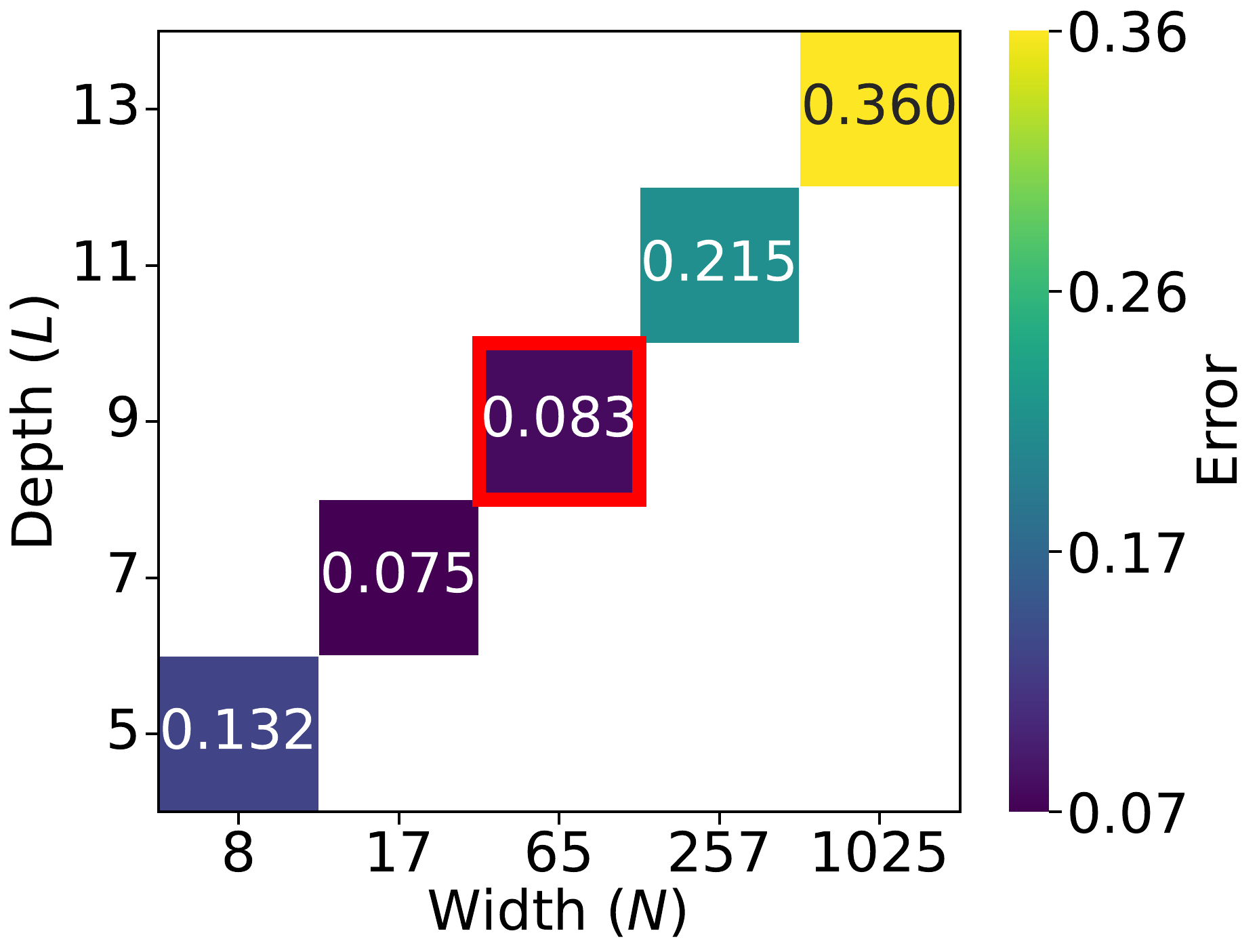}
  \caption{$\delta=0.00002$}
  \label{fig:Ex2_A1_results_delta_0_00002}
\end{subfigure}
\begin{subfigure}[t]{0.36\linewidth}
  \centering
  \includegraphics[width=\linewidth]{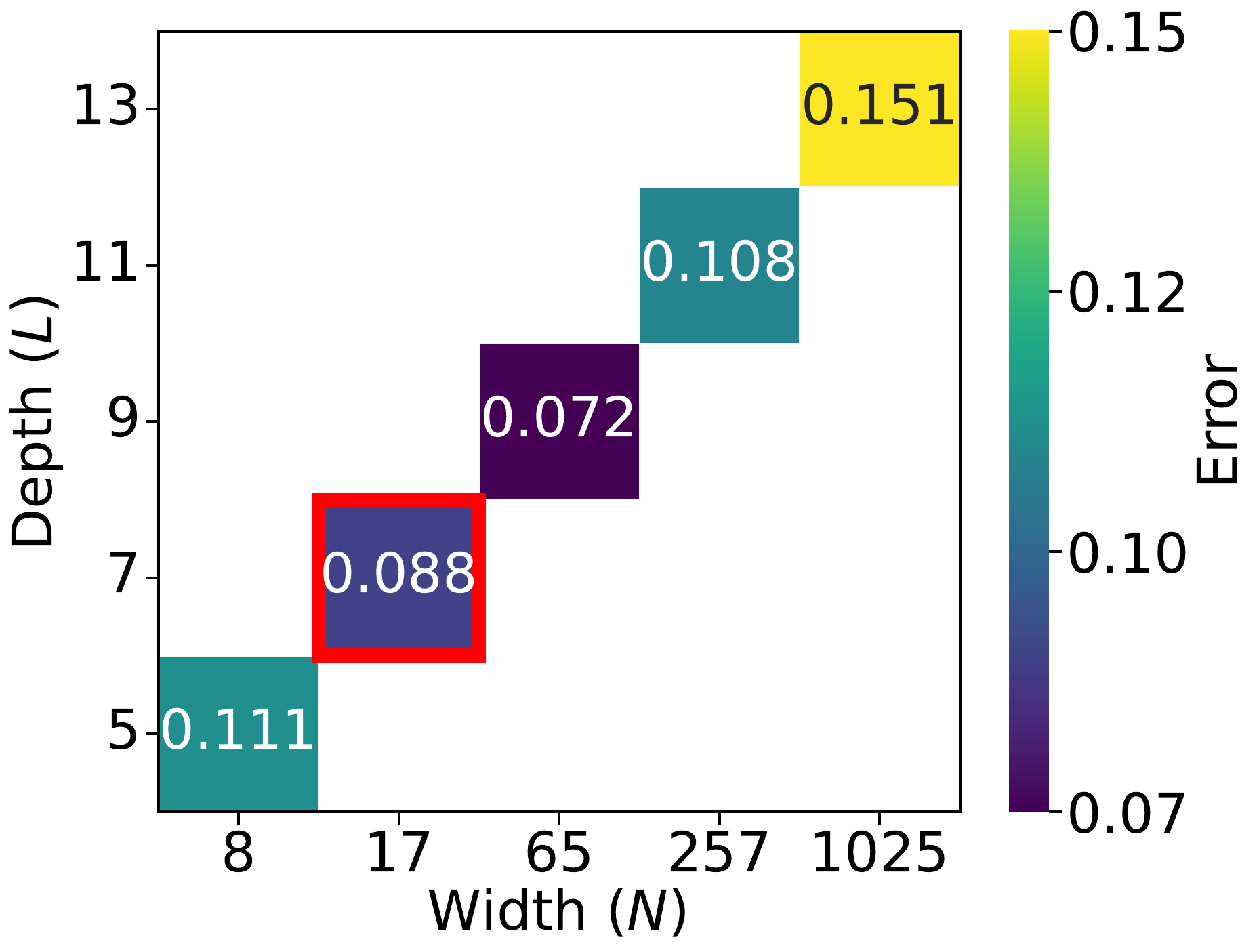}
  \caption{$\delta=0.0005$}
  \label{fig:Ex2_A1_results_delta_0_0005}
\end{subfigure}
\begin{subfigure}[t]{0.36\linewidth}
  \centering
  \includegraphics[width=\linewidth]{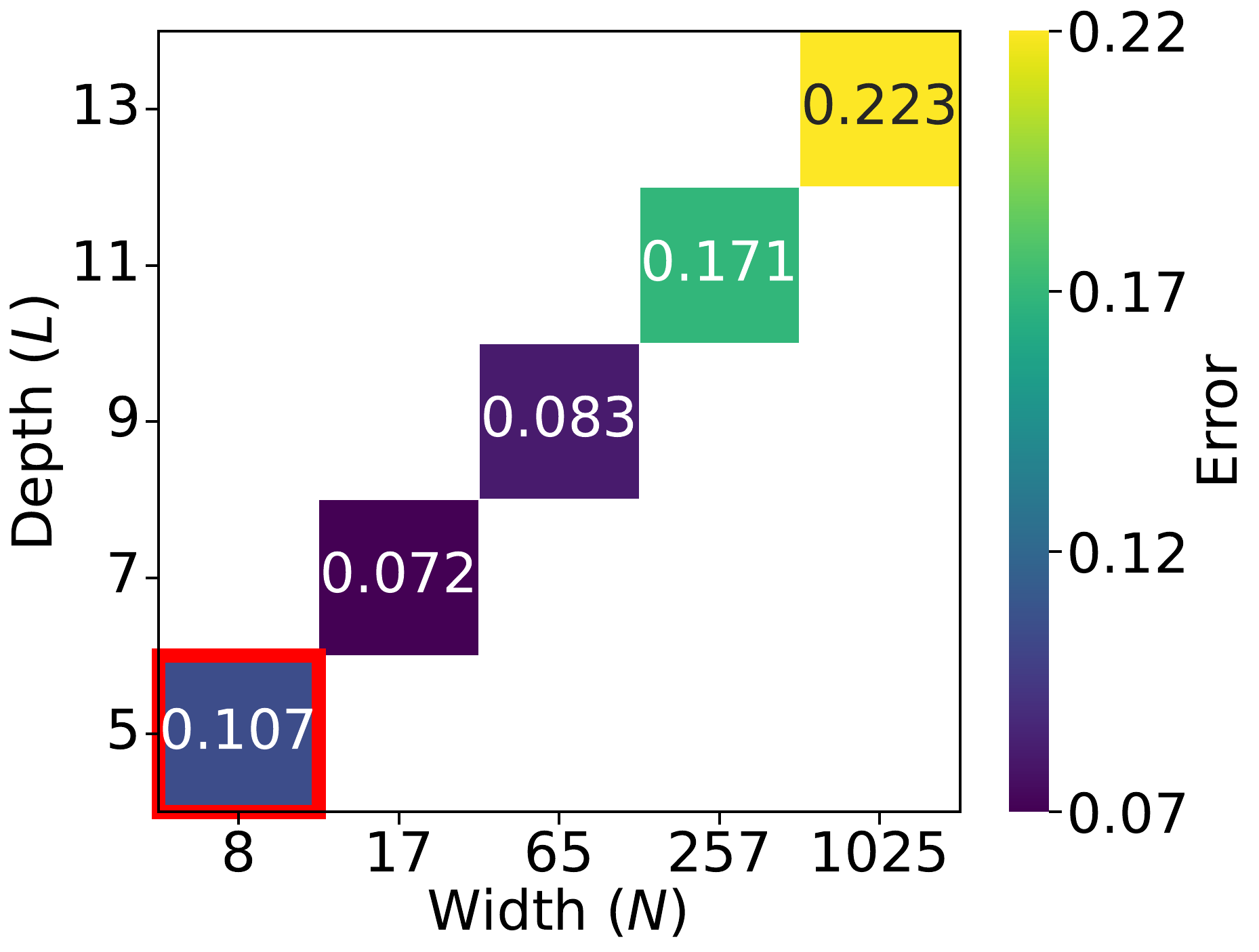}
  \caption{$\delta=0.002$}
  \label{fig:Ex2_A1_results_delta_0_002}
\end{subfigure}
\begin{subfigure}[t]{0.36\linewidth}
  \centering
  \includegraphics[width=\linewidth]{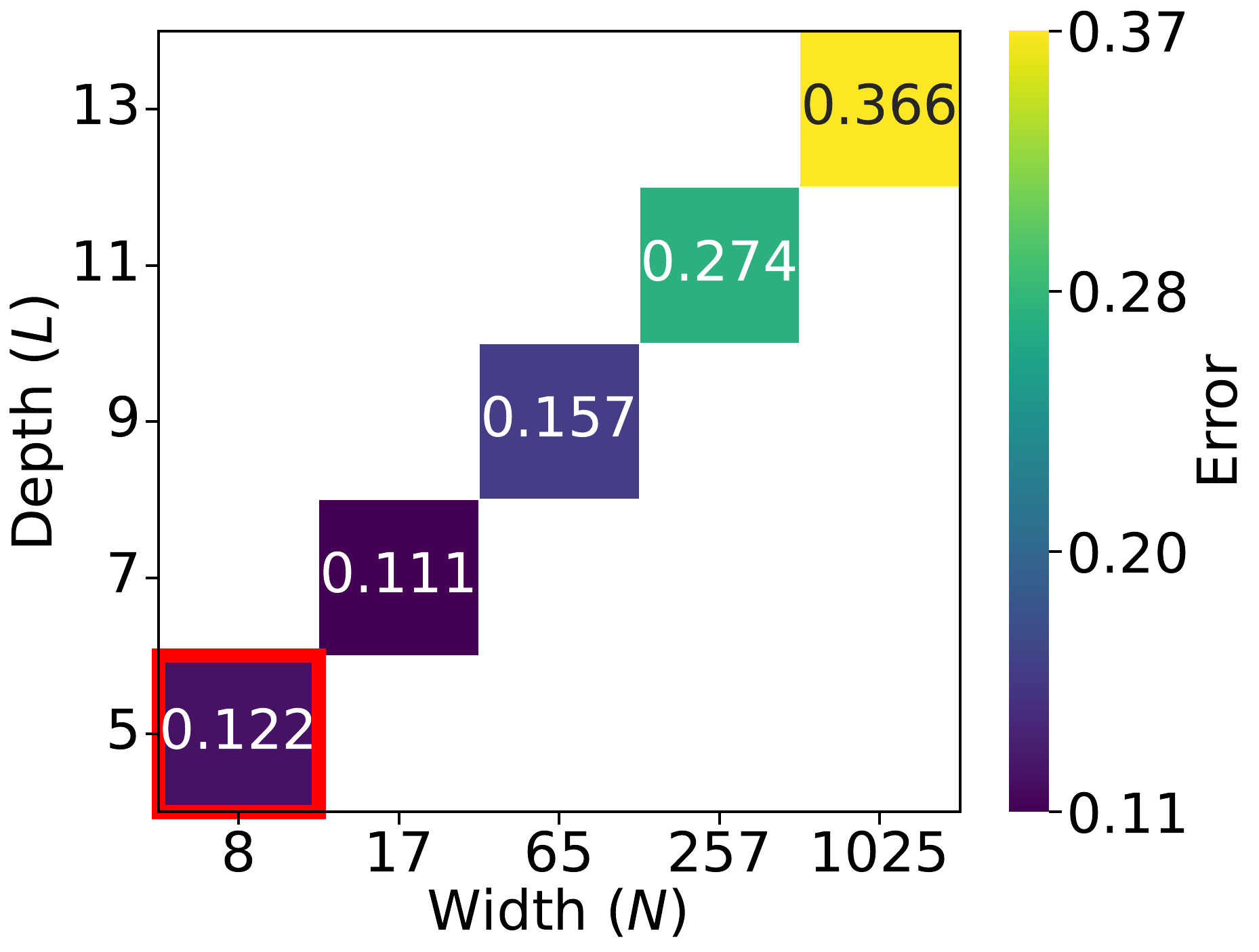}
  \caption{$\delta=0.003$}
  \label{fig:Ex2_A1_results_delta_0_003}
\end{subfigure}

\vspace{0.6ex}

\begin{subfigure}[t]{0.36\linewidth}
  \centering
  \includegraphics[width=\linewidth]{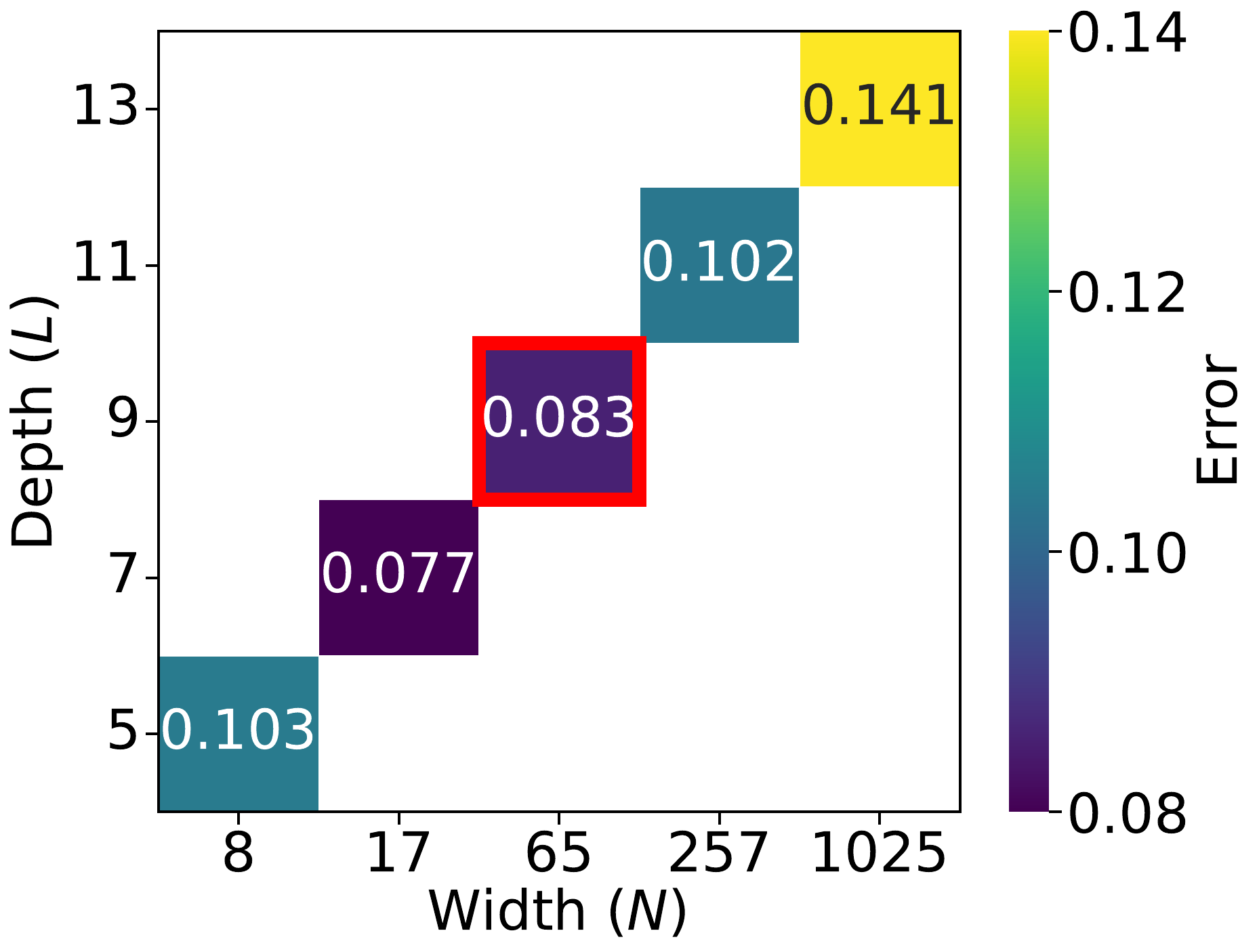}
  \caption{$\delta=0.00002$}
  \label{fig:Ex2_A2_results_delta_0_00002}
\end{subfigure}
\begin{subfigure}[t]{0.36\linewidth}
  \centering
  \includegraphics[width=\linewidth]{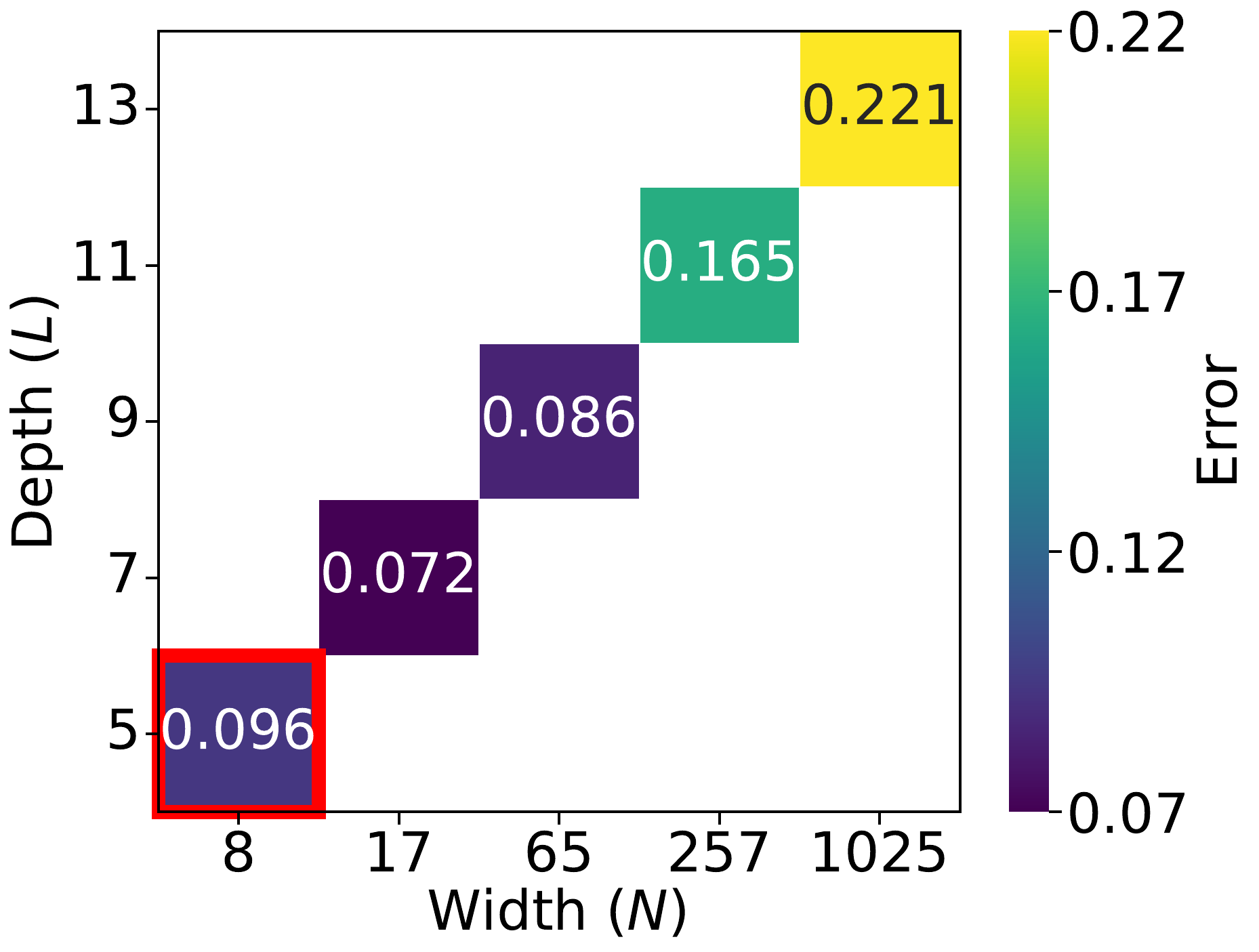}
  \caption{$\delta=0.0005$}
  \label{fig:Ex2_A2_results_delta_0_0005}
\end{subfigure}
\begin{subfigure}[t]{0.36\linewidth}
  \centering
  \includegraphics[width=\linewidth]{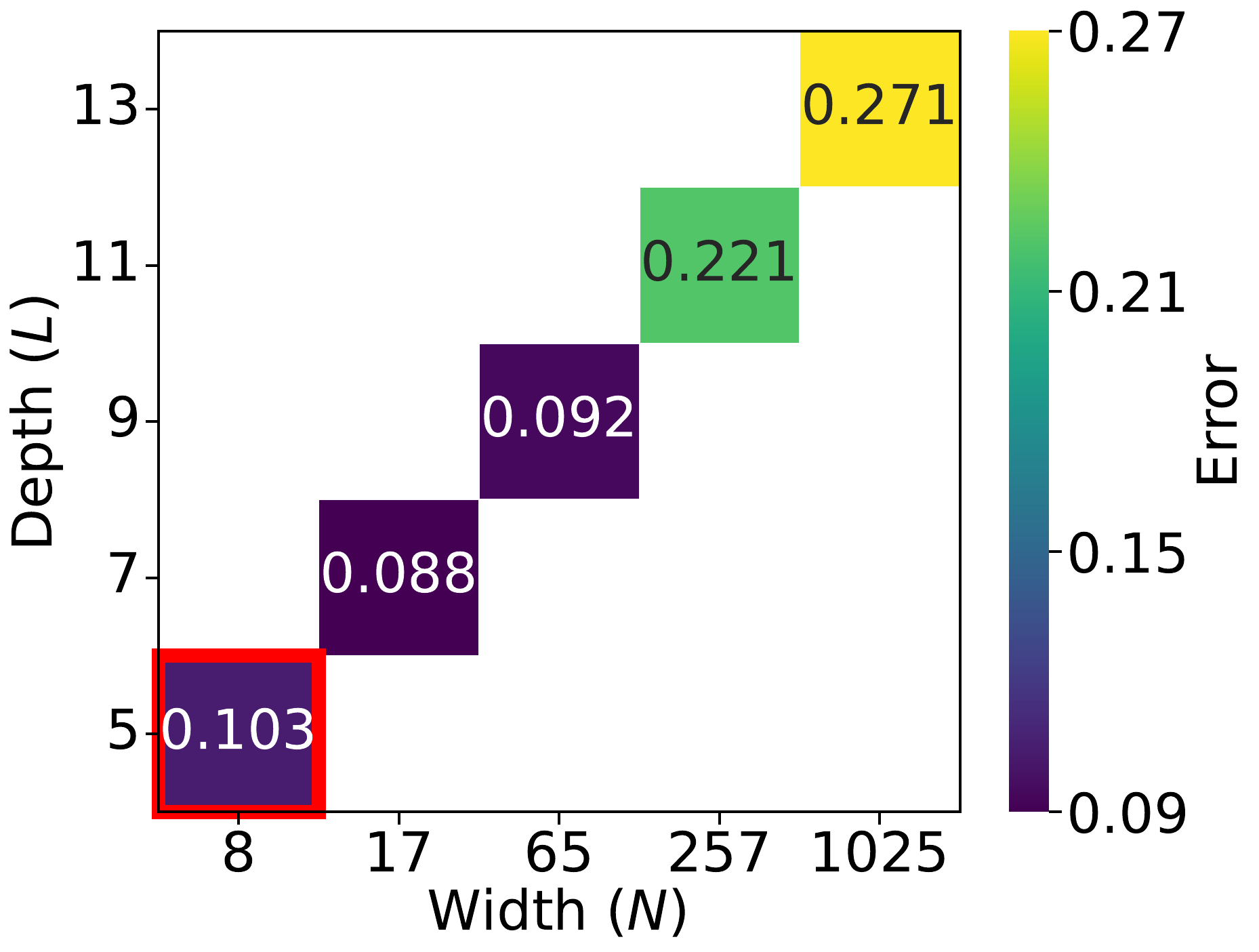}
  \caption{$\delta=0.002$}
  \label{fig:Ex2_A2_results_delta_0_002}
\end{subfigure}
\begin{subfigure}[t]{0.36\linewidth}
  \centering
  \includegraphics[width=\linewidth]{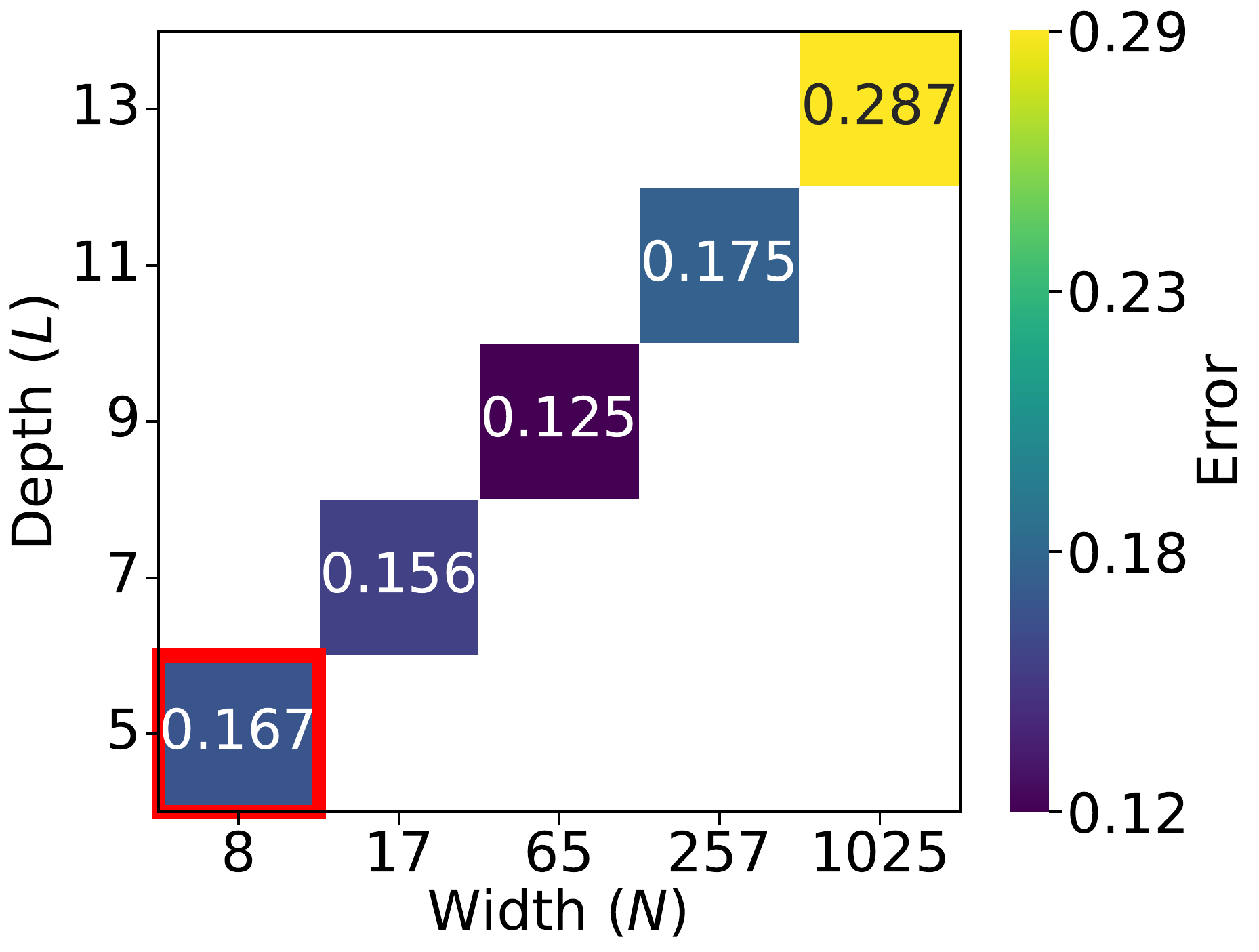}
  \caption{$\delta=0.003$}
  \label{fig:Ex2_A2_results_delta_0_003}
\end{subfigure}

\caption{Example \ref{ex:heat}: Relative \(L^2\) error \(e_f\) of the predicted \(f\) obtained by Algorithm~\ref{alg:main_algorithm} (first row) and Algorithm~\ref{alg:two_stage} (second row) for different network architectures (\(L\) layers, \(N\) neurons). Each column corresponds to a different noise level \(\delta\) in the data \(g\). The red box marks the first architecture for which the stopping criterion is satisfied.}
\label{fig:heat_delta_grid}
\end{figure}

In the numerical implementation, the full DN map is approximated by finitely many boundary excitations and the corresponding Neumann measurements. We apply voltage patterns to one side of \(\partial\Omega\) at a time and ground the remaining three sides. Writing \(\Gamma_{\mathrm{bottom}}=\{x_2=0\}\), \(\Gamma_{\mathrm{top}}=\{x_2=1\}\), \(\Gamma_{\mathrm{left}}=\{x_1=0\}\), and \(\Gamma_{\mathrm{right}}=\{x_1=1\}\), for each side \(\gamma\in\{\Gamma_{\mathrm{bottom}},\Gamma_{\mathrm{top}},\Gamma_{\mathrm{left}},\Gamma_{\mathrm{right}}\}\) and each frequency index \(\omega\in\{1,2\}\), we define
\begin{equation}
    h^{(\omega,\gamma)}(\boldsymbol{x})=
    \begin{cases}
        \sin(\omega\pi x_1), & \text{if } \boldsymbol{x}\in \gamma \text{ and } \gamma\in\{\Gamma_{\mathrm{bottom}},\Gamma_{\mathrm{top}}\},\\
        \sin(\omega\pi x_2), & \text{if } \boldsymbol{x}\in \gamma \text{ and } \gamma\in\{\Gamma_{\mathrm{left}},\Gamma_{\mathrm{right}}\},\\
        0, & \text{otherwise on } \partial\Omega\setminus \gamma.
    \end{cases}
\end{equation}
This choice preserves continuity at the corner points and yields \(4\times 2=8\) boundary measurements.
\end{example}
\begin{figure}[H]
    \centering
    \includegraphics[width=0.9\linewidth]{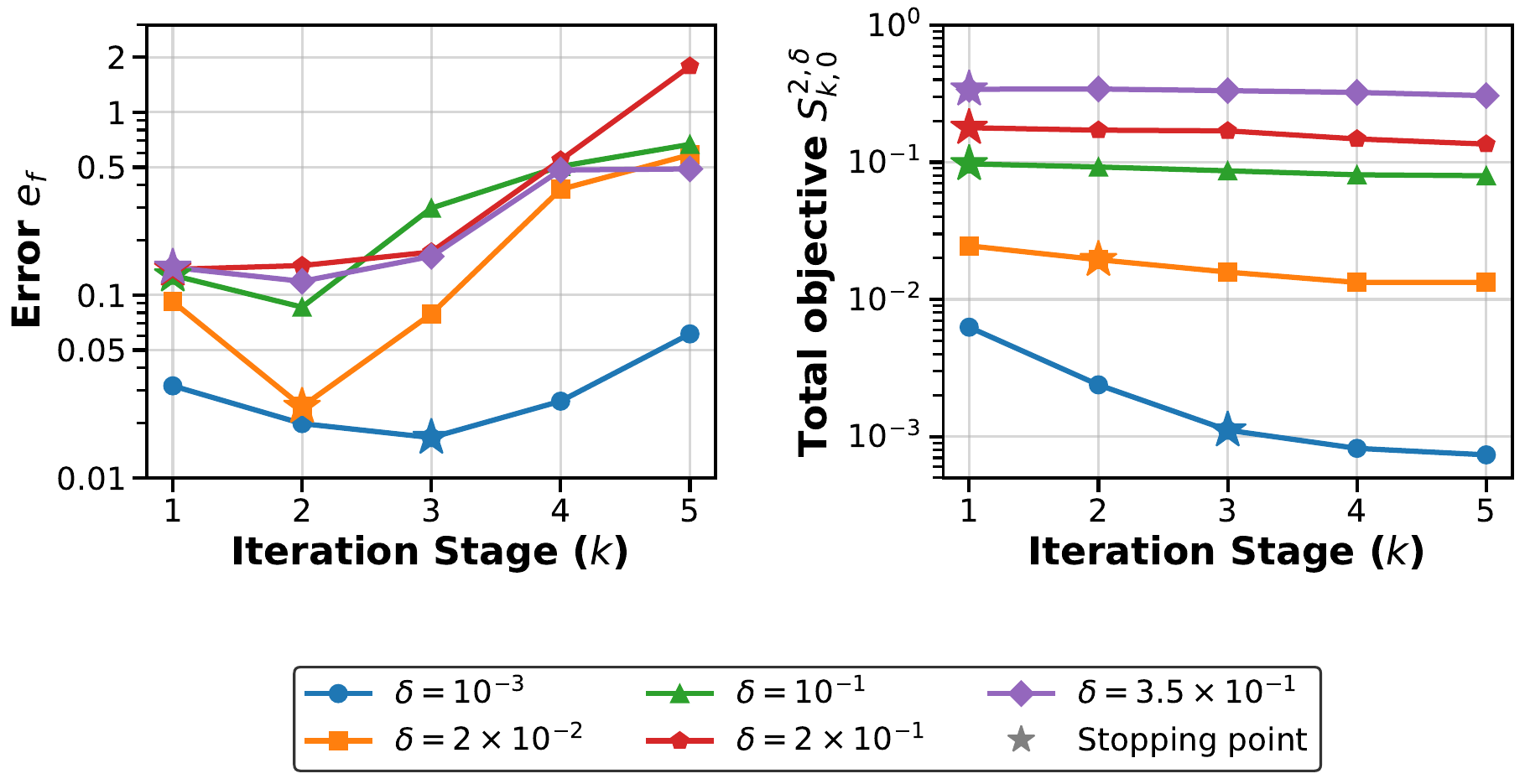}
    \caption{Example \ref{ex:EIT}: Evolution of the relative test error $e_f$ (left) and the stopping criterion value (right) for Algorithm~\ref{alg:two_stage} ($S_{k,0}^{2,\delta}$) under different noise levels $\delta$. Stars indicate the first iteration at which the stopping criterion is satisfied.}
    \label{fig:EIT_error_residual}
\end{figure}
Algorithm~\ref{alg:main_algorithm} is not considered in this example, since the \(L^p\)-based approximation framework in Theorem~\ref{thm:explicit_weights} does not provide a natural regularization space continuously embedded into \(L^\infty(\Omega)\), which is needed to preserve the positivity condition \(f\ge c>0\).

\paragraph{Verification for Algorithm~\ref{alg:two_stage} (unknown bound)}
We set \(\mathcal X_1=W^{1,3}(\Omega)\), \(\mathcal X_0=L^3(\Omega)\), \(\mathcal X_A=L^\infty(\Omega)\), and \(\mathcal Y=\mathcal L(H^{1/2}(\partial\Omega),H^{-1/2}(\partial\Omega))\). By Theorem~\ref{thm:sobolev_approx} with \(s=2\) and \(p=3\), Assumption~\ref{ass:generic_approx} holds for the present example. Assumption~\ref{ass:X1_pivot_general} is satisfied by choosing \(\mathcal X_0=L^3(\Omega)\). Since \(W^{1,3}(\Omega)\hookrightarrow L^\infty(\Omega)\) continuously, conductivities sufficiently close to \(f^\dagger\) in \(W^{1,3}(\Omega)\) remain uniformly positive, because \(f^\dagger\ge 0.1\) on \(\Omega\). Thus the DN map is well defined on a sufficiently small admissible neighborhood of \(f^\dagger\), which verifies Assumption~\ref{ass:forward_full}(i). Moreover, the compact embedding \(W^{1,3}(\Omega)\hookrightarrow\hookrightarrow L^\infty(\Omega)\) implies that weak convergence in \(W^{1,3}(\Omega)\) yields strong convergence in \(L^\infty(\Omega)\). Together with the local Lipschitz continuity of the DN map with respect to the \(L^\infty\)-norm on uniformly elliptic conductivities, this gives Assumption~\ref{ass:forward_full}(ii). Finally, let \(C_A>0\) denote the local Lipschitz constant of the forward map in \(L^\infty(\Omega)\), and let \(C_{\mathrm{emb}}\) denote the embedding constant of \(W^{1,3}(\Omega)\hookrightarrow L^\infty(\Omega)\). Then, for \(f\) sufficiently close to \(f^\dagger\), we have
\[
\|A(f)-A(f^\dagger)\|_{\mathcal Y}
\le C_A\|f-f^\dagger\|_{L^\infty(\Omega)}
\le C_A C_{\mathrm{emb}}\|f-f^\dagger\|_{W^{1,3}(\Omega)}.
\]
Hence Assumption~\ref{ass:forward_full}(iv) holds with \(\theta=1\) and \(L_A=C_A C_{\mathrm{emb}}\). The uniqueness condition in Assumption~\ref{ass:forward_full}(iii) follows from the classical uniqueness result for the Calder\'on problem in two dimensions.

\begin{figure}[htbp]
  \centering
  \includegraphics[width=\linewidth]{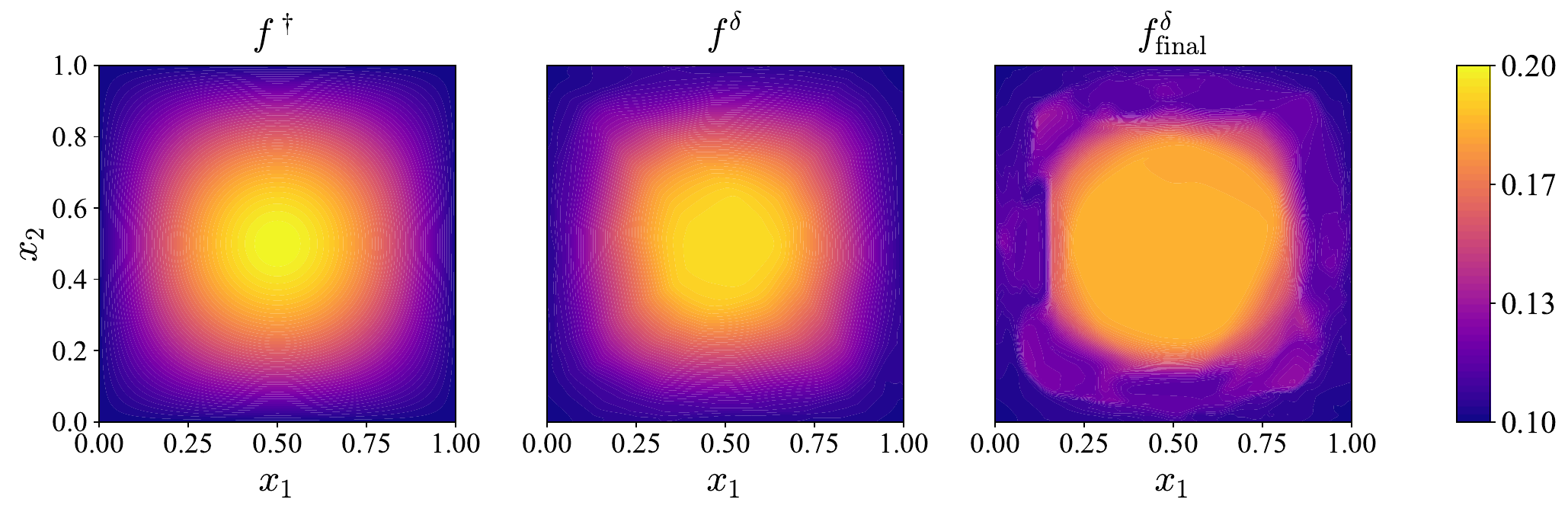}
  \caption{Example \ref{ex:EIT}: True solution \(f^\dagger\) and the corresponding reconstructions \(f^\delta\) and \(f_{\text{final}}^\delta\) obtained by Algorithm~\ref{alg:main_algorithm} (first row) and Algorithm~\ref{alg:two_stage} (second row) for \(\delta = 0.001\). Here, \(f_{\text{final}}^\delta\) denotes the reconstruction at the final expansion stage (\(k=5\)).}
  \label{fig:EIT_reconstructions}
\end{figure}

We evaluate Algorithm~\ref{alg:main_algorithm} and Algorithm~\ref{alg:two_stage} under various noise levels \(\delta\). Although the absolute noise bounds \(\delta\) may appear small, the corresponding relative noise levels \(\|g^\delta-g\|_{\mathcal Y}/\|g\|_{\mathcal Y}\) can still be substantial relative to the norm of the exact data \(g\). The chosen \(\delta\)-values therefore span a broad range of noise regimes. In Example~\ref{ex:conv}, \(\delta \in [10^{-4},\,3\times10^{-2}]\) gives relative noise levels from \(0.22\%\) to \(50\%\). The corresponding ranges are \(0.14\%\) to \(50\%\) in Example~\ref{ex:heat}, with \(\delta \in [2\times10^{-5},\,6\times10^{-3}]\), and \(0.14\%\) to \(50\%\) in Example~\ref{ex:EIT}, with \(\delta \in [10^{-3},\,0.35]\).

To validate the theoretical guarantees in Theorem~\ref{maintheorem}, we examine the algorithms from four perspectives. Although the stopping criterion may already be satisfied at an earlier stage, in all experiments we still run the full five-stage expansion process so as to show the complete evolution of the algorithms and to provide a more comprehensive demonstration of their effectiveness. Specifically, we report: (i) the trajectories of errors and stopping quantities to verify finite termination; (ii) architecture heatmaps to reveal noise-dependent structural adaptation; (iii) visual reconstructions to demonstrate effective noise suppression; and (iv) empirical convergence rates as \(\delta \to 0\).

\begin{figure}[H]
\centering

\begin{subfigure}[t]{0.36\linewidth}
  \centering
  \includegraphics[width=\linewidth]{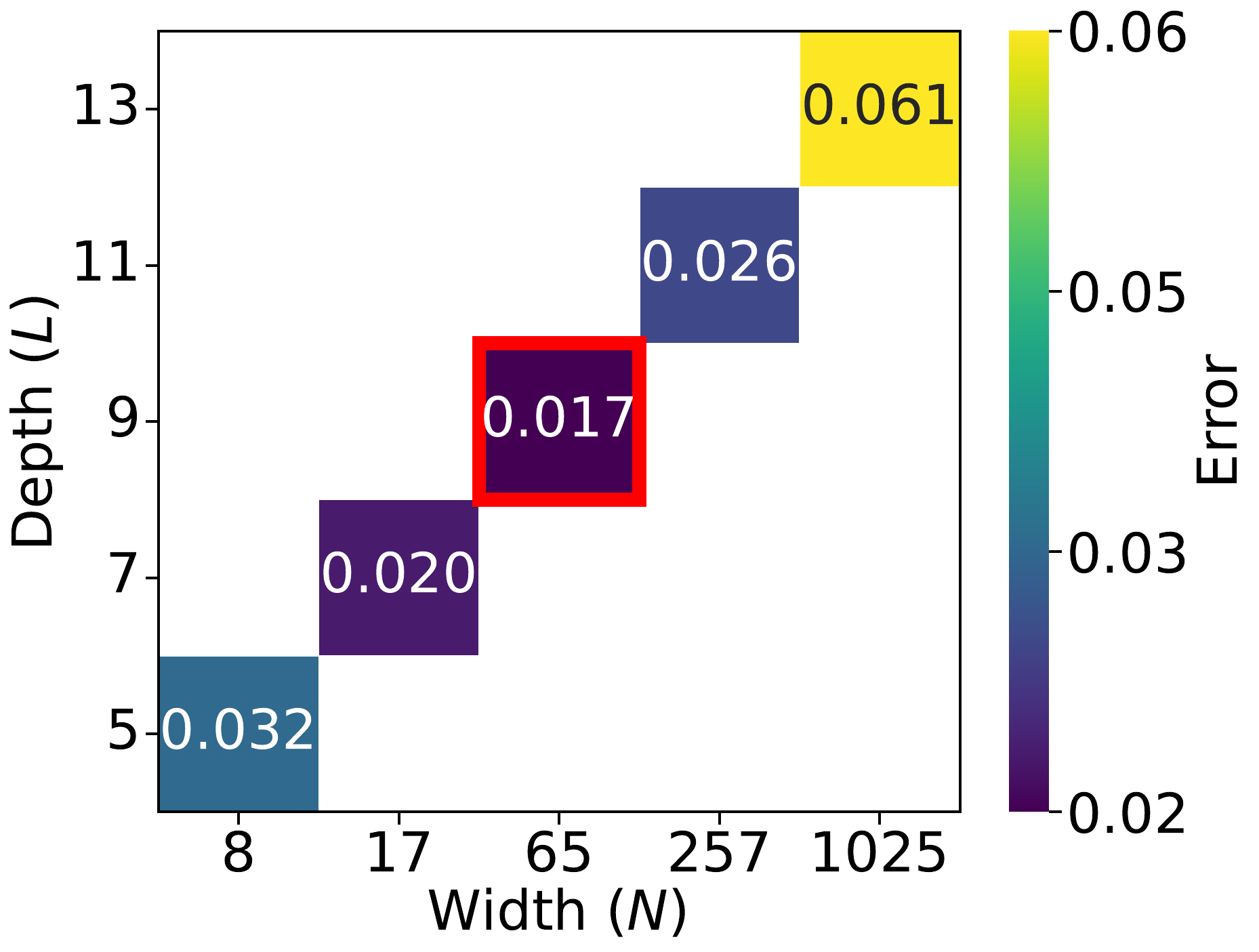}
  \caption{$\delta=0.001$}
  \label{fig:EIT_results_delta_0_001}
\end{subfigure}
\begin{subfigure}[t]{0.36\linewidth}
  \centering
  \includegraphics[width=\linewidth]{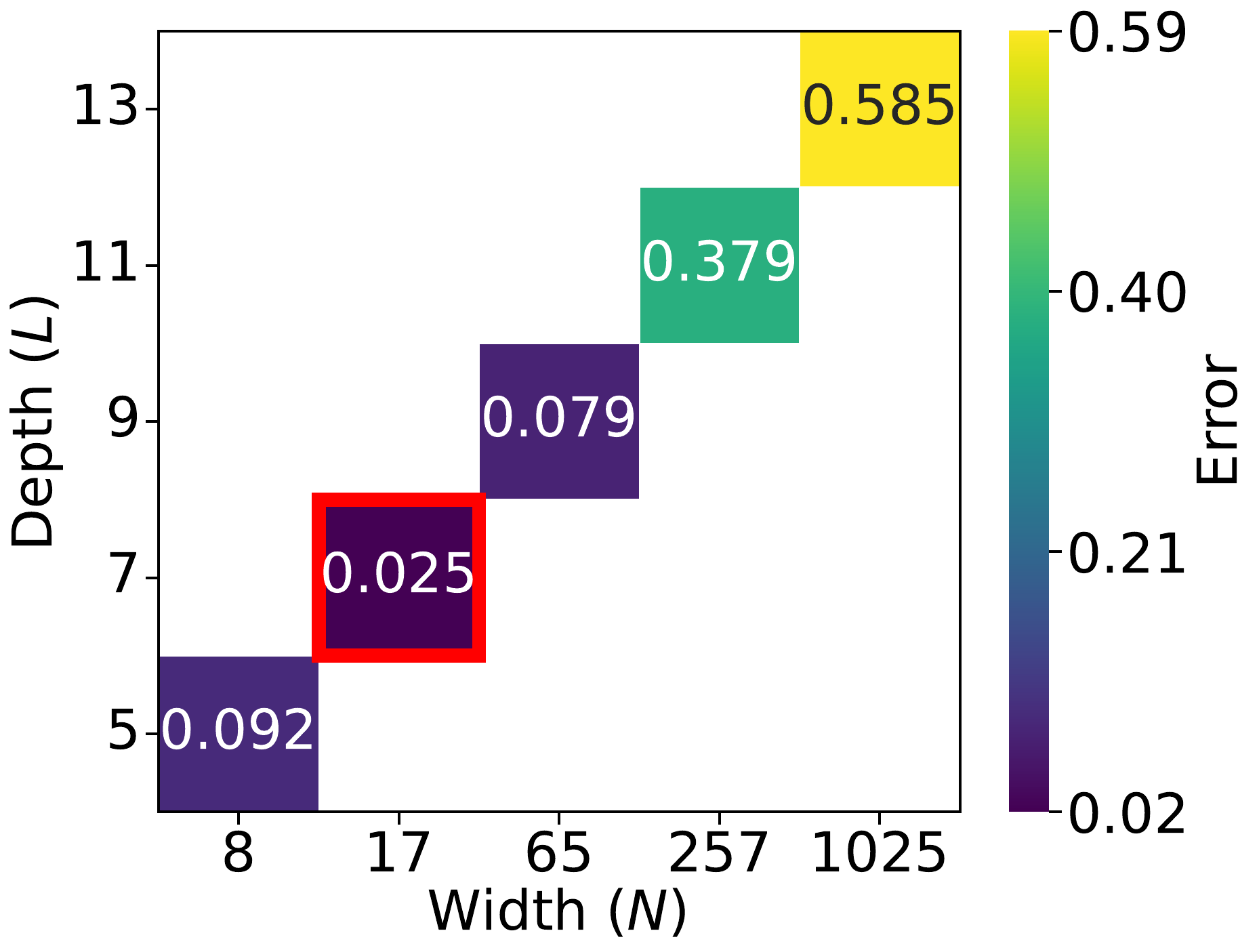}
  \caption{$\delta=0.02$}
  \label{fig:EIT_results_delta_0_02}
\end{subfigure}
\begin{subfigure}[t]{0.36\linewidth}
  \centering
  \includegraphics[width=\linewidth]{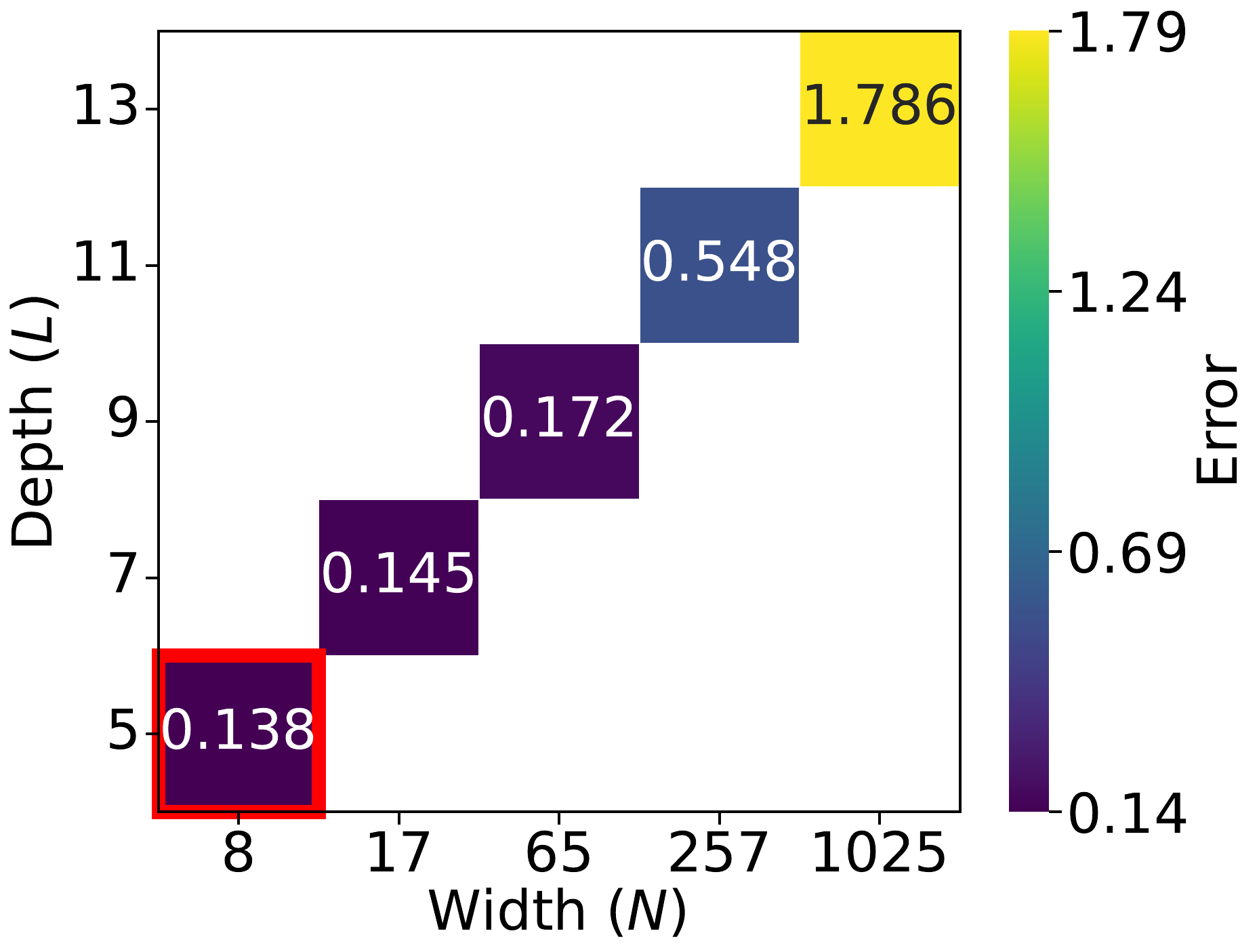}
  \caption{$\delta=0.2$}
  \label{fig:EIT_results_delta_0_2}
\end{subfigure}
\begin{subfigure}[t]{0.36\linewidth}
  \centering
  \includegraphics[width=\linewidth]{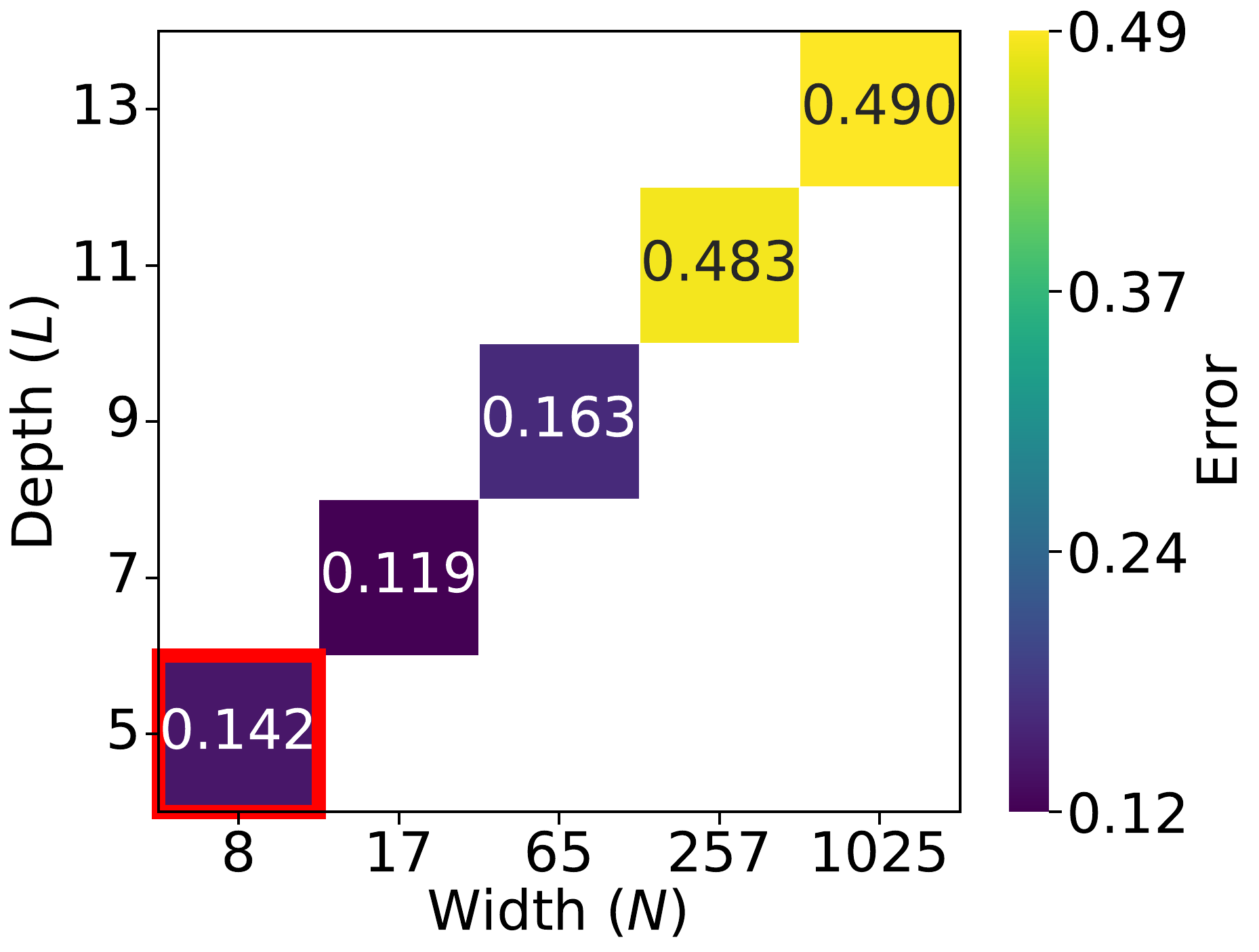}
  \caption{$\delta=0.35$}
  \label{fig:EIT_results_delta_0_35}
\end{subfigure}

\caption{Example \ref{ex:EIT}: Relative \(L^2\) error \(e_f\) of the predicted \(f\) obtained by Algorithm~\ref{alg:two_stage} for different network architectures (\(L\) layers, \(N\) neurons). Each column corresponds to a different noise level \(\delta\) in the data \(g\). The red box marks the first architecture for which the stopping criterion is satisfied.}
\label{fig:EIT_delta_grid}
\end{figure}

We first examine the finite termination property and the noise-dependent architecture complexity predicted by Theorem~\ref{maintheorem}(b). The trajectory plots (Figs.~\ref{fig:deconv_error_residuals}, \ref{fig:heat_error_residual}, \ref{fig:EIT_error_residual}) display the relative test error \(e_f\) together with the discrepancy-based stopping quantities \(S_k^{1,\delta}\) and \(S_{k,0}^{2,\delta}\). In all three examples, these quantities decrease overall and cross the discrepancy threshold after finitely many steps, confirming the stopping behavior of the proposed schemes. At the same time, the error curves exhibit the classical semi-convergence phenomenon of ill-posed inverse problems: for low to moderate noise levels, the error typically decreases at first and then increases once the architecture becomes overly rich, while for heavily contaminated data it may increase from the first iteration stage. This behavior highlights the necessity of early stopping. The same overfitting phenomenon is visible in both the reconstructions (Figs.~\ref{fig:deconv_reconstructions}, \ref{fig:heat_reconstructions}, \ref{fig:EIT_reconstructions}) and the architecture heatmaps (Figs.~\ref{fig:deconv_delta_grid}, \ref{fig:heat_delta_grid}, \ref{fig:EIT_delta_grid}). Reconstructions selected by the discrepancy principle preserve the main structural features, whereas over-expanded models develop pronounced spurious oscillations. In the heatmaps, the upper-right region is dominated by yellow, indicating relatively large reconstruction errors associated with excessively wide and deep architectures. By contrast, the blue regions correspond to more stable and accurate reconstructions.

\begin{figure}[H]
  \centering
  \includegraphics[width=0.68\linewidth]{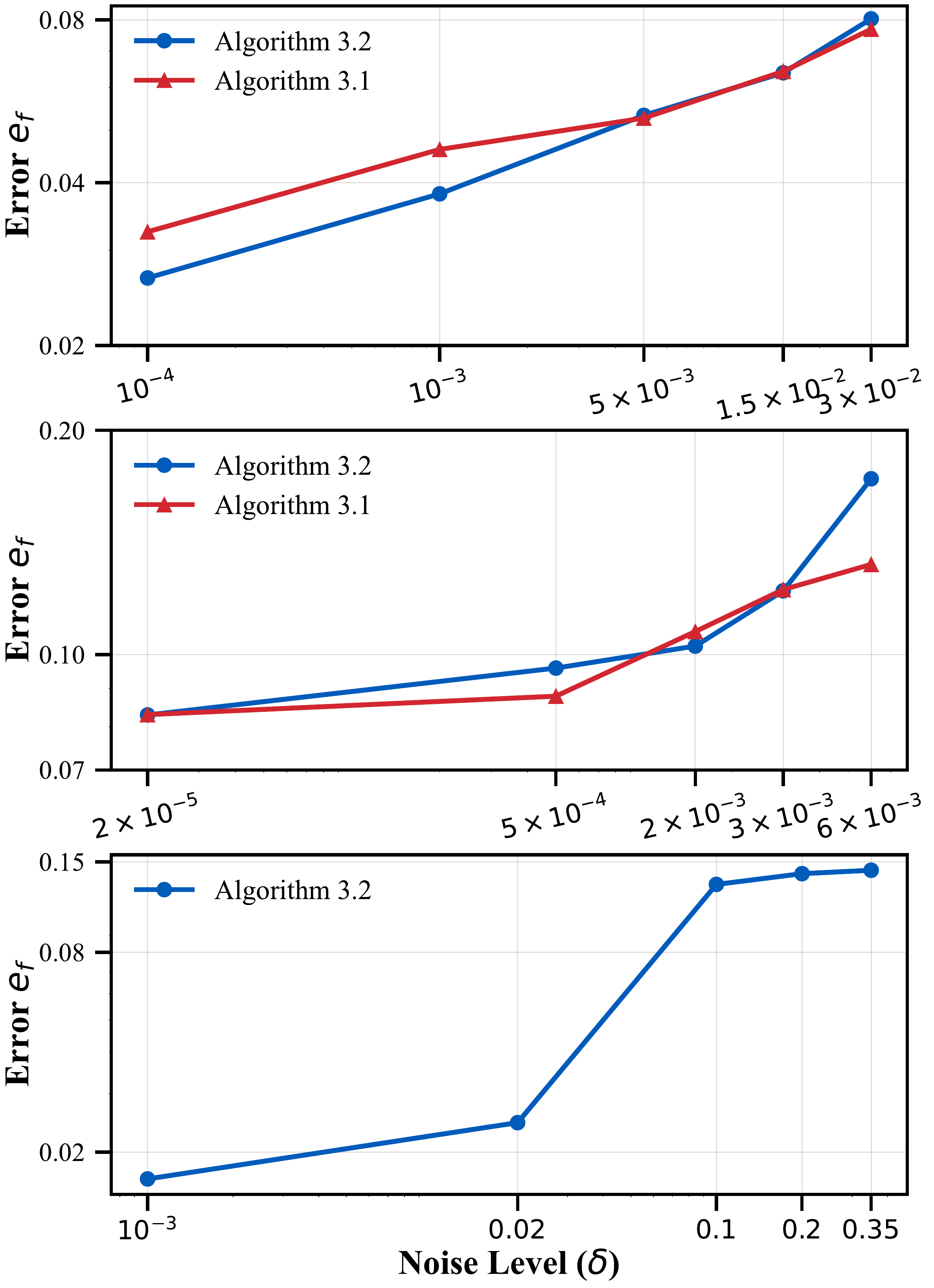}
  \caption{Relative \(L^2\) error \(e_f\) versus the noise level \(\delta\) for Examples~\ref{ex:conv} (top),~\ref{ex:heat} (middle), and~\ref{ex:EIT} (bottom). The error \(e_f\) is evaluated at the first iteration satisfying the stopping criterion.}
  \label{fig:All_Examples_Convergence}
\end{figure}

The selected architecture depends clearly on the noise level: smaller noise permits richer network architectures, while larger noise leads to earlier stopping at smaller widths and depths. This pattern is visible in all three heatmaps, namely Figs.~\ref{fig:deconv_delta_grid}, \ref{fig:heat_delta_grid}, and \ref{fig:EIT_delta_grid}, where the red boxes marking the stopping configurations selected by the discrepancy principle shift systematically toward deeper and wider networks as \(\delta\) decreases. As a concrete illustration, consider Example~\ref{ex:conv}; see Fig.~\ref{fig:deconv_delta_grid}. For the relatively large noise level \(\delta=0.005\), both Algorithm~\ref{alg:main_algorithm} and Algorithm~\ref{alg:two_stage} stop already at \(k=1\) with \((N,L)=(8,5)\). For the much smaller noise level \(\delta=0.0001\), Algorithm~\ref{alg:main_algorithm} expands to \(k=2\) with \((N,L)=(17,7)\), whereas Algorithm~\ref{alg:two_stage} expands to \(k=3\) with \((N,L)=(65,9)\). This empirical trend is consistent with the asymptotic complexity laws in \eqref{eq:scaling_sobolev} and \eqref{eq:scaling_hoelder}. Since all examples are posed in spatial dimension \(d=2\) with \(\theta=1\), and the exact solutions satisfy the assumed regularity (\(\alpha=1\) in the H\"older setting and \(s=2\) in the Sobolev setting), the theoretical bounds reduce to \(L(\delta)=\mathcal O(\log(1/\delta))\) and \(N(\delta)=\mathcal O(\delta^{-2})\) for Algorithm~\ref{alg:main_algorithm}, and to \(N(\delta)L(\delta)=\mathcal O(\delta^{-1}|\log\delta|^2)\) for Algorithm~\ref{alg:two_stage}. These bounds mean that, as the noise level decreases, the admissible stopping architecture may grow at the corresponding rates. The experiments support this prediction: larger noise levels confine the admissible capacity to relatively small architectures, whereas smaller noise levels allow richer models before overfitting emerges.

Finally, we examine the asymptotic convergence behavior as the noise level tends to zero, in accordance with Theorem~\ref{maintheorem}(c). As shown in Fig.~\ref{fig:All_Examples_Convergence}, the relative \(L^2\) error \(e_f\), evaluated at the first iteration satisfying the stopping criterion, is plotted against \(\delta\) on a log-log scale. In all three examples, the plotted relative \(L^2\) error decreases clearly as the noise level \(\delta\) tends to zero, indicating an empirical convergence trend with respect to \(\delta\). These quantitative results support the theoretical conclusion that the regularized reconstructions converge to the exact solution as \(\delta\to0\), while also demonstrating that the proposed expanding framework performs consistently across diverse inverse problems.

\section{Conclusions}
\label{sec:conclusion}

In this paper, we extended architecture-based regularization from shallow to deep neural networks for ill-posed inverse problems. Previous studies typically used the width of a single hidden layer as the sole regularization parameter. Here, we advanced this idea by developing an expanding framework in which both network width and depth increase adaptively. Under the approximation property for deep neural networks, we proposed two discrepancy-principle-based regularization algorithms corresponding to the cases where an explicit \emph{a priori} bound on the parameter radius is available or unavailable. For both methods, we established the existence of regularized minimizers, finite termination of the expansion procedure, asymptotic convergence of the reconstructed solutions as the noise level vanishes, and explicit asymptotic bounds on the terminal network architecture.

By combining these general bounds with concrete approximation results, we further derived explicit scaling laws for the stopping architecture in the H\"older and Sobolev settings. Numerical experiments on several representative inverse problems supported the theoretical analysis from multiple perspectives, including finite termination in practice, noise-dependent structural adaptation, effective noise suppression, and empirical algebraic convergence as the noise level decreases. Overall, these results show that architecture expansion provides a mathematically justified and practically effective regularization strategy for deep neural networks in linear and non-linear inverse problems. Future work includes deriving sharper approximation and parameter-radius estimates for concrete network classes, as well as extending the present framework to more structured architectures, such as convolutional neural networks and transformers.

\begin{appendices}

\section{Proof of Lemma \ref{lem:NN_compact_Lq}}\label{appendix:lem_NN_compact}

\begin{proof}
Fix an admissible network architecture, namely a depth $K\le L$ and widths
$N_1,\ldots,N_K$ satisfying $\max_{1\le \ell\le K}N_\ell\le N$. For this fixed
architecture, the parameter collection $\Theta$ consists of finitely many scalar
weights and biases. The constraint $\|\Theta\|\le r$ therefore defines a closed
and bounded subset of a finite-dimensional parameter space. Since all norms on
finite-dimensional spaces are equivalent, this set is compact by the
Heine--Borel theorem.

For this fixed architecture, consider the realization map $T(\Theta)=\phi(\cdot;\Theta)$. Because $\sigma$ is continuous and the network is obtained by finitely many
compositions of affine maps and $\sigma$, the map
$(x,\Theta)\mapsto \phi(x;\Theta)$ is continuous on the compact set
$\overline{\Omega}\times\{\Theta:\|\Theta\|\le r\}$. Hence it is uniformly
continuous. Therefore, if $\Theta_n\to\Theta$, then
\[
\|\phi(\cdot;\Theta_n)-\phi(\cdot;\Theta)\|_{C(\overline{\Omega})}\to 0.
\]
Thus the realization map is continuous from the compact parameter set into
$C(\overline{\Omega})$. Its image is therefore compact in $C(\overline{\Omega})$. Since $K\le L$ and $1\le N_\ell\le N$, there are only finitely many admissible
architectures. Hence $\mathcal N(N,L,r)$ is a finite union of compact subsets of
$C(\overline{\Omega})$, and is therefore compact in $C(\overline{\Omega})$.

Finally, the embedding $C(\overline{\Omega})\hookrightarrow L^p(\Omega)$ is
continuous for every $p\in[1,\infty)$, since $\|f\|_{L^p(\Omega)}\le |\Omega|^{1/p}\|f\|_{C(\overline{\Omega})}$. Therefore $\mathcal N(N,L,r)$ is also compact in $L^p(\Omega)$. This completes
the proof.
\end{proof}




\end{appendices}

\section*{Acknowledgments}
This work was funded by the Shenzhen Sci-Tech Fund (No. RCJC20231211090030059), National Key Research and Development Program of China (No. 2025YFE0113400) and National Natural Science Foundation of China (No. W2421102).

\bibliography{references}


\begin{thebibliography}{62}
\ifx \bisbn   \undefined \def \bisbn  #1{ISBN #1}\fi
\ifx \binits  \undefined \def \binits#1{#1}\fi
\ifx \bauthor  \undefined \def \bauthor#1{#1}\fi
\ifx \batitle  \undefined \def \batitle#1{#1}\fi
\ifx \bjtitle  \undefined \def \bjtitle#1{#1}\fi
\ifx \bvolume  \undefined \def \bvolume#1{\textbf{#1}}\fi
\ifx \byear  \undefined \def \byear#1{#1}\fi
\ifx \bissue  \undefined \def \bissue#1{#1}\fi
\ifx \bfpage  \undefined \def \bfpage#1{#1}\fi
\ifx \blpage  \undefined \def \blpage #1{#1}\fi
\ifx \burl  \undefined \def \burl#1{\textsf{#1}}\fi
\ifx \doiurl  \undefined \def \doiurl#1{\url{https://doi.org/#1}}\fi
\ifx \betal  \undefined \def \betal{\textit{et al.}}\fi
\ifx \binstitute  \undefined \def \binstitute#1{#1}\fi
\ifx \binstitutionaled  \undefined \def \binstitutionaled#1{#1}\fi
\ifx \bctitle  \undefined \def \bctitle#1{#1}\fi
\ifx \beditor  \undefined \def \beditor#1{#1}\fi
\ifx \bpublisher  \undefined \def \bpublisher#1{#1}\fi
\ifx \bbtitle  \undefined \def \bbtitle#1{#1}\fi
\ifx \bedition  \undefined \def \bedition#1{#1}\fi
\ifx \bseriesno  \undefined \def \bseriesno#1{#1}\fi
\ifx \blocation  \undefined \def \blocation#1{#1}\fi
\ifx \bsertitle  \undefined \def \bsertitle#1{#1}\fi
\ifx \bsnm \undefined \def \bsnm#1{#1}\fi
\ifx \bsuffix \undefined \def \bsuffix#1{#1}\fi
\ifx \bparticle \undefined \def \bparticle#1{#1}\fi
\ifx \barticle \undefined \def \barticle#1{#1}\fi
\bibcommenthead
\ifx \bconfdate \undefined \def \bconfdate #1{#1}\fi
\ifx \botherref \undefined \def \botherref #1{#1}\fi
\ifx \url \undefined \def \url#1{\textsf{#1}}\fi
\ifx \bchapter \undefined \def \bchapter#1{#1}\fi
\ifx \bbook \undefined \def \bbook#1{#1}\fi
\ifx \bcomment \undefined \def \bcomment#1{#1}\fi
\ifx \oauthor \undefined \def \oauthor#1{#1}\fi
\ifx \citeauthoryear \undefined \def \citeauthoryear#1{#1}\fi
\ifx \endbibitem  \undefined \def \endbibitem {}\fi
\ifx \bconflocation  \undefined \def \bconflocation#1{#1}\fi
\ifx \arxivurl  \undefined \def \arxivurl#1{\textsf{#1}}\fi
\csname PreBibitemsHook\endcsname

\bibitem[\protect\citeauthoryear{Engl et~al.}{1996}]{engl1996regularization}
\begin{bbook}
\bauthor{\bsnm{Engl}, \binits{H.W.}},
\bauthor{\bsnm{Hanke}, \binits{M.}},
\bauthor{\bsnm{Neubauer}, \binits{A.}}:
\bbtitle{Regularization of Inverse Problems}.
\bpublisher{Springer},
\blocation{New York}
(\byear{1996})
\end{bbook}
\endbibitem

\bibitem[\protect\citeauthoryear{Isakov}{2006}]{isakov2006inverse}
\begin{bbook}
\bauthor{\bsnm{Isakov}, \binits{V.}}:
\bbtitle{Inverse Problems for Partial Differential Equations}.
\bpublisher{Springer},
\blocation{New York}
(\byear{2006})
\end{bbook}
\endbibitem

\bibitem[\protect\citeauthoryear{Schuster
  et~al.}{2012}]{schuster2012regularization}
\begin{bbook}
\bauthor{\bsnm{Schuster}, \binits{T.}},
\bauthor{\bsnm{Kaltenbacher}, \binits{B.}},
\bauthor{\bsnm{Hofmann}, \binits{B.}},
\bauthor{\bsnm{Kazimierski}, \binits{K.S.}}:
\bbtitle{Regularization Methods in Banach Spaces}.
\bpublisher{Walter de Gruyter},
\blocation{Berlin}
(\byear{2012})
\end{bbook}
\endbibitem

\bibitem[\protect\citeauthoryear{Tikhonov and
  Arsenin}{1977}]{tikhonov1977solutions}
\begin{bbook}
\bauthor{\bsnm{Tikhonov}, \binits{A.N.}},
\bauthor{\bsnm{Arsenin}, \binits{V.Y.}}:
\bbtitle{Solutions of Ill-Posed Problems}.
\bpublisher{Winston \/ Wiley},
\blocation{Washington, DC \/ New York}
(\byear{1977})
\end{bbook}
\endbibitem

\bibitem[\protect\citeauthoryear{Ito and Jin}{2014}]{ito2014inverse}
\begin{bbook}
\bauthor{\bsnm{Ito}, \binits{K.}},
\bauthor{\bsnm{Jin}, \binits{B.}}:
\bbtitle{Inverse Problems: Tikhonov Theory and Algorithms}.
\bpublisher{World Scientific},
\blocation{Singapore}
(\byear{2014})
\end{bbook}
\endbibitem

\bibitem[\protect\citeauthoryear{Gong et~al.}{2020}]{gong2020new}
\begin{barticle}
\bauthor{\bsnm{Gong}, \binits{R.}},
\bauthor{\bsnm{Hofmann}, \binits{B.}},
\bauthor{\bsnm{Zhang}, \binits{Y.}}:
\batitle{A new class of accelerated regularization methods, with application to
  bioluminescence tomography}.
\bjtitle{Inverse Probl.}
\bvolume{36}(\bissue{5}),
\bfpage{055013}
(\byear{2020})
\end{barticle}
\endbibitem

\bibitem[\protect\citeauthoryear{Jin and Kereta}{2023}]{jin2023convergence}
\begin{barticle}
\bauthor{\bsnm{Jin}, \binits{B.}},
\bauthor{\bsnm{Kereta}, \binits{{\v{Z}}.}}:
\batitle{On the convergence of stochastic gradient descent for linear inverse
  problems in banach spaces}.
\bjtitle{SIAM J. Imaging Sci.}
\bvolume{16}(\bissue{2}),
\bfpage{671}--\blpage{705}
(\byear{2023})
\end{barticle}
\endbibitem

\bibitem[\protect\citeauthoryear{Zhang and Chen}{2023}]{zhang2023stochastic}
\begin{barticle}
\bauthor{\bsnm{Zhang}, \binits{Y.}},
\bauthor{\bsnm{Chen}, \binits{C.}}:
\batitle{Stochastic asymptotical regularization for linear inverse problems}.
\bjtitle{Inverse Probl.}
\bvolume{39}(\bissue{1}),
\bfpage{015007}
(\byear{2023})
\end{barticle}
\endbibitem

\bibitem[\protect\citeauthoryear{Jin et~al.}{2025}]{jin2025regularizing}
\begin{bchapter}
\bauthor{\bsnm{Jin}, \binits{B.}},
\bauthor{\bsnm{Xia}, \binits{Y.}},
\bauthor{\bsnm{Zhou}, \binits{Z.}}:
\bctitle{On the regularizing property of stochastic iterative methods for
  solving inverse problems}.
In: \bbtitle{Handbook of Numerical Analysis}
vol. \bseriesno{26},
pp. \bfpage{211}--\blpage{272}.
\bpublisher{Elsevier},
\blocation{Amsterdam}
(\byear{2025})
\end{bchapter}
\endbibitem

\bibitem[\protect\citeauthoryear{Natterer}{2001}]{natterer2001mathematics}
\begin{bbook}
\bauthor{\bsnm{Natterer}, \binits{F.}}:
\bbtitle{The Mathematics of Computerized Tomography}.
\bpublisher{SIAM},
\blocation{Philadelphia}
(\byear{2001})
\end{bbook}
\endbibitem

\bibitem[\protect\citeauthoryear{McCann et~al.}{2017}]{mccann2017convolutional}
\begin{barticle}
\bauthor{\bsnm{McCann}, \binits{M.T.}},
\bauthor{\bsnm{Jin}, \binits{K.H.}},
\bauthor{\bsnm{Unser}, \binits{M.}}:
\batitle{Convolutional neural networks for inverse problems in imaging: {A}
  review}.
\bjtitle{IEEE Signal Process. Mag.}
\bvolume{34}(\bissue{6}),
\bfpage{85}--\blpage{95}
(\byear{2017})
\end{barticle}
\endbibitem

\bibitem[\protect\citeauthoryear{Arridge et~al.}{2019}]{arridge2019solving}
\begin{barticle}
\bauthor{\bsnm{Arridge}, \binits{S.}},
\bauthor{\bsnm{Maass}, \binits{P.}},
\bauthor{\bsnm{{\"O}ktem}, \binits{O.}},
\bauthor{\bsnm{Sch{\"o}nlieb}, \binits{C.-B.}}:
\batitle{Solving inverse problems using data-driven models}.
\bjtitle{Acta Numer.}
\bvolume{28},
\bfpage{1}--\blpage{174}
(\byear{2019})
\end{barticle}
\endbibitem

\bibitem[\protect\citeauthoryear{Ongie et~al.}{2020}]{ongie2020deep}
\begin{barticle}
\bauthor{\bsnm{Ongie}, \binits{G.}},
\bauthor{\bsnm{Jalal}, \binits{A.}},
\bauthor{\bsnm{Metzler}, \binits{C.A.}},
\bauthor{\bsnm{Baraniuk}, \binits{R.G.}},
\bauthor{\bsnm{Dimakis}, \binits{A.G.}},
\bauthor{\bsnm{Willett}, \binits{R.}}:
\batitle{Deep learning techniques for inverse problems in imaging}.
\bjtitle{IEEE J. Sel. Areas Inf. Theory}
\bvolume{1}(\bissue{1}),
\bfpage{39}--\blpage{56}
(\byear{2020})
\end{barticle}
\endbibitem

\bibitem[\protect\citeauthoryear{Scarlett
  et~al.}{2023}]{scarlett2023theoretical}
\begin{barticle}
\bauthor{\bsnm{Scarlett}, \binits{J.}},
\bauthor{\bsnm{Heckel}, \binits{R.}},
\bauthor{\bsnm{Rodrigues}, \binits{M.R.}},
\bauthor{\bsnm{Hand}, \binits{P.}},
\bauthor{\bsnm{Eldar}, \binits{Y.C.}}:
\batitle{Theoretical perspectives on deep learning methods in inverse
  problems}.
\bjtitle{IEEE J. Sel. Areas Inf. Theory}
\bvolume{3}(\bissue{3}),
\bfpage{433}--\blpage{453}
(\byear{2023})
\end{barticle}
\endbibitem

\bibitem[\protect\citeauthoryear{Jin et~al.}{2017}]{jin2017deep}
\begin{barticle}
\bauthor{\bsnm{Jin}, \binits{K.H.}},
\bauthor{\bsnm{McCann}, \binits{M.T.}},
\bauthor{\bsnm{Froustey}, \binits{E.}},
\bauthor{\bsnm{Unser}, \binits{M.}}:
\batitle{Deep convolutional neural network for inverse problems in imaging}.
\bjtitle{IEEE Trans. Image Process.}
\bvolume{26}(\bissue{9}),
\bfpage{4509}--\blpage{4522}
(\byear{2017})
\end{barticle}
\endbibitem

\bibitem[\protect\citeauthoryear{Ronneberger et~al.}{2015}]{ronneberger2015u}
\begin{bchapter}
\bauthor{\bsnm{Ronneberger}, \binits{O.}},
\bauthor{\bsnm{Fischer}, \binits{P.}},
\bauthor{\bsnm{Brox}, \binits{T.}}:
\bctitle{{U-Net: Convolutional Networks for Biomedical Image Segmentation}}.
In: \bbtitle{International Conference on Medical Image Computing and
  Computer-assisted Intervention},
pp. \bfpage{234}--\blpage{241}
(\byear{2015}).
\bcomment{Springer}
\end{bchapter}
\endbibitem

\bibitem[\protect\citeauthoryear{Adler and {\"O}ktem}{2018}]{adler2018learned}
\begin{barticle}
\bauthor{\bsnm{Adler}, \binits{J.}},
\bauthor{\bsnm{{\"O}ktem}, \binits{O.}}:
\batitle{Learned primal-dual reconstruction}.
\bjtitle{IEEE Trans. Med. Imaging}
\bvolume{37}(\bissue{6}),
\bfpage{1322}--\blpage{1332}
(\byear{2018})
\end{barticle}
\endbibitem

\bibitem[\protect\citeauthoryear{Hammernik
  et~al.}{2018}]{hammernik2018learning}
\begin{barticle}
\bauthor{\bsnm{Hammernik}, \binits{K.}},
\bauthor{\bsnm{Klatzer}, \binits{T.}},
\bauthor{\bsnm{Kobler}, \binits{E.}},
\bauthor{\bsnm{Recht}, \binits{M.P.}},
\bauthor{\bsnm{Sodickson}, \binits{D.K.}},
\bauthor{\bsnm{Pock}, \binits{T.}},
\bauthor{\bsnm{Knoll}, \binits{F.}}:
\batitle{Learning a variational network for reconstruction of accelerated {MRI}
  data}.
\bjtitle{Magn. Reson. Med.}
\bvolume{79}(\bissue{6}),
\bfpage{3055}--\blpage{3071}
(\byear{2018})
\end{barticle}
\endbibitem

\bibitem[\protect\citeauthoryear{Bora et~al.}{2017}]{bora2017compressed}
\begin{bchapter}
\bauthor{\bsnm{Bora}, \binits{A.}},
\bauthor{\bsnm{Jalal}, \binits{A.}},
\bauthor{\bsnm{Price}, \binits{E.}},
\bauthor{\bsnm{Dimakis}, \binits{A.G.}}:
\bctitle{Compressed sensing using generative models}.
In: \bbtitle{International Conference on Machine Learning},
pp. \bfpage{537}--\blpage{546}
(\byear{2017}).
\bcomment{PMLR}
\end{bchapter}
\endbibitem

\bibitem[\protect\citeauthoryear{Mardani et~al.}{2018}]{mardani2018deep}
\begin{barticle}
\bauthor{\bsnm{Mardani}, \binits{M.}},
\bauthor{\bsnm{Gong}, \binits{E.}},
\bauthor{\bsnm{Cheng}, \binits{J.Y.}},
\bauthor{\bsnm{Vasanawala}, \binits{S.S.}},
\bauthor{\bsnm{Zaharchuk}, \binits{G.}},
\bauthor{\bsnm{Xing}, \binits{L.}},
\bauthor{\bsnm{Pauly}, \binits{J.M.}}:
\batitle{Deep generative adversarial neural networks for compressive sensing
  {MRI}}.
\bjtitle{IEEE Trans. Med. Imaging}
\bvolume{38}(\bissue{1}),
\bfpage{167}--\blpage{179}
(\byear{2018})
\end{barticle}
\endbibitem

\bibitem[\protect\citeauthoryear{Antun et~al.}{2020}]{antun2020instabilities}
\begin{barticle}
\bauthor{\bsnm{Antun}, \binits{V.}},
\bauthor{\bsnm{Renna}, \binits{F.}},
\bauthor{\bsnm{Poon}, \binits{C.}},
\bauthor{\bsnm{Adcock}, \binits{B.}},
\bauthor{\bsnm{Hansen}, \binits{A.C.}}:
\batitle{On instabilities of deep learning in image reconstruction and the
  potential costs of {AI}}.
\bjtitle{Proc. Natl. Acad. Sci. U. S. A.}
\bvolume{117}(\bissue{48}),
\bfpage{30088}--\blpage{30095}
(\byear{2020})
\end{barticle}
\endbibitem

\bibitem[\protect\citeauthoryear{Li et~al.}{2020}]{li2020nett}
\begin{barticle}
\bauthor{\bsnm{Li}, \binits{H.}},
\bauthor{\bsnm{Schwab}, \binits{J.}},
\bauthor{\bsnm{Antholzer}, \binits{S.}},
\bauthor{\bsnm{Haltmeier}, \binits{M.}}:
\batitle{{NETT}: solving inverse problems with deep neural networks}.
\bjtitle{Inverse Probl.}
\bvolume{36}(\bissue{6}),
\bfpage{065005}
(\byear{2020})
\end{barticle}
\endbibitem

\bibitem[\protect\citeauthoryear{Lunz et~al.}{2018}]{lunz2018adversarial}
\begin{botherref}
\oauthor{\bsnm{Lunz}, \binits{S.}},
\oauthor{\bsnm{{\"O}ktem}, \binits{O.}},
\oauthor{\bsnm{Sch{\"o}nlieb}, \binits{C.-B.}}:
Adversarial regularizers in inverse problems.
Adv. Neural Inf. Process. Syst.
\textbf{31}
(2018)
\end{botherref}
\endbibitem

\bibitem[\protect\citeauthoryear{Lunz}{2022}]{lunz2022learned}
\begin{bchapter}
\bauthor{\bsnm{Lunz}, \binits{S.}}:
\bctitle{Learned regularizers for inverse problems}.
In: \bbtitle{Handbook of Mathematical Models and Algorithms in Computer Vision
  and Imaging: Mathematical Imaging and Vision},
pp. \bfpage{1}--\blpage{21}.
\bpublisher{Springer},
\blocation{New York}
(\byear{2022})
\end{bchapter}
\endbibitem

\bibitem[\protect\citeauthoryear{Kobler et~al.}{2020}]{kobler2020total}
\begin{bchapter}
\bauthor{\bsnm{Kobler}, \binits{E.}},
\bauthor{\bsnm{Effland}, \binits{A.}},
\bauthor{\bsnm{Kunisch}, \binits{K.}},
\bauthor{\bsnm{Pock}, \binits{T.}}:
\bctitle{Total deep variation for linear inverse problems}.
In: \bbtitle{Proceedings of the IEEE/CVF Conference on Computer Vision and
  Pattern Recognition},
pp. \bfpage{7549}--\blpage{7558}
(\byear{2020})
\end{bchapter}
\endbibitem

\bibitem[\protect\citeauthoryear{Jin et~al.}{2020}]{jin2020convergence}
\begin{barticle}
\bauthor{\bsnm{Jin}, \binits{B.}},
\bauthor{\bsnm{Zhou}, \binits{Z.}},
\bauthor{\bsnm{Zou}, \binits{J.}}:
\batitle{On the convergence of stochastic gradient descent for nonlinear
  ill-posed problems}.
\bjtitle{SIAM J. Optim.}
\bvolume{30}(\bissue{2}),
\bfpage{1421}--\blpage{1450}
(\byear{2020})
\end{barticle}
\endbibitem

\bibitem[\protect\citeauthoryear{Long et~al.}{2024}]{long2024accelerated}
\begin{barticle}
\bauthor{\bsnm{Long}, \binits{H.}},
\bauthor{\bsnm{Zhang}, \binits{Y.}},
\bauthor{\bsnm{Gao}, \binits{G.}}:
\batitle{An accelerated inexact newton regularization scheme with a learned
  feature-selection rule for non-linear inverse problems}.
\bjtitle{Inverse Probl.}
\bvolume{40}(\bissue{8}),
\bfpage{085011}
(\byear{2024})
\end{barticle}
\endbibitem

\bibitem[\protect\citeauthoryear{Raissi et~al.}{2019}]{raissi2019physics}
\begin{barticle}
\bauthor{\bsnm{Raissi}, \binits{M.}},
\bauthor{\bsnm{Perdikaris}, \binits{P.}},
\bauthor{\bsnm{Karniadakis}, \binits{G.E.}}:
\batitle{Physics-informed neural networks: A deep learning framework for
  solving forward and inverse problems involving nonlinear partial differential
  equations}.
\bjtitle{J. Comput. Phys.}
\bvolume{378},
\bfpage{686}--\blpage{707}
(\byear{2019})
\end{barticle}
\endbibitem

\bibitem[\protect\citeauthoryear{Karniadakis
  et~al.}{2021}]{karniadakis2021physics}
\begin{barticle}
\bauthor{\bsnm{Karniadakis}, \binits{G.E.}},
\bauthor{\bsnm{Kevrekidis}, \binits{I.G.}},
\bauthor{\bsnm{Lu}, \binits{L.}},
\bauthor{\bsnm{Perdikaris}, \binits{P.}},
\bauthor{\bsnm{Wang}, \binits{S.}},
\bauthor{\bsnm{Yang}, \binits{L.}}:
\batitle{Physics-informed machine learning}.
\bjtitle{Nat. Rev. Phys.}
\bvolume{3}(\bissue{6}),
\bfpage{422}--\blpage{440}
(\byear{2021})
\end{barticle}
\endbibitem

\bibitem[\protect\citeauthoryear{Ulyanov et~al.}{2018}]{ulyanov2018deep}
\begin{bchapter}
\bauthor{\bsnm{Ulyanov}, \binits{D.}},
\bauthor{\bsnm{Vedaldi}, \binits{A.}},
\bauthor{\bsnm{Lempitsky}, \binits{V.}}:
\bctitle{Deep image prior}.
In: \bbtitle{Proceedings of the IEEE Conference on Computer Vision and Pattern
  Recognition},
pp. \bfpage{9446}--\blpage{9454}
(\byear{2018})
\end{bchapter}
\endbibitem

\bibitem[\protect\citeauthoryear{Heckel et~al.}{2019}]{heckel2019deep}
\begin{bchapter}
\bauthor{\bsnm{Heckel}, \binits{R.}}, \betal:
\bctitle{Deep {Decoder}: Concise {Image Representations from Untrained
  Non-convolutional Networks}}.
In: \bbtitle{International Conference on Learning Representations}
(\byear{2019})
\end{bchapter}
\endbibitem

\bibitem[\protect\citeauthoryear{Dittmer
  et~al.}{2020}]{dittmer2020regularization}
\begin{barticle}
\bauthor{\bsnm{Dittmer}, \binits{S.}},
\bauthor{\bsnm{Kluth}, \binits{T.}},
\bauthor{\bsnm{Maass}, \binits{P.}},
\bauthor{\bsnm{Otero~Baguer}, \binits{D.}}:
\batitle{Regularization by architecture: A deep prior approach for inverse
  problems}.
\bjtitle{J. Math. Imaging Vis.}
\bvolume{62}(\bissue{3}),
\bfpage{456}--\blpage{470}
(\byear{2020})
\end{barticle}
\endbibitem

\bibitem[\protect\citeauthoryear{Buskulic
  et~al.}{2024}]{buskulic2024convergence}
\begin{barticle}
\bauthor{\bsnm{Buskulic}, \binits{N.}},
\bauthor{\bsnm{Fadili}, \binits{J.}},
\bauthor{\bsnm{Qu{\'e}au}, \binits{Y.}}:
\batitle{Convergence and recovery guarantees of unsupervised neural networks
  for inverse problems}.
\bjtitle{J. Math. Imaging Vis.}
\bvolume{66}(\bissue{4}),
\bfpage{584}--\blpage{605}
(\byear{2024})
\end{barticle}
\endbibitem

\bibitem[\protect\citeauthoryear{Wang et~al.}{2023}]{wang2023early}
\begin{botherref}
\oauthor{\bsnm{Wang}, \binits{H.}},
\oauthor{\bsnm{Li}, \binits{T.}},
\oauthor{\bsnm{Zhuang}, \binits{Z.}},
\oauthor{\bsnm{Chen}, \binits{T.}},
\oauthor{\bsnm{Liang}, \binits{H.}},
\oauthor{\bsnm{Sun}, \binits{J.}}:
Early stopping for deep image prior.
Transact. Mach. Learn. Res.
\textbf{2023}
(2023)
\end{botherref}
\endbibitem

\bibitem[\protect\citeauthoryear{Qayyum et~al.}{2022}]{qayyum2022untrained}
\begin{barticle}
\bauthor{\bsnm{Qayyum}, \binits{A.}},
\bauthor{\bsnm{Ilahi}, \binits{I.}},
\bauthor{\bsnm{Shamshad}, \binits{F.}},
\bauthor{\bsnm{Boussaid}, \binits{F.}},
\bauthor{\bsnm{Bennamoun}, \binits{M.}},
\bauthor{\bsnm{Qadir}, \binits{J.}}:
\batitle{Untrained neural network priors for inverse imaging problems: A
  survey}.
\bjtitle{IEEE Trans. Pattern Anal. Mach. Intell.}
\bvolume{45}(\bissue{5}),
\bfpage{6511}--\blpage{6536}
(\byear{2022})
\end{barticle}
\endbibitem

\bibitem[\protect\citeauthoryear{Cybenko}{1989}]{cybenko1989approximation}
\begin{barticle}
\bauthor{\bsnm{Cybenko}, \binits{G.}}:
\batitle{Approximation by superpositions of a sigmoidal function}.
\bjtitle{Math. Control Signals Syst.}
\bvolume{2}(\bissue{4}),
\bfpage{303}--\blpage{314}
(\byear{1989})
\end{barticle}
\endbibitem

\bibitem[\protect\citeauthoryear{Hornik et~al.}{1989}]{hornik1989multilayer}
\begin{barticle}
\bauthor{\bsnm{Hornik}, \binits{K.}},
\bauthor{\bsnm{Stinchcombe}, \binits{M.}},
\bauthor{\bsnm{White}, \binits{H.}}:
\batitle{Multilayer feedforward networks are universal approximators}.
\bjtitle{Neural Netw.}
\bvolume{2}(\bissue{5}),
\bfpage{359}--\blpage{366}
(\byear{1989})
\end{barticle}
\endbibitem

\bibitem[\protect\citeauthoryear{Hornik}{1991}]{hornik1991approximation}
\begin{barticle}
\bauthor{\bsnm{Hornik}, \binits{K.}}:
\batitle{Approximation capabilities of multilayer feedforward networks}.
\bjtitle{Neural Netw.}
\bvolume{4}(\bissue{2}),
\bfpage{251}--\blpage{257}
(\byear{1991})
\end{barticle}
\endbibitem

\bibitem[\protect\citeauthoryear{Stinchcombe}{1989}]{stinchcombe1989universal}
\begin{bchapter}
\bauthor{\bsnm{Stinchcombe}}:
\bctitle{Universal approximation using feedforward networks with non-sigmoid
  hidden layer activation functions}.
In: \bbtitle{International 1989 Joint Conference on Neural Networks},
pp. \bfpage{613}--\blpage{617}
(\byear{1989}).
\bcomment{IEEE}
\end{bchapter}
\endbibitem

\bibitem[\protect\citeauthoryear{Pinkus}{1999}]{pinkus1999approximation}
\begin{barticle}
\bauthor{\bsnm{Pinkus}, \binits{A.}}:
\batitle{Approximation theory of the {MLP} model in neural networks}.
\bjtitle{Acta Numer.}
\bvolume{8},
\bfpage{143}--\blpage{195}
(\byear{1999})
\end{barticle}
\endbibitem

\bibitem[\protect\citeauthoryear{Barron}{2002}]{barron2002universal}
\begin{barticle}
\bauthor{\bsnm{Barron}, \binits{A.R.}}:
\batitle{Universal approximation bounds for superpositions of a sigmoidal
  function}.
\bjtitle{IEEE Trans. Inf. Theory}
\bvolume{39}(\bissue{3}),
\bfpage{930}--\blpage{945}
(\byear{2002})
\end{barticle}
\endbibitem

\bibitem[\protect\citeauthoryear{Bach}{2017}]{bach2017breaking}
\begin{barticle}
\bauthor{\bsnm{Bach}, \binits{F.}}:
\batitle{Breaking the curse of dimensionality with convex neural networks}.
\bjtitle{J. Mach. Learn. Res.}
\bvolume{18}(\bissue{19}),
\bfpage{1}--\blpage{53}
(\byear{2017})
\end{barticle}
\endbibitem

\bibitem[\protect\citeauthoryear{Li et~al.}{2024}]{li2024two}
\begin{barticle}
\bauthor{\bsnm{Li}, \binits{Y.}},
\bauthor{\bsnm{Lu}, \binits{S.}},
\bauthor{\bsnm{Math{\'e}}, \binits{P.}},
\bauthor{\bsnm{Pereverzev}, \binits{S.V.}}:
\batitle{Two-layer networks with the {$\mathrm{ReLU}^k$} activation function:
  Barron spaces and derivative approximation}.
\bjtitle{Numer. Math.}
\bvolume{156}(\bissue{1}),
\bfpage{319}--\blpage{344}
(\byear{2024})
\end{barticle}
\endbibitem

\bibitem[\protect\citeauthoryear{Eldan and Shamir}{2016}]{eldan2016power}
\begin{bchapter}
\bauthor{\bsnm{Eldan}, \binits{R.}},
\bauthor{\bsnm{Shamir}, \binits{O.}}:
\bctitle{The power of depth for feedforward neural networks}.
In: \bbtitle{Conference on Learning Theory},
pp. \bfpage{907}--\blpage{940}
(\byear{2016}).
\bcomment{PMLR}
\end{bchapter}
\endbibitem

\bibitem[\protect\citeauthoryear{Telgarsky}{2016}]{telgarsky2016benefits}
\begin{bchapter}
\bauthor{\bsnm{Telgarsky}, \binits{M.}}:
\bctitle{Benefits of depth in neural networks}.
In: \bbtitle{Conference on Learning Theory},
pp. \bfpage{1517}--\blpage{1539}
(\byear{2016}).
\bcomment{PMLR}
\end{bchapter}
\endbibitem

\bibitem[\protect\citeauthoryear{Mhaskar and Poggio}{2016}]{mhaskar2016deep}
\begin{barticle}
\bauthor{\bsnm{Mhaskar}, \binits{H.N.}},
\bauthor{\bsnm{Poggio}, \binits{T.}}:
\batitle{Deep vs. shallow networks: An approximation theory perspective}.
\bjtitle{Anal. Appl.}
\bvolume{14}(\bissue{06}),
\bfpage{829}--\blpage{848}
(\byear{2016})
\end{barticle}
\endbibitem

\bibitem[\protect\citeauthoryear{Poggio et~al.}{2017}]{poggio2017and}
\begin{barticle}
\bauthor{\bsnm{Poggio}, \binits{T.}},
\bauthor{\bsnm{Mhaskar}, \binits{H.}},
\bauthor{\bsnm{Rosasco}, \binits{L.}},
\bauthor{\bsnm{Miranda}, \binits{B.}},
\bauthor{\bsnm{Liao}, \binits{Q.}}:
\batitle{Why and when can deep-but not shallow-networks avoid the curse of
  dimensionality: a review}.
\bjtitle{Int. J. Autom. Comput.}
\bvolume{14}(\bissue{5}),
\bfpage{503}--\blpage{519}
(\byear{2017})
\end{barticle}
\endbibitem

\bibitem[\protect\citeauthoryear{Lu et~al.}{2017}]{lu2017expressive}
\begin{botherref}
\oauthor{\bsnm{Lu}, \binits{Z.}},
\oauthor{\bsnm{Pu}, \binits{H.}},
\oauthor{\bsnm{Wang}, \binits{F.}},
\oauthor{\bsnm{Hu}, \binits{Z.}},
\oauthor{\bsnm{Wang}, \binits{L.}}:
The expressive power of neural networks: A view from the width.
Adv. Neural Inf. Process. Syst.
\textbf{30}
(2017)
\end{botherref}
\endbibitem

\bibitem[\protect\citeauthoryear{Yarotsky}{2017}]{yarotsky2017error}
\begin{barticle}
\bauthor{\bsnm{Yarotsky}, \binits{D.}}:
\batitle{Error bounds for approximations with deep {ReLU} networks}.
\bjtitle{Neural Netw.}
\bvolume{94},
\bfpage{103}--\blpage{114}
(\byear{2017})
\end{barticle}
\endbibitem

\bibitem[\protect\citeauthoryear{Yarotsky}{2018}]{yarotsky2018optimal}
\begin{bchapter}
\bauthor{\bsnm{Yarotsky}, \binits{D.}}:
\bctitle{Optimal approximation of continuous functions by very deep {ReLU}
  networks}.
In: \bbtitle{Conference on Learning Theory},
pp. \bfpage{639}--\blpage{649}
(\byear{2018}).
\bcomment{PMLR}
\end{bchapter}
\endbibitem

\bibitem[\protect\citeauthoryear{Shen et~al.}{2019}]{shen2019nonlinear}
\begin{barticle}
\bauthor{\bsnm{Shen}, \binits{Z.}},
\bauthor{\bsnm{Yang}, \binits{H.}},
\bauthor{\bsnm{Zhang}, \binits{S.}}:
\batitle{Nonlinear approximation via compositions}.
\bjtitle{Neural Netw.}
\bvolume{119},
\bfpage{74}--\blpage{84}
(\byear{2019})
\end{barticle}
\endbibitem

\bibitem[\protect\citeauthoryear{Shen et~al.}{2020}]{shen2020deep}
\begin{barticle}
\bauthor{\bsnm{Shen}, \binits{Z.}},
\bauthor{\bsnm{Yang}, \binits{H.}},
\bauthor{\bsnm{Zhang}, \binits{S.}}:
\batitle{Deep network approximation characterized by number of neurons}.
\bjtitle{Commun. Comput. Phys.}
\bvolume{28}(\bissue{5}),
\bfpage{1768}--\blpage{1811}
(\byear{2020})
\end{barticle}
\endbibitem

\bibitem[\protect\citeauthoryear{Shen et~al.}{2022}]{shen2022optimal}
\begin{barticle}
\bauthor{\bsnm{Shen}, \binits{Z.}},
\bauthor{\bsnm{Yang}, \binits{H.}},
\bauthor{\bsnm{Zhang}, \binits{S.}}:
\batitle{Optimal approximation rate of {ReLU} networks in terms of width and
  depth}.
\bjtitle{J. Math. Pures Appl.}
\bvolume{157},
\bfpage{101}--\blpage{135}
(\byear{2022})
\end{barticle}
\endbibitem

\bibitem[\protect\citeauthoryear{Yarotsky and
  Zhevnerchuk}{2020}]{yarotsky2020phase}
\begin{barticle}
\bauthor{\bsnm{Yarotsky}, \binits{D.}},
\bauthor{\bsnm{Zhevnerchuk}, \binits{A.}}:
\batitle{The phase diagram of approximation rates for deep neural networks}.
\bjtitle{Adv. Neural Inf. Process. Syst.}
\bvolume{33},
\bfpage{13005}--\blpage{13015}
(\byear{2020})
\end{barticle}
\endbibitem

\bibitem[\protect\citeauthoryear{G{\"u}hring et~al.}{2020}]{guhring2020error}
\begin{barticle}
\bauthor{\bsnm{G{\"u}hring}, \binits{I.}},
\bauthor{\bsnm{Kutyniok}, \binits{G.}},
\bauthor{\bsnm{Petersen}, \binits{P.}}:
\batitle{Error bounds for approximations with deep {ReLU} neural networks in
  {$W^{s,p}$} norms}.
\bjtitle{Anal. Appl.}
\bvolume{18}(\bissue{05}),
\bfpage{803}--\blpage{859}
(\byear{2020})
\end{barticle}
\endbibitem

\bibitem[\protect\citeauthoryear{Lu et~al.}{2021}]{lu2021deep}
\begin{barticle}
\bauthor{\bsnm{Lu}, \binits{J.}},
\bauthor{\bsnm{Shen}, \binits{Z.}},
\bauthor{\bsnm{Yang}, \binits{H.}},
\bauthor{\bsnm{Zhang}, \binits{S.}}:
\batitle{Deep network approximation for smooth functions}.
\bjtitle{SIAM J. Math. Anal.}
\bvolume{53}(\bissue{5}),
\bfpage{5465}--\blpage{5506}
(\byear{2021})
\end{barticle}
\endbibitem

\bibitem[\protect\citeauthoryear{Hon and Yang}{2022}]{hon2022simultaneous}
\begin{barticle}
\bauthor{\bsnm{Hon}, \binits{S.}},
\bauthor{\bsnm{Yang}, \binits{H.}}:
\batitle{Simultaneous neural network approximation for smooth functions}.
\bjtitle{Neural Netw.}
\bvolume{154},
\bfpage{152}--\blpage{164}
(\byear{2022})
\end{barticle}
\endbibitem

\bibitem[\protect\citeauthoryear{Petersen and
  Voigtlaender}{2018}]{petersen2018optimal}
\begin{barticle}
\bauthor{\bsnm{Petersen}, \binits{P.}},
\bauthor{\bsnm{Voigtlaender}, \binits{F.}}:
\batitle{Optimal approximation of piecewise smooth functions using deep {ReLU}
  neural networks}.
\bjtitle{Neural Netw.}
\bvolume{108},
\bfpage{296}--\blpage{330}
(\byear{2018})
\end{barticle}
\endbibitem

\bibitem[\protect\citeauthoryear{Jiao et~al.}{2023}]{jiao2023approximation}
\begin{barticle}
\bauthor{\bsnm{Jiao}, \binits{Y.}},
\bauthor{\bsnm{Wang}, \binits{Y.}},
\bauthor{\bsnm{Yang}, \binits{Y.}}:
\batitle{Approximation bounds for norm constrained neural networks with
  applications to regression and {GAN}s}.
\bjtitle{Appl. Comput. Harmon. Anal.}
\bvolume{65},
\bfpage{249}--\blpage{278}
(\byear{2023})
\end{barticle}
\endbibitem

\bibitem[\protect\citeauthoryear{Schmidt-Hieber}{2021}]{schmidt2021kolmogorov}
\begin{barticle}
\bauthor{\bsnm{Schmidt-Hieber}, \binits{J.}}:
\batitle{The {Kolmogorov--Arnold} representation theorem revisited}.
\bjtitle{Neural Netw.}
\bvolume{137},
\bfpage{119}--\blpage{126}
(\byear{2021})
\end{barticle}
\endbibitem

\bibitem[\protect\citeauthoryear{Wang et~al.}{2025}]{wang2025shallow}
\begin{botherref}
\oauthor{\bsnm{Wang}, \binits{L.}},
\oauthor{\bsnm{Zhu}, \binits{Q.}},
\oauthor{\bsnm{Jin}, \binits{B.}},
\oauthor{\bsnm{Zhang}, \binits{Y.}}:
Shallow neural network yields regularization for ill-posed inverse problems.
arXiv preprint arXiv:2511.16171
(2025)
\end{botherref}
\endbibitem

\bibitem[\protect\citeauthoryear{Yang}{2025}]{yang2025optimal}
\begin{botherref}
\oauthor{\bsnm{Yang}, \binits{Y.}}:
On the optimal approximation of {S}obolev and {B}esov functions using deep
  {ReLU} neural networks.
Appl. Comput. Harmon. Anal.,
101797
(2025)
\end{botherref}
\endbibitem

\end{thebibliography}

\end{document}